\documentclass[12pt]{amsart}
\usepackage{amssymb, amsmath, amsthm,  latexsym, array,euscript}
\usepackage{fullpage}
\usepackage{verbatim}
\usepackage[scr=boondoxo,scrscaled=1.05]{mathalfa}
\usepackage{accents}

\newcommand{\vardbtilde}[1]{\tilde{\raisebox{-2.5pt}[0.65\height]{$\tilde{#1}$}}}
\newcommand{\varbtilde}[1]{\tilde{\raisebox{-0.5pt}[0.85\height]{$\tilde{#1}$}}}

\usepackage{color}
\definecolor{darkblue}{rgb}{0,0,.5}
\definecolor{darkgreen}{rgb}{.2,0.5,.2}
\usepackage[colorlinks=true,linkcolor=darkblue,urlcolor=darkgreen,citecolor=magenta]{hyperref}

\numberwithin{equation}{section}

\input{cyracc.def}

\font\tencyr=wncyr10 %scaled \magstephalf
\font\tencyi=wncyi10 %scaled \magstephalf
\font\tencysc=wncysc10 %scaled \magstephalf

\def\rus{\tencyr\cyracc}
\def\rusi{\tencyi\cyracc}
\def\rusc{\tencysc\cyracc}

\newtheorem{lm}{Lemma}[section]
\newtheorem{thm}[lm]{Theorem}%[section]
\newtheorem{conj}[lm]{Conjecture}%[section]
\newtheorem{cl}[lm]{Corollary}
\newtheorem{prop}[lm]{Proposition}

\theoremstyle{remark}
\newtheorem{ex}[lm]{Example}%[section]
\newtheorem{rmk}[lm]{Remark}%[section]
\theoremstyle{definition}
\newtheorem*{Rem}{Remark}%[section]

\newcommand {\beq}{\begin{equation}}
\newcommand {\eeq}{\end{equation}}
\newcommand {\vp}{\varphi}

\newcommand{\gt}{\mathfrak}

\newcommand{\GL}{{\rm GL}}

\newcommand{\ind}{{\rm ind\,}}

\newcommand{\codim}{\mathrm{codim\,}}
\newcommand{\rk}{\mathrm{rk\,}}

\newcommand{\Lie}{\mathrm{Lie\,}}

\newcommand{\Pol}{\mathrm{Pol}}
\newcommand{\gr}{\mathrm{gr\,}}

\newcommand{\wg}{{\widehat{\gt g}}}
\newcommand{\wq}{{\widehat{\gt q}}}

\newcommand {\cG}{{\mathcal G}}

\newcommand {\oG}{\overline{\mathcal G}}

\newcommand{\esi}{\varepsilon}
\newcommand {\eus}{\EuScript}

\newcommand {\ca}{{\mathcal A}}

\newcommand {\cF}{{\mathcal F}}

\newcommand {\cH}{{\mathcal H}}

\newcommand {\gS}{{\mathcal S}} 
\newcommand{\zu}{\mathcal{ZU}}

\newcommand {\cP}{{\mathcal P}}
\newcommand {\eP}{{\eus P}}
\newcommand {\eR}{{\eus R}}
\newcommand {\eL}{{\eus L}}
\newcommand {\cX}{{\mathcal X}}

\newcommand {\ec}{{\eus C}}

\newcommand {\mW}{{\mathbb W}}
\newcommand {\mA}{{\mathbb A}}
\newcommand {\mX}{{\mathbb X}}
\newcommand {\mP}{{\mathbb P}}
\newcommand {\cz}{{\mathcal Z}}
\newcommand {\gZ}{{\eus Z}}
\newcommand {\gzu}{{\gZ\!\mathscr{v}\!}}%%\eus Z}\upsilon}
\newcommand{\bb}{\boldsymbol{b}}
\newcommand{\dd}{\boldsymbol{j}}
\newcommand{\bH}{\boldsymbol{H}}
\newcommand{\bF}{\boldsymbol{F}}
\newcommand{\bh}{\boldsymbol{h}}
\newcommand{\br}{\boldsymbol{r}}
\newcommand {\calv}{{\mathscr{V}\!}}

\newcommand {\Z}{{\mathbb Z}}

\newcommand{\ap}{\alpha}
\newcommand{\ups}{\upsilon}

\renewcommand{\le}{\leqslant}
\renewcommand{\ge}{\geqslant}

\newfam\eusfam%
\font\euszw=eusm10 scaled 1200%
\font\eusac=eusm7 scaled 1200%
\font\eusacc=eusm7 scaled 1000%
\textfont\eusfam=\euszw\scriptfont\eusfam=\eusac%
\scriptscriptfont\eusfam=\eusacc%
\newcommand{\U}{\mathcal{U}}

\newcommand {\trdeg}{{\mathrm{tr.deg\,}}}
\newcommand {\mK}{{\Bbbk}}

\newcommand {\qqn}{{\q\langle n\rangle}}

\newcommand {\qqk}{{\q\langle k\rangle}}
\newcommand {\qgk}{{\gt g\langle k\rangle}}
\newcommand {\qqmi}{{\q\langle m_i\rangle}}
\newcommand {\qqi}{{\q\langle 1\rangle}}

\newcommand{\q}{\mathfrak{q}}

\begin{document}
\hfill {\scriptsize   August 5,  2021} %%  \today}
\vskip1ex

\title[A bi-Hamiltonian nature of the Gaudin algebras]{A bi-Hamiltonian nature of the Gaudin algebras}
\author[O.\,Yakimova]{Oksana Yakimova}
\address[O.\,Yakimova]{Institut f\"ur Mathematik, Friedrich-Schiller-Universit\"at Jena,  07737 Jena,
Deutschland}
\thanks{This work  is funded by  the Deutsche Forschungsgemeinschaft (DFG, German Research Foundation) --- project  number  454900253.}
\begin{abstract}
Let $\q$ be a Lie algebra over a field $\mK$ and $p,\tilde p\in\mK[t]$ two different normalised polynomials  
of degree $n\ge 2$. As  vector spaces both quotient Lie algebras $\q[t]/(p)$ and $\q[t]/(\tilde p)$ can be identified 
with $\mW=\q{\cdot}1\oplus\q\bar t\oplus\ldots\oplus\q\bar t^{n-1}$. %% stand for any of these spaces.
If $\deg(p-\tilde p)\le 1$, then the Lie brackets
$[\,\,,\,]_p$, $[\,\,,\,]_{\tilde p}$ induced on $\mW$ by $p$ and $\tilde p$, respectively, are compatible.  
%%By a general method, known as 
Making use of 
the Lenard--Magri scheme, we construct a 
subalgebra $\gZ=\gZ(p,\tilde p)\subset\gS(\mW)^{\q{\cdot}1}$ such that $\{\gZ,\gZ\}_p=\{\gZ,\gZ\}_{\tilde p}=0$. 
If $\trdeg\gS(\q)^{\q}=\ind\q$ and $\q$ has the {\sl codim}--$2$ property, then $\trdeg\gZ$ takes the maximal possible value, which is $\frac{n-1}{2}\dim\q+\frac{n+1}{2}\ind\q$. If $\q=\gt g$ is semisimple, then $\gZ$ contains  
the Hamiltonians of  a suitably chosen Gaudin model. 
Furthermore,  if $p$ and $\tilde p$ do not have common roots, then there is a Gaudin subalgebra $\eus C\subset\U(\gt g^{\oplus n})$ such that 
$\gZ=\gr\!(\eus C)$, up to  a certain identification. 
%%% Therefore, 
In a non-reductive case, we obtain a completely integrable
generalisation of Gaudin models. 

For a wide class of Lie algebras, which extends the reductive setting, 
$\gZ(p,p+t)$ coincides with the image of  the Poisson-commutative algebra $\gZ(\wq,t)=\gS(t\q[t])^{\q[t^{-1}]}$ %in 
under the quotient map $\psi_p\!:\gS(\q[t])\to\gS(\mW)$,  providing $p(0)\ne 0$. 
\end{abstract}
\maketitle

\section*{Introduction}

The ground field $\mK$ is algebraically closed and $\mathsf{char}\,\mK=0$.
Let $\q$ be a  Lie algebra defined over $\mK$.  Most of the time we assume that $\dim\q<\infty$. 
The symmetric algebra $\gS(\q)$ of $\q$ carries the standard Poisson structure. 
A subalgebra $\ca\subset\gS(\q)$ is said to be {\it Poisson-commutative} if $\{\ca,\ca\}=0$.  
Poisson-commutative subalgebras  attract a great deal of attention, because of
their relationship to integrable systems and, more recently, to geometric representation theory. 
There are several prominent constructions of Poisson-commutative subalgebras. In this paper, we use two of them. 

The first approach is to replace $\q$ with something larger, for instance, with the current algebra $\q[t]$. 
The symmetric algebra $\gS(\gt q[t])$ contains a large Poisson-commutative 
subalgebra $\gZ(\wq,t)$, see Section~\ref{sec-inf}.  Then, by taking an appropriate quotient, one obtains a 
Poisson-commutative subalgebra of $\gS(\q)$ or of $\gS(\q^{\oplus n})$. 

The second approach employs {\it compatible} Poisson brackets on $\gS(\q)$ and is widely known as
the {\it Lenard--Magri scheme}. An advantage of this method is a well-developed geometric machinery, see 
e.g. \cite{GZ,kruks}. 
A really nice situation happens if both approaches produce one and the same result.

We are interested in $\q[t]/(p)$, where $p\in\mK[t]$ is a normalised polynomial of degree $n\ge 2$.  
As a vector space  $\q[t]/(p)$ can be identified with $\mW=\mW(\gt q,n)=\gt q{\cdot}1\oplus\gt q\bar t\oplus\ldots\oplus\gt q \bar t^{n-1}$. The Lie structure on it depends on $p$. We let $[\,\,,\,]_p$ stand for the Lie bracket of 
$\q[t]/(p)$.  
If $p$ has pairwise distinct roots, then $(\mW,[\,\,,\,]_p)\cong\q^{\oplus n}$. In an opposite case, where $p=t^n$,\,
$\gt q[t]/(t^n)$ is  a (generalised) {\it Takiff algebra} modelled on $\gt q$. 
These Lie algebras have been studied since 1971, first in case $\gt q=\gt g$ is reductive 
 \cite{takiff,rt,Geof,Wil,GM} and then for a more general class of $\q$ \cite{ap,k-T}. 
For any $p$, $(\mW,[\,\,,\,]_p)$ is a direct sum of several Takiff algebras, which may include copies of $\q$. 

Our first observation is that the brackets $[\,\,,\,]_p$ and $[\,\,,\,]_{\tilde p}$ on $\mW$ are compatible if 
$p,\tilde p\in\mK[t]$ are of degree $n$ and $\deg(p-\tilde p)\le 1$. Here we assume that $p\ne\tilde p$. 
Let $\gZ=\gZ(p,\tilde p)\subset\gS(\mW)$ be the Poisson-commutative subalgebra constructed according 
to the  Lenard--Magri scheme. Then $\gZ\subset\gS(\mW)^{\q{\cdot}1}$, see Section~\ref{sub-Z}.  
Thereby $\trdeg\gZ$ is bounded by a certain number $\bb(\q,n)$ related to $\q$ and $n$ \eqref{bound}.  

For an arbitrary $\q$, one cannot say much about $\gZ$. We need to select  a nice class of Lie algebras.
In this selection, the {\it symmetric invariants} $\gS(\q)^{\q}=\cz(\q)$ of $\q$  and properties of $\gt q^*$ play a major r\^ole. 
The index of $\q$ is defined by $\ind\q=\min_{\gamma\in\q^*}\dim\q_\gamma$, where $\q_\gamma$ is the stabiliser of $\gamma$. Then \[\q^*_{\sf sing}=\{ \gamma\in\q^* \mid \dim\q_\gamma>\ind\q\}.
\] 
Suppose that $\trdeg\gS(\q)^{\q}=\ind\q$ and $\dim\q^*_{\sf sing}\le \dim\q-2$, then 
$\trdeg\gZ=\bb(\q,n)$, see Theorem~\ref{trdeg}. 
The symmetric invariants $\cz_{\vardbtilde{p}}$ of $(\mW,[\,\,,\,]_{\vardbtilde{p}})$ with 
${\varbtilde{p}}=ap+(1-a)\tilde p$ and $a\in\mK$ are building blocks for $\gZ$. In order to describe  
generators of $\gZ$, one needs to understand  algebras $\cz_{\vardbtilde{p}}$. 

If $\qgk\cong\gt g[t]/(t^k)$ is a Takiff algebra modeld on a reductive $\gt g$, then 
$\cz(\qgk)$ is a polynomial ring with $k{\cdot}\rk\gt g=\ind\qgk$ generators \cite{takiff,rt}. 
A similar statement holds for certain non-reductive $\gt q$
%%in 
\cite{ap,k-T}. %% for some . 
Our main results for these (very suitable) Lie algebras 
are listed below. 

Suppose that $\gS(\q)^\q=\mK[F_1,\ldots,F_m]$, where  $F_i$ are  homogeneous algebraically independent polynomials with $\deg F_i=d_i$,  \ $m=\ind\q$, %%  
and $\sum_{i=1}^m d_i=\bb(\q)=\frac{1}{2}(\dim\q+\ind\q)$. 
\begin{itemize}
\item[$\diamond$] If $\dim\q^*_{\sf sing}\le \dim\q-2$, then  $\gZ=\gZ(p,\tilde p)$ is a polynomial ring with 
$\bb(\q,n)$ generators. Furthermore, $\gZ$ has a set of algebraically independent generators 
$\{F_{i,u}\}$ such that each $F_{i,u}$ is a {\it polarisation} of $F_i$, cf. Theorem~\ref{free}. 
\item[$\diamond$] If $\dim\q^*_{\sf sing}\le \dim\q-3$, then $\gZ$  is a maximal (w.r.t. inclusion) Poisson-commutative 
subalgebra of $(\gS(\mW),\{\,\,,\,\}_p)^{\gt q{\cdot}1}$, see Theorem~\ref{max}.
\end{itemize}

The assumptions on $\q$ in \cite{k-T} are slightly weaker than in \cite{ap}.  The requirement  $\dim\q^*_{\sf sing}\le \dim\q-2$, imposed in \cite{ap}, is replaced by 
a similar condition on the differentials $\textsl{d}F_i$ with $1\le i\le m$,
see Theorem~\ref{thm:k-T}  for details. %%% 
Also the sum $\sum_{i=1}^m d_i$ is allowed to be smaller than or equal to $\bb(\q)$. 
In Sections~3 and 4 of \cite{k-T}, one can find examples of Lie algebras that  
satisfy the assumptions of Theorem~\ref{thm:k-T}, but do not have the {\sl codim}--$2$ property. 

We prove that the assumptions of \cite[Thm~0.1]{k-T} imply the identity  $\psi_p(\gZ(\wq,t))=\gZ(p,p+t)$ %% if 
for the quotient map $\psi_p\!:\gS(\q[t])\to \gS(\mW)$  if %% and 
$p(0)\ne 0$, see Theorem~\ref{sovp-t}.  Changing the variable $t\mapsto \esi t+1$, we obtain a new 
Poisson-commutative subalgebra $\gZ(\wq,\esi t +1)$ of $\gS(\q[t])$ (here $\esi\in\mK)$. 
Set $\gzu:=\lim_{\esi\to 0}\gZ(\wq,\esi t +1)$. Then $\{\gzu, \gzu\}=0$. 
If $\q$ satisfies  the assumptions of Theorem~\ref{thm:k-T}, then $\psi_p(\gzu)=\gZ(p,p+1)$ for any $p$, see  
%%\colorbox{red}{???}  Proposition~\ref{123} and 
Theorem~\ref{ft-gzu}. 

Suppose  now that $\gt q=\gt g$ is a finite-dimensional simple (non-Abelian)  Lie algebra. 
Note that the reductive Lie algebras satisfy all 
assumptions of \cite{ap,k-T} and we have also $\dim\gt g^*_{\sf sing}=\dim\gt g-3$, cf.~\cite[Remark~1.1]{kruks}.
Suppose further that 
$p$ has pairwise distinct nonzero roots $a_1,\ldots,a_n$.  Then $\gt g[t]/(p)\cong \gt h=\gt g^{\oplus n}$ and 
$\psi_p(\gZ(\wg,t))\subset\gS(\gt h)$  identifies with the associated graded algebra  $\gr\!(\ec(\vec{z}))$ of a well-studied {\it Gaudin algebra} $\ec(\vec z)\subset\U(\gt h)$.    Gaudin algebras were introduced in 
\cite{FFRe}. 
The %%r \colorbox{yellow}{
purpose of that paper  was to present a new method of diagonalisation of Gaudin Hamiltonians. 
Each Gaudin algebra depends on a vector $\vec z=(z_1,\ldots,z_n)\in(\mK^\times)^n$. 
It is commutative and contains the quadratic Gaudin Hamiltonians associated with $\gt h$ and  $\vec z$,
see  Section~\ref{sb-GH}. 
With its help one  shows 
%% Gaudin algebras show 
quantum complete integrability of the corresponding Gaudin model.  
%%% associated with . 
Due to some change of notation,  %%we have
$\psi_p(\gZ(\wg,t))=\gZ(p,p+t)$ identifies  with $\gr\!(\ec(\vec{z}))$ for  $\vec z=(a_1^{-1},\ldots,a_n^{-1})$, see 
Proposition~\ref{G-psi}.   Using certain limit constructions, we can conclude that 
$\gZ(p,p+1)=\gr\!(\eus C(\vec a))$, up to the same identification, for $\vec a=(a_1,\ldots,a_n)$, see
Section~\ref{s-gzu}. 

%%% ........ \colorbox{yellow}{{\it some limits}} .......... 

%%A Gaudin model related to $\gt h$ consists of $n$ linearly  dependent (or, equivalently, $n-1$ linearly
%% independent) quadratic Hamiltonians depending on $\vec z$, 
Suppose that $\q$ is a quadratic Lie algebra. Then it is possible to define a Gaudin model related to 
$\q^{\oplus n}$ with practically the same formulas as in the reductive case. 
In Section~\ref{sec-nrG}, we express these Gaudin models in terms of
$(\mW,[\,\,,\,]_p)$ and show that they are completely integrable if $\trdeg\gZ(p,p+t)=\bb(\q,n)$.   
Some types of Lie algebras, where this applies, are presented in Example~\ref{qv-dim2}. 
The setting makes sense also if $p$ has iterated roots. In this way, we obtain new limits of Gaudin models,
where the underlying  Lie algebra is no longer $\q^{\oplus n}$, but a direct sum of various Takiff algebras 
$\qqk$. 
%%%%  ..........  

%%%%%%%%%%%%%%%
\section{Preliminaries on Lie algebras} %%% 
\label{sect:prelim}

%\noindent
Let $\gt q$ be a %%finite-dimensional 
Lie algebra over $\mK$.
%% Let $Q$
The  symmetric algebra of 
$\q$ over $\mK$ is $\mathbb N_{0}$-graded, i.e., $\gS(\q)=\bigoplus_{i\ge 0}\gS^i(\q)$. 
The standard  Lie--Poisson bracket on $\gS(\q)$ is defined on $\gS^1(\q)=\q$ by $\{x,y\}:=[x,y]$. It is then extended to higher degrees via the Leibniz rule. Hence $\gS(\q)$ has the usual associative-commutative structure and additional Poisson structure.  
Whenever we refer to {\sl subalgebras\/} of $\gS(\q)$, we always mean the associative-commutative structure.

The {\it centre\/} of the Poisson algebra $(\gS(\q), \{\,\,,\,\})$ is 
\beq \label{PZ1}
    \cz(\q):=\{H\in \gS(\q)\mid \{H,F\}=0 \ \ \forall F\in\gS(\q)\} .
\eeq
Using the Leibniz rule, we obtain that $\cz(\q)$ is a graded Poisson-commutative subalgebra of $\gS(\q)$, which coincides with the algebra of symmetric invariants of $\q$, i.e.,
\[
    \cz(\q)=\{H\in \gS(\q)\mid \{H,x\}=0 \ \ \forall x\in\q\}=:\gS(\q)^\q. %%=\mK[\q^*]^\q .
\] 
In order to avoid confusion, we let $\zu(\gt q)$ stand for the centre of the enveloping algebra $\U(\gt q)$.
If $\gt l\subset\gt q$ is a Lie subalgebra, then $\gS(\gt q)^{\gt l}=\{ H\in\gS(\gt q) \mid \{H,x\}=0 \ \ \forall x\in\gt l\}$.

If $F\in \gS(\q)$ and $V\subset \gS(\q)$ is a subspace, then the {\it Poisson centraliser of $F$ in $V$} is the sub\-space
$\{ H\in V \mid \{F,H\}=0\}$.   

\subsection{The coadjoint representation}
\label{subs:coadj}
%%%
%%%%
From now until the end of this section, assume that 
$\gt q$ is finite-dimensional.  Set  
\[\bb(\q)=(\dim\q+\ind\q)/2.\] Clearly, $\bb(\q)$ is an integer.
 If $\q$ is reductive, then  
$\ind\q=\rk\q$ and $\bb(\q)$ equals the dimension of a Borel subalgebra.  
%% It is well 
If $\ca\subset\gS(\gt q)$ is Poisson-commutative, then $\trdeg\ca\le \bb(\q)$, see e.g.~\cite[0.2]{vi90}.

We  identify $\gS(\gt q)$ with the 
algebra of polynomial functions on the dual 
space $\q^*$, and we also write $\mK[\q^*]=\bigoplus_{i\ge 0}\mK[\q^*]_i$ for it. 
The set of {\it regular\/} elements of $\q^*$ is 
\beq       \label{eq:regul-set}
    \q^*_{\sf reg}=\{\eta\in\q^*\mid \dim \q_\eta=\ind\q\} =\q^*\setminus \q^*_{\sf sing}.
\eeq
It is a dense open subset of $\q^*$. 
We say that $\q$ has the {\sl codim}--$n$ property if $\codim \q^*_{\sf sing}\ge n$. 
The {\sl codim}--$2$ property  is going to be most important for us. 
%By~\cite{ko63}, 

Suppose that $\gt q=\Lie Q$, where $Q$ is a  connected algebraic group. Then 
$Q$ acts on $\gt q^*$ via the coadjoint representation and $\gS(\q)^\q=\gS(\q)^{Q}=\mK[\q^*]^Q$.
We have also \[\ind\gt q=\dim\gt q-\max_{\xi\in\gt q^*}\dim(Q\xi).\] 
Write $\mK(\q^*)^Q$ for the field of $Q$-invariant rational functions on $\q^*$.
By the Rosenlicht theorem (see~\cite[IV.2]{spr}), one  has 
$\ind\q=\trdeg\mK(\q^*)^Q$. 
Since the quotient field of
$\mK[\q^*]^Q$ is contained in $\mK(\q^*)^Q$, we deduce from the Rosenlicht theorem that
\beq    \label{eq:neravenstvo-ind}
    \trdeg (\gS(\q)^\q)\le \ind\q .
\eeq

For $\gamma\in\q^*$, let $\hat\gamma$ be the skew-symmetric bilinear form on $\q$ defined by 
$\hat\gamma(\xi,\eta)=\gamma([\xi,\eta])$ for $\xi,\eta\in\q$. It follows that
$\ker\hat\gamma=\q_\gamma$. The $2$-form $\hat\gamma$ is related to 
the {\it Poisson tensor (bivector)} $\pi$ of the Lie--Poisson bracket $\{\,\,,\,\}$ as follows.

Let $\textsl{d}H$ denote the differential of $H\in \gS(\q)=\mK[\q^*]$. Then 
$\pi$ is defined by the formula
$\pi(\textsl{d}H\wedge \textsl{d}F)=\{H,F\}$ for $H,F\in\gS(\q)$. Then 
$\pi(\gamma)(\textsl{d}_\gamma H\wedge \textsl{d}_\gamma F)=\{H,F\}(\gamma)$ and therefore
$\hat\gamma=\pi(\gamma)$.
In this terms, $\ind\q=\dim\q-\rk\pi$, where $\rk\pi=\max_{\gamma\in\q^*}\rk\pi(\gamma)$. 

For a subalgebra $A\subset\gS(\q)$ and $\gamma\in\gt q^*$, set 
$\textsl{d}_\gamma A=\left<\textsl{d}_\gamma F \mid F\in A \right>$. By the definition of the Poisson centre 
$\cz(\q)$, we have 
\begin{equation} \label{incl}
\textsl{d}_\gamma \cz(\q)\subset \ker\pi(\gamma)
\end{equation}
 for each $\gamma\in\q^*$. 

\subsection{Contractions and invariants}
\label{subs:contr-&-inv} 
We refer to \cite[Ch.\,7,\,\S\,2]{t41} for basic facts on contractions of Lie algebras.
%% In this . 
Let $\mK^\times=\mK\setminus\{0\}$ be 
the multiplicative group of $\mK$ and 
%let 
$\vp: \mK^\times\to \GL(\q)$, $s\mapsto \vp_s$, a polynomial representation. That is, 
%group 
the matrix entries of $\vp_{s}:\q\to \q$ are polynomials in $s$ w.r.t. some (any) basis of $\q$.
Define a new Lie algebra structure on the vector space $\q$ and the associated Lie--Poisson bracket by 
\beq       \label{eq:fi_s}
      [x, y]_{(s)}=\{x,y\}_{(s)}:=\vp_s^{-1}[\vp_s( x), \vp_s( y)], \ x,y \in \q, \ s\in\mK^\times.
\eeq
The corresponding Lie algebra is denoted by $\q_{(s)}$. Then $\q_{(1)}=\q$ and all these algebras 
are isomorphic. The induced $\mK^\times$-action on the variety of structure constants is not necessarily 
polynomial, i.e., \ $\lim_{s\to 0}[x, y]_{(s)}$ may not exist for all $x,y\in\q$. Whenever such a limit exists, 
we obtain a new linear Poisson bracket $\lim_{s\to 0}\{\,\,,\,\}_{(s)}$, and thereby a new Lie algebra $\q_{(0)}$, 
which is said to be a {\it contraction\/} of $\q$. 

%There 
The map $\vp_s$, $s\in\mK^\times$, is naturally extended to an invertible transformation of 
$\gS^j(\q)$, which we also denote by $\vp_s$. The resulting graded map 
$\vp_s:\gS(\q)\to\gS(\q)$ is nothing but the comorphism associated with $s\in\mK^\times$ and
the dual representation
$\vp^*:\mK^\times\to \GL(\q^*)$.
Since $\gS^j(\q)$ has a basis that consists of $\vp(\mK^\times)$-eigenvectors, any $F\in\gS^j(\q)$  
can be written as $F=\sum_{i\ge 0}F_i$, %we have 
where the sum is finite and $\vp_s(F_i)=s^iF_i\in\gS^j(\q)$. Let $F^\bullet$ denote the nonzero component $F_i$ with maximal $i$.

\begin{prop}[{\cite[Lemma~3.3]{contr}}]     \label{prop:bullet}
If $F\in\cz(\q)$ and $\q_{(0)}$ exists, then $F^\bullet\in \cz(\q_{(0)})$. 
\end{prop}

\section{Quotients of the current algebra} \label{q-cur}

Let $\gt q$ be a Lie algebra over $\mK$ and let $\gt q[t]=\gt q\otimes\mK[t]$ be the associated current algebra,
where $[xt^k,yt^m]=[x,y]t^{k+m}$ for $x,y\in\gt q$.  Let $p\in\mK[t]$ be a normalised polynomial of 
degree $n\ge 1$. Then $\gt q[t]/(p)\cong\gt q\otimes(\mK[t]/(p))$ is a Lie algebra and as a vector space it is isomorphic to 
$$
\mW=\mW(\gt q,n)=\gt q{\cdot}1\oplus\gt q\bar t\oplus\ldots\oplus\gt q \bar t^{n-1},
$$
where $\bar t$ identifies with $t+(p)$.   Let $[\,\,,\,]_p$ be the Lie bracket on $\mW$ given by $p$, i.e., 
$\gt q[t]/(p)\cong (\mW,[\,\,,\,]_p)$ as a Lie algebra. 
We identify $\gt q$ with 
$\gt q{\cdot}1\subset \mathbb W$. 
In a particular case $p=t^n$, set 
$\qqn=\gt q[t]/(t^n)$. The Lie algebra $\qqn$ is known as a (generalised) {\it Takiff algebra} modelled on $\gt q$. 
Note that $\qqi\cong\gt q$. 
If $\dim\gt q<\infty$, then 
by~\cite[Thm\,2.8]{rt}, we have 
\beq \label{ind-RT}
\ind\qqn=n{\cdot}\ind\gt q.
\eeq 
We identify each $(\gt q\bar t^k)^*$ with $\gt q^*$.

\begin{prop}\label{roots}
Suppose $p=\prod_{i=1}^u (t-a_i)^{m_i}$, where $a_i\ne a_j$ for $i\ne j$ and  $m_i\ge 1$ for each $i\le u$.  
%%%$\sum_{i=1}^u m_i=n$. 
Then $\gt q[t]/(p)\cong  \bigoplus_{i=1}^u \qqmi$. 
\end{prop}
\begin{proof}
Since $\mK[t]$ is a principal ideal domain, we have $\mK[t]/(p)\cong   \bigoplus_{i=1}^u \mK[t]/((t-a_i)^{m_i})$. 
The isomorphism extends to the tensor product with $\gt q$.
Finally notice that $\gt q[t]/((t-a_i)^{m_i})\cong\gt q[t]/(t^{m_i})$. 
\end{proof}

In the following two statements, we assume that $\gt q$ is finite-dimensional. 

\begin{cl}[cf.~\eqref{ind-RT}] \label{roots-ind}
For any $p\in\mK[t]$ of degree $n$, we have $\ind \gt q[t]/(p)=n{\cdot}\ind\gt q$.  \qed
\end{cl} 

\begin{lm}\label{lm-codim}
For any $p\in\mK[t]$ of degree $n$, we have $n{\cdot}\dim\gt q-\dim(\gt q[t]/(p))^*_{\sf sing}=\dim\gt q-\dim\gt q^*_{\sf sing}$. 
\end{lm}
\begin{proof}
According to ~\cite[Thm\,2.8]{rt}, $\gamma\in\qqn^*$ is regular if and only if 
$\gamma|_{\gt q\bar t^{n{-}1}}\in\gt q^*_{\sf reg}$.  This takes care  of the case $p=t^n$. The general case follows in view of Proposition~\ref{roots}.  
\end{proof}

\begin{ex}\label{ex-diff}
Let us consider an important particular case, where %%the 
$m_i=1$ for each $i$. 
Set $r_i=\dfrac{p}{(t-a_i)}\prod\limits_{j\ne i}(a_i-a_j)^{-1}$.  Then $r_i^2 \equiv r_i \pmod p$. 
This is an explicit application of the Chinese remainder theorem. 
Each subspace $\gt q\bar r_i$ is a Lie  subalgebra of $\gt q[t]/(p)$, isomorphic to $\gt q$, and
\begin{equation} \label{sum-r}
\gt q[t]/(p)=\gt q \bar r_1\oplus \gt q \bar r_2\oplus\ldots\oplus \gt q \bar r_n. 
\end{equation} 
In  particular, $\gt g[t]/(p)$ is semisimple if $\gt g$ is semisimple.  
\end{ex}

\begin{lm} \label{contr0}
Each Lie bracket $[\,\,,\,]_p$ on $\mW$ can be contracted to $\ell_0=[\,\,,\,]_{t^n}$. 
\end{lm}
\begin{proof}
Set $\varphi_s(x \bar t^k)=s^k x\bar t^k$ for $s\in\mK^\times$ and $x\in\gt q$, define further 
$$[\tilde x,\tilde y]_{p,s}=\varphi_s^{-1}([\varphi_s(\tilde x),\varphi_s(\tilde y)]_p) \ \ \text{ for } \ \ \tilde x,\tilde y\in\mW.$$ 
Take  $x,y\in \q$. 
Then $[x\bar t^a,y\bar t^b]_{p,s}=[x,y]\bar t^{a+b}$ if $a+b<n$ and $[x\bar t^a,y\bar t^b]_{p,s}=\sum\limits_{k=0}^{n-1}s^{a+b-k}c_k[x,y]\bar t^k$ with some $c_k\in\mK$ if  $a+b\ge n$. Therefore  
$\lim_{s\to 0}[\,\,,\,]_{p,s}=\ell_0$.     
\end{proof}   

Following the usual terminology of Poisson geometry, say that two Lie brackets on $\mW$ are
{\it compatible} if their sum (or, equivalently, any  linear combination of them) is again a Lie bracket; % In 
compatible   Lie brackets on $\mW$ lead to 
compatible Poisson brackets on $\gS(\mW)$, see e.g. \cite[Sect.~1.8.3]{duzu} or 
\cite[Sect.~1.1]{OY} for details. From now on, assume that $n\ge 2$.

\begin{prop} \label{comp}
Let $p=p_1\in\mK[t]$ be a normalised polynomial of degree $n$. Suppose that
$p_2=p_1+l$, where $l\in\mK[t]$ is a nonzero polynomial of degree at most $1$. %%  linear  
Then the brackets $[\,\,,\,]_{p_1}$ and $[\,\,,\,]_{p_2}$ are compatible.
\end{prop}
\begin{proof}
For $k\le 2n-2$, write $t^k=qp+R$, where $\deg R\le n-1$. Then we have $\deg q\le n-2$ and 
$\deg ql\le n-1$. Thereby $\deg(t^k-q(p+l))\le n-1$ as well.  
In other words, $R-ql$ is the remainder  of $t^k$ modulo $p_2$. 
If  $c_1,c_2\in\mK$ and  $c_1+c_2=1$, then  
$$t^k=q(c_1p_1+c_2p_2)+c_1R+c_2(R-ql),$$ 
where $c_1R+c_2(R-ql)$ is the remainder of $t^k$ modulo 
$c_1p_1+c_2p_2$. Thus here 
\begin{equation} \label{sum}
c_1[\,\,,\,]_{p_1}+c_2 [\,\,,\,]_{p_2}= [\,\,,\,]_{p_3} \ \ \text{ with } \ \ p_3=c_1p_1+c_2p_2.
\end{equation}
In particular, the (half-)sum of $[\,\,,\,]_{p_1}$ and $[\,\,,\,]_{p_2}$ is again a Lie bracket on $\mW$. 
\end{proof} 

%% If 
For  any $p\in\mK[t]$ of degree $n$, %% 
the restriction of $[\,\,,\,]_p$ to $\gt q=\q{\cdot}1\subset\mW$ is the Lie bracket of $\gt q$. However, %% 
$\ell(\gt q,\gt q)=0$ for a bi-linear operation 
 $\ell=[\,\,,\,]_{p_1}-[\,\,,\,]_{p_2}$, where $p_1,p_2\in\mK[t]$ are normalised polynomials of degree $n$.  %%. 
Furthermore, $\ell(\q\bar t^a,\q\bar t^b)=0$, whenever $a+b<n$.

\begin{ex} \label{inf}
Set $[\,\,,\,]_{\infty}=[\,\,,\,]_{t^n-1}-[\,\,,\,]_{t^n}$. Then 
$[x\bar t^a,y\bar t^b]_{\infty}=\begin{cases} 0, & \text{ if }\  a+b<n; \\
[x,y]\bar t^{a+b-n},  & \text{ if }\  a+b\ge n,  \end{cases}$ 
for $x,y\in\gt q$.
\end{ex}

For  $\ell=[\,\,,\,]_{p_1}-[\,\,,\,]_{p_2}$,  $s\in\mK^\times$, and $x,y\in\mW$ set 
$\ell_{(s)}(x,y)=\varphi_s^{-1}(\ell(\varphi_s(x),\varphi_s(y))$, where $\varphi_s\!:\mW\to\mW$ is the same map as in the proof of Lemma~\ref{contr0}. For any $k\in \Z$, we have 
\[s^k\ell_{(s)}(x,y)=s^{-k}\varphi_s^{-1}(\ell(s^k\varphi_s(x),s^k\varphi_s(y)).
\]
If $\ell$ is a Lie bracket and  $\lim_{s\to 0} s^k \ell_{(s)}$ exists, then it is a contraction of $\ell$ in the sense of \cite[Sect.~3]{contr}.

\begin{lm}\label{contr-}
If $p_2-p_1=1$, then $\lim_{s\to 0}s^{-n}\ell_{(s)}=[\,\,,\,]_{\infty}$; if 
$p_2-p_1= t+c$ with $c\in\mK$, then $\lim_{s\to 0}s^{1-n}\ell_{(s)}=[\,\,,\,]_{t^n- t}-[\,\,,\,]_{t^n}$.
\end{lm}
\begin{proof}
If $a+b<n$ and $x,y\in\q$, then $\ell(x\bar t^a,y\bar t^b)=\ell_{(s)}(x\bar t^a,y\bar t^b)=0$. Suppose next  that $a+b=n+u$ with $u\ge 0$. 
Assume first that $p_2=p_1+1$. Then \[s^{-n}\ell_{(s)}(x\bar t^a,y\bar t^b)=[x,y]\bar t^u+\sum_{k=0}^{u-1}s^{u-k}c_k[x,y]\bar t^k\]
 with some $c_k\in\mK$. This case is settled. 

Assume now that $p_2-p_1=t+c$. Then  $s^{1-n}\ell_{(s)}(x\bar t^a,y\bar t^b)=[x,y]\bar t^{u+1}+\sum_{k=0}^{u}s^{u+1-k}c_k[x,y]\bar t^k$ with some $c_k\in\mK$. The limit at zero is exactly the difference 
$[\,\,,\,]_{t^n- t}-[\,\,,\,]_{t^n}$.
\end{proof}

\subsection{Polarisations} \label{s-pol}
%%For  elements
Let $\vec{k}=(k_1,\ldots,k_d)$ be a 
$d$-tuple   of integers, such that   \[0\le k_1\le k_2\le\ldots\le k_d<n.\] 
Suppose $y_1,\ldots , y_d \in \gt q$. %% . 
If we consider the product $Y=\prod_{i} y_i\in\gS^d(\gt q)$, then there is no uniquely defined sequence of factors $y_i$.  
However, the {\it $\vec{k}$-polarisation}  
\[
Y[\vec{k}]:=%%%\frac{1}{m!}
(d!)^{-1}|{\tt S}_d{\cdot}{\vec{k}}|\sum\limits_{\sigma\in{\tt S}_d}   y_1\bar t^{\sigma(k_1)}\ldots y_d \bar t^{\sigma(k_d)}
\]
 %% \Upsilon
of $Y$ is  well-defined. %% element.   
We extend this notion to all elements of $\gS^d(\gt q)$ by linearity. 
%% Furthermore, 
For $F\in\gS^d(\gt q)$, set $\Pol(F)=\left<F[\vec{k}] \mid \vec{k} \ \,\text{as above}\,\right>$. 
%%Since  it,  
It would be convenient to use also the %set 
partition notation for $\vec{k}$, for example,
$(1,\ldots,1,2)$ one can write as $(1^{d{-}1},2)$.

By the construction, if $F\in\cz(\gt q)$, then $F[\vec{k}]\in\gS(\mW)^{\gt q}$ for any $\vec{k}$. 
However, $\gS(\mW)^{\gt q}$ may contain other $\gt q$-invariants, even such invariants  
that are not % linear 
polynomials in polarisations of $\gt q$-invariants in $\gS(\gt q)$. 

\subsection{Symmetric invariants} \label{sub-sym}
Let $\cz_{p}\subset \gS(\mW)$ denote the Poisson centre of  $(\gS(\mW), [\,\,,\,]_{p})$, see~\eqref{PZ1} for the definition.
As Example~\ref{ex-diff} shows, if the roots of $p$ are pairwise distinct, then 
$\cz_p\cong \cz(\gt q)^{\otimes n}$. In an opposite case, where $p=t^n$, the symmetric invariants
of  Takiff algebras $\qqn$ modelled on {\bf reductive} $\gt q$ have been studied in \cite{rt}.  
The answer is that $\cz(\qqn)=\cz_{t^n}$ is a polynomial ring in $n{\cdot}\rk\gt q$ generators. Similar results are obtained 
in \cite{ap,k-T} for some non-reductive $\gt q$. 

An open subset of $\gt q^*$ is said to be {\it big} if its complement does not contain divisors. 

\begin{thm}[\cite{k-T}]   \label{thm:k-T}
Suppose that $\mK[\q^*]^\q=\mK[F_1,\dots,F_m]$ is a graded 
polynomial ring, where $m=\ind\q$. Set 
$\Omega_{\q^*}=\{\xi\in\q^*\mid (\textsl{d}_\xi F_1)\wedge\ldots \wedge(\textsl{d}_\xi F_m)\ne 0\}$, and assume 
that $\Omega_{\q^*}$ is big. For any $n\ge 1$,  the Takiff algebra $\qqn$ has the same properties as $\q$, 
in particular,
%%\begin{itemize} \item 
$\mK[\qqn^*]^{\qqn}$ is a graded polynomial ring of Krull dimension $\ind\qqn=nm$.
%%% }
\end{thm}

In the following, we always assume that each $F_i$ is homogeneous. 

For $\vec{k}$ and $F$ as defined in Section~\ref{s-pol},
set $|\vec{k}|=\sum\limits_i k_i$ and  %% and 
$F^{[\dd]}=\sum\limits_{|\vec{k}|=\dd}F[\vec{k}]$.
%%%, \ \. 
Under the assumptions of Theorem~\ref{thm:k-T}, the algebraically independent generators of 
$\cz(\qqn)$ can be described in the following way, which goes back to \cite{rt},   %% Then 
\begin{equation} \label{kT-gen}
\cz(\qqn)=\mK\!\left[ F_i^{[\dd]} \mid 1\le i\le m,    \  (n-1)\deg F_i-n < |\dd|\le (n-1)\deg F_i \right].
\end{equation}
From this description it is clear that the Poincar{\'e} series of  $\cz(\qqn)$ coincides with that of 
$\cz(\gt q^{\oplus n})$. 

In view of Proposition~\ref{roots}, Theorem~\ref{thm:k-T} leads now to the following statement.

\begin{prop} \label{sym-T}
Suppose $\gt q$ satisfies the assumptions of  Theorem~\ref{thm:k-T}. Then any
$\cz_p$  is a graded polynomial ring of Krull dimension $nm$ and its Poincar{\'e} series  coincides with that of 
$\cz(\gt q^{\oplus n})$. 
\qed
\end{prop}

Suppose that we have a converging sequence $p_j\in\mK[t]$ with $j\in\mathbb N$ of normalised polynomials of degree 
$n$ and $p=\lim\limits_{j\to\infty}p_j$. Then, by the construction,  
\begin{equation} \label{eq-lim-pol}
\lim_{j\to\infty} [\,\,,\,]_{p_j}=[\,\,,\,]_p \ \ \text{ and, assuming that each $p_j$ has  distinct roots,} \  \lim_{j\to\infty} \cz_{p_j}\subset \cz_p,
\end{equation}
where the second limit is taken in the appropriate Grassmannians. 

\begin{prop} \label{sym-lim}
Suppose $\gt q$ satisfies the assumptions of  Theorem~\ref{thm:k-T}, in particular $\mK[\q^*]^\q=\mK[F_1,\dots,F_m]$.
If  $p=\lim\limits_{j\to\infty}p_j$, then  $\lim_{j\to\infty} \cz_{p_j} = \cz_p$. Furthermore, any $\cz_p$ has a set of algebraically independent 
generators $F_{i,\ups}$ with $1\le i\le m$,\, $1\le \ups\le n$ 
 such that $F_{i,\ups}\in \Pol(F_i)$ for all $i$ and $\ups$.
\end{prop}
\begin{proof}
According to Proposition~\ref{sym-T}, $\dim(\cz_p\cap\gS^d(\mW))$ is independent of $p$. Therefore there is a well defined limit $\lim\limits_{j\to\infty} \cz_{p_j}\subset\gS(\mW)$ for any converging sequence  of normalised polynomials $p_j$.  In case $\lim\limits_{j\to\infty} p_j=p$, we have $\lim\limits_{j\to\infty}\cz_{p_j}\subset \cz_p$, because $\lim\limits_{j\to\infty} [\,\,,\,]_{p_j}=[\,\,,\,]_p$. Since the Poincar{\'e} series of 
$\lim\limits_{j\to\infty}\cz_{p_j}$ and of $\cz_p$ coinside, $\lim\limits_{j\to\infty} \cz_{p_j} = \cz_p$. 
%% by . 

Suppose that a normalised polynomial  $\tilde p\in\mK[t]$ with $\deg\tilde p=n$ has $n$ distinct roots $a_1,\ldots,a_n$. Let $\bar r_1,\ldots,\bar r_n$ be the  polynomials from Example~\ref{ex-diff}.   Then 
\begin{equation}  \label{czr}
\cz_{\tilde p}=\cz(\gt q\bar r_1)\otimes\ldots\otimes\cz(\gt q\bar r_n) \  \ \text{ and } \ \  
\cz(\gt q\bar r_\ups)=\mK\left[F_1[\bar r_\ups],\ldots,F_m[\bar r_\ups]\right],
\end{equation}
where $F_i[\bar r_\ups]$ is obtained from $F_i$ by extending the canonical isomorphism $\gt q\to \gt q\bar r_j$ to 
$\gS(\gt q)$. Clearly $F_{i,\ups}:=F_i[\bar r_\ups]\in \Pol(F_i)$ for all $i$, $\ups$ and %Therefore 
$\cz_{\tilde p}$ is freely generated
by $F_{i,\ups}$ with $1\le i\le m$, $1\le \ups\le n$.  

In general, $(\mW,[\,\,,\,]_p)$  is a direct sum of Lie algebras, corresponding to roots of $p$, 
see Proposition~\ref{roots} for details. If $\alpha$ is a simple root, then the corresponding summand is of the form 
$\gt q\bar r_{(\alpha)}$
with $\bar r_{(\alpha)}\in\mK[t]/(p)$ and $\cz(\gt q\bar r_{(\alpha)})$ is generated by $F_i[\bar r_{(\alpha)}]\in\Pol(F_i)$ with $1\le i\le m$.

Suppose now that a root $\alpha$ has multiplicity $k\ge 2$.
Let $\q\langle k\rangle_{(\alpha)}\cong\q\langle k\rangle$ be the direct summand of $(\mW,[\,\,,\,]_p)$ 
corresponding to $\alpha$. 
Then there are $\bar r_{(\alpha,0)},\bar r_{(\alpha,1)}\in\mK[t]/(p)$ such that $\bar r_{(\alpha,0)}\bar r_{(\alpha,1)}=\bar r_{(\alpha,1)}$, \ $\bar r_{(\alpha,1)}^k=0$, 
\ $\bar r_{(\alpha,0)}^2=\bar r_{(\alpha,0)}$, and 
\[
\q\langle k\rangle_{(\alpha)}=\gt q\bar r_{(\alpha,0)}\oplus\gt q\bar r_{(\alpha,1)}\oplus\gt q\bar r_{(\alpha,1)}^2\oplus\ldots\oplus \gt q\bar r_{(\alpha,1)}^{k{-}1}.
\]
Let $F_i^{[\dd]}\in\cz(\q\langle k\rangle)$ be given by \eqref{kT-gen} with $n$ replaced by $k$. 
Define  $\Phi_\alpha\!:\mK[t]/(t^k)\to \mK[t]/(p)$ by setting $\Phi_\alpha(1)=\bar r_{(\alpha,0)}$ and 
$\Phi_\alpha(\bar t^u)=\bar r_{(\alpha,1)}^u$ for $1\le u<k$. The map $\Phi_\alpha$ extends to 
$\q\langle k\rangle$ and the polynomials $\Phi_\alpha(F_i^{[\dd]})$ with $1\le i\le m$, \ $(k-1)\deg F_i-k < |\dd|\le (k-1)\deg F_i$ 
generate $\cz(\q\langle k\rangle_{(\alpha)})$. 

Since $F_i^{[\dd]}\in\Pol(F_i)$ for all $i$ and $\dd$, this holds also for all $\Phi_\alpha(F_i^{[\dd]})$. The result about generators of 
$\cz_p$ follows now from Proposition~\ref{roots}.
%%%
%% }
\end{proof}

\begin{Rem}
Let us consider the contraction $\lim\limits_{s\to 0}[\,\,,\,]_{p,s}\to \ell_0$ presented in Lemma~\ref{contr0}.
For any $F\in \cz(\mW,[\,\,,\,]_p)$, we have then  $\varphi_s^{-1}(F)\in\cz(\mW,[\,\,,\,]_{p,s})$ and 
%%% we have 
$\lim\limits_{s\to 0} \mK \varphi_s^{-1}(F)=\mK F^\bullet $ in terms of Proposition~\ref{prop:bullet}.
Suppose that $\q$ satisfies  the assumptions of Theorem~\ref{thm:k-T}.
Arguing by induction on $\deg F_i$, %% one can show  
 one can find a generating set 
$\{F_{i,u}\}\subset    \cz(\mW,[\,\,,\,]_p)$ such that $\cz(\mW,\ell_0)$ is freely generated by the components $F_{i,u}^\bullet$. 
\end{Rem}

\begin{ex} \label{-t}
For the feature use, we describe generators of  $\cz_p$ in case $p=t^n-ct$ with $c\in\mK^\times$.
Fix  $\alpha\in\mK$ with $\alpha^{n-1}=c$ and let $\zeta\in\mK$ be a  primitive $(n{-}1)$-th root of unity.
Then $\bar r_1=\frac{-1}{c}(t^{n{-}1}-c)$ and 
\[
\bar r_i=\frac{1}{c(n{-}1)}t(t^{n-2}+\alpha\zeta^{i} t^{n{-}3}+\ldots+(\alpha\zeta^{i})^{n{-}3}t+(\alpha\zeta^{i})^{n{-}2}) \  
\text{ if } \ i\ge 2
\]
 %%for the polynomials from 
 in terms of 
 Example~\ref{ex-diff}. For each $1\le i\le n$, we have %%Therefore 
\[
\cz(\gt q\bar r_i)=\{ F[\bar r_i] \mid  F\in\cz(\gt q)\}.
\]
 For 
$F\in\gS^d(\gt q)^{\gt q}$, one obtains
\begin{align} 
& (-c)^d F[\bar r_1]=F[(n{-}1)^d]+\sum_{k=1}^{d} (-c)^k F[0^k,(n{-}1)^{d{-}k}]; \label{tn-c0} \\
& c^d(n{-}1)^d F[\bar r_i]=\sum_{\dd=d}^{d(n{-}1)}  (\alpha\zeta^i)^{d(n{-}1)-\dd}  \sum_{|\vec{k}|=\dd;\, k_1\ge 1} F[\vec{k}] \ \  \text{ if } \ i\ge 2.  \label{tn-c}
\end{align}
%%where formula. 
\end{ex}

\section{Poisson-commutative subalgebras} %% 
 \label{sub-PC}

Roughly speaking, 
a {\it bi-Hamiltonian system} is a pair of compatible Poisson structures
$\{\,\,,\,\}',\{\,\,,\,\}''$ (or rather a pencil $\{a\{\,\,,\,\}'+b\{\,\,,\,\}'' \mid a,b\in\mK\}$ spanned by them) \cite[Sect.~1.8]{duzu}, \cite{GZ}.  Let $\pi'$, $\pi''$ be the Poisson tensors of  $\{\,\,,\,\}',\{\,\,,\,\}''$.  
Then $\pi_{a,b}=a\pi'{+}b\pi''$ is the Poisson tensors of  $a\{\,\,,\,\}'+b\{\,\,,\,\}''$. 
For almost
all $(a,b)\in\mK^2$, $\rk\!(a\pi'{+}b\pi'')$  has one and the same (maximal) value, let it be $\br$,
 and we say that $a\{\,\,,\,\}'+b\{\,\,,\,\}''$ is
{\it regular} (or that $(a,b)$ is a {\it regular} point) if  $\rk\!(a\pi'{+}b\pi'')=\br$. %% is
%%% Let 
The Poisson centres of regular structures in the pencil generate a subalgebra, which  is Poisson-commutative w.r.t. 
all  Poisson brackets in the pencil,  see e.g. \cite[Sect.~10]{GZ} or \cite[Sect.~2]{OY} for an explanation of this phenomenon.   %%

Let $L=    \left<\{\,\,,\,\}',\{\,\,,\,\}''\right>_{\mK}$ be a pencil of compatible polynomial Poisson structures
on the affine space $\mA^N$, i.e., each $a\{\,\,,\,\}'+b\{\,\,,\,\}''$ is a Poisson bracket on the 
polynomial ring $\mK[\mA^N]$. Let 
$\mA^N_{(a,b),{\sf sing}}$ be the singular set of $a\{\,\,,\,\}'+b\{\,\,,\,\}''$, i.e.,
\[\mA^N_{(a,b),{\sf sing}}=\{\xi\in\mA^N \mid \rk\pi_{a,b}(\xi)<\rk\pi_{a,b}\}.
\] 
Set further $\mA^N_{L,{\sf sing}}=\bigcup_{(a,b)}    \mA^N_{(a,b),{\sf sing}}$ and 
$\mA^N_{L,{\sf reg}}=\mA^N\setminus \mA^N_{L,{\sf sing}}$. Since 
$c\pi_{a,b}=\pi_{ca,cb}$ for any $c$ in $\mK^\times$, the Zariski closed subset $\mA^N_{(a,b),{\sf sing}}$ depends on 
$(a\,{:}\,b)\in\mP^1$ and not  on its representative $(a,b)$. Also the numbers $\rk\pi_v$
and $\rk\pi_v(\xi)$ with $v\in\mP^1$, $\xi\in\mA^N$ are  well-defined. 
Consider now 
\[
\mX=\{(\xi,v)\mid v\in\mP^1, \xi\in \mA^N_{v,{\sf sing}}\}\subset \mA^N\times\mP^1.
\]

\begin{lm} \label{lm-X}
%%{\sf (i)} The subset $\mX$ is an algebraic subvariety of  $\mA^N\times\mP^1$. \\[.2ex]
{\sf (i)} If $\dim  \mA^N_{v,{\sf sing}} \le N-2$ for generic $v\in \mP^1$, then $\mA^N_{L,{\sf reg}}$
contains a non-empty open subset of $\mA^N$.
%%%\ne\varnothing$.
 \\[.2ex]
{\sf (ii)} If $\dim  \mA^N_{v,{\sf sing}} \le N-3$ for generic $v\in \mP^1$ and $\dim \mA^N_{v,{\sf sing}} \le N-2$ for all  $v\in \mP^1$, then we have $\dim  \overline{\mA^N_{L,{\sf sing}}}\le N-2$.
\end{lm}
\begin{proof}
%\noindent
 {\sf (i)} 
Let $\mathrm{pr}_1\!:\mA^N\times\mP^1\to \mA^N$, $\mathrm{pr}_2\!:\mA^N\times\mP^1\to \mP^1$
 be the projections on the first  and the second factors. Let $\mX_i\subset\overline{\mX}$ be an irreducible component. Then 
either $\overline{\mathrm{pr}_2(\mX_i)}=\mP^1$ or $\mathrm{pr}_2(\mX_i)$ is a point, say $v_i$. 
In the latter case, $\mX_i\cap \mA^N_{v_i,{\sf sing}}{\times}\{v_i\}$ is a dense subset of $\mX_i$. Hence 
$\mX_i\subset \mA^N_{v_i,{\sf sing}}{\times}\{v_i\}$ and 
$\dim\mX_i\le N-1$, because     
$\mA^N_{v_i,{\sf sing}}$ is a proper closed  subset of $\mA^N$.  

In the former case, 
\[
\mX_i \cap \{(\xi,v)\in  \mA^N{\times}\,\mP^1 \mid  \rk\pi_v=\br=\max_{a,b}\rk\pi_{a,b}, \ \rk\pi_v(\xi)<\br\}
%%
%%% \mA^N_{v,{\sf sing}}{\times} \{v\}\mid \}
\]
is a dense subset of $\mX_i$.  
In particular, $\mX_i$ is contained in the closed subset 
$$\{(\xi,v)\in  \mA^N{\times}\,\mP^1 \mid \rk\pi_v(\xi)<\br\}.$$ 
By the fibre dimension theorem,
$\dim\mX_i\le  N-1$.
By  definition,  $\mA^N_{L,{\sf sing}}=\mathrm{pr}_1(\mX)\subset\mathrm{pr}_1(\overline{\mX})$ and hence 
$\dim\overline{\mA^N_{L,{\sf sing}}}\le N-1$ as well.

%%The Poisson tensors $\pi'$ and $\pi''$ are polynomial bi-vector fields on $\mA^N$. 
%% %Hence $\mX$ is defined by polynomial equations on $\xi$ and $v=(a\,{:}\,b)$ that are 
%% homogeneous (of degree $\br=\max_{a,b}\rk\pi_{a,b}$) in $a$, $b$.  This settles {\sf (i)}. 

%\noindent
{\sf (ii)} In this case, we have  $\dim\mX_i\le N-2$ for each irreducible component of $\overline{\mX}$ and hence 
$\dim\overline{\mA^N_{L,{\sf sing}}}\le \dim\overline{\mX} \le N-2$. 
\end{proof}

Let  $\cP(\xi)=\{ \pi_{a,b}(\xi)\mid a,b\in\mK\}$ with $\xi\in\mA^N$ be a pencil of skew-symmetric $2$-forms on 
$T^*_\xi\mA^N$,
which is spanned by $\pi'(\xi)$ and $\pi''(\xi)$.  A $2$-form in this pencil is said to be {\it regular\/} if 
$\rk\pi_{a,b}(\xi)=\br$. Otherwise, it is {\it singular}. 
%% S big. 
Set 
\[V(\xi)=\sum_{a,b:\, \rk\pi_{a,b}(\xi)=\br} \ker\pi_{a,b}(\xi),
\] 
$\bb=\bb(L)=\frac{1}{2}\br + (N-\br)$. 

\begin{prop}[{\cite[Proposition~1.5]{kruks}}]       \label{P} 
If %%
$\xi\in\mA^N$ and $\rk\pi_v(\xi)=\br$ for some $v\in \mP^1$, then   
\begin{itemize}
\item[{\sf (i)}] $\dim V(\xi)=\bb$ if and only if all  nonzero $2$-forms in $\cP(\xi)$ are regular; 
\item[{\sf (ii)}] if \ $\mathcal P(\xi)$  contains a singular line, say
$\mK \omega$, % is a 
then $\dim V(\xi)\le N-\br+\frac{1}{2}\rk \omega$, 
\item[{\sf (iii)}] furthermore, $\dim V(\xi) = N-\br+\frac{1}{2}\rk\omega$  if and only if\/ 
\begin{itemize}
\item \ $\mK \omega$ is the unique 
 singular line in $\mathcal P(\xi)$, and 
\item \ $\rk\!(\pi(\xi)|_{\gt v})=\dim\gt v-(N-\br)$ for $\gt v=\ker \omega$ and any $\pi(\xi)\in\cP(\xi)\setminus \mK \omega$.  
\end{itemize}\end{itemize}
\end{prop} 

\subsection{Subalgebras $\gZ(p,p+l)$}\label{sub-Z} 
From now on, assume that $[\gt q,\gt q]\ne 0$ and  $\dim\gt q<\infty$. 
Let $\{\,\,,\,\}_p$ be the Poisson bracket on $\gS(\mW)$ corresponding to $[\,\,,\,]_p$. 
Fix two different normalised polynomials
$p_1,p_2\in\mK[t]$ of degree $n\ge 2$ such that $\deg(p_1-p_2)\le 1$. 
Set  $L(p_1,p_2):=\left<\{\,\,,\,\}_{p_1},\{\,\,,\,\}_{p_2}\right>$.
By Proposition~\ref{comp},  
$$\{\,\,,\,\}_{a,b}:=a\{\,\,,\,\}_{p_1}+b\{\,\,,\,\}_{p_2}\in L(p_1,p_2)
$$ 
is a Poisson bracket on $\gS(\mW)$ for all   $a,b\in\mK$.
Therefore  $L(p_1,p_2)$ is  a pencil of compatible polynomial Poisson structures on $\mW^*\cong\mA^N$.
Here $N=n{\cdot}\dim\gt q$.
Let $\cz_{a,b}\subset \gS(\mW)$ denote the Poisson centre of  $(\gS(\mW), \{\,\,,\,\}_{a,b})$ and
let $\pi_{a,b}$ be the Poisson tensor of $\{\,\,,\,\}_{a,b}$. 
If $a+b=1$, then $\{\,\,,\,\}_{a,b}=\{\,\,,\,\}_{ap_1{+}bp_2}$. 
%%$ }.  
Corollary~\ref{roots-ind} implies now that 
$\rk\pi_{a,b}=n(\dim\gt q-\ind\gt q)$  if  $a\ne -b$. 
Hence $(a,b)$ is a regular point, whenever $a\ne -b$. %%% . %%% , then . 

Let $\gZ=\gZ(p_1,p_2)\subset\gS(\mW)$ be the subalgebra  generated by  $\cz_{a,b}$  with $(a,b)$ being regular.
%% $a+b=1$.   
Then $\{\gZ,\gZ\}_{a,b}=0$ for all $a,b$ according to the general method outlined at the beginning of Section~\ref{sub-PC}. 
%%above. %=\{\gZ,\gZ\}_{p_2}$. 
If $c\in\mK^\times$, then $\cz_{a,b}=\cz_{ca,cb}$ for all $a,b\in\mK$.
We have %also % $a+b=0$, then
$\gt q\subset \cz_{1,-1}$. Since $[\gt q,\gt q]_{p_1}\ne 0$, $(1,-1)$ is not a regular point. 
Thus $\gZ=\mathsf{alg}\langle \cz_{a,b} \mid a+b=1\rangle$.
The transcendence degree of $\gZ$ depends on the properties of
$\gt q$ and of 
$\ell=\{\,\,,\,\}_{1,-1}$.  
As we have already observed, $\gt q$ is a Lie subalgebra of $(\gS(\mW),[\,\,,\,]_{ap_1+bp_2})$, whenever $a+b=1$, cf.~\eqref{sum}. Thereby $\cz_{a,1-a}\subset \gS(\mW)^{\gt q}$ for each $a\in\mK$. Hence also 
 $\gZ\subset\gS(\mW)^{\gt q}$ and we have 
 %%  it is 
 \begin{equation} \label{bound}
\trdeg \gZ  \le  \bb(\mW,[\,\,,\,]_{p_1})-\bb(\gt q)+\ind\gt q=(n-1)\bb(\gt q)+\ind\gt q=:\bb(\q,n)
\end{equation}
according to \cite[Prop.\,1.1]{m-y} (for the computation of $\bb(\mW,[\,\,,\,]_{p_1})$, we have used Corollary~\ref{roots-ind}). 

In case $(p_1-p_2)\in\mK^\times$, we have $L(p_1,p_2)=L(p_1,p_1+1)$. 
If $\deg l=1$, then $L(p_1,p_2)=L(p_1,p_1+t+c)$ for some $c\in \mK$. By a change of variable, 
$t\mapsto t+c$, we can pass to the pencil $L(p,p+t)$. This will change the initial polynomial $p_1$.
However, if we consider a statement that is valid for all $p$, then in the proof we may safely assume that
$l$ is either $1$ or $t$. 

\begin{rmk} \label{rem-rd}
If $l=1$, then clearly almost all polynomials $p+\alpha\in\{ap+(1-a)(p+1)\mid a\in\mK\}$ have distinct roots. 
This is also true for $\tilde p\in \{ap+(1-a)(p+t)\mid a\in\mK\}$, although may not be obvious. 
A brief explanation is that the resultant $\mathrm{Res}(\tilde p,\partial_t \tilde p)$, where $\tilde p=p+\alpha t$,
 is a polynomial in 
$\alpha$ of degree $n$ with the leading coefficient $(1-n)^{n-1}$, hence it is nonzero for almost all 
$\alpha\in\mK$.  Another  explanation involves the homomorphism $\varphi_s\!: \mK[t]\to\mK[t]$ 
such that $\varphi_s(t)=st$ with $s\in\mK^\times$. This map extends to $\gt q[t]$ and is  defined on $\mW$ as well. 
We have already used  it in Section~\ref{q-cur}. 
\end{rmk}

According to Lemmas~\ref{contr0}, \ref{contr-},  
   %%%%%%% 
%% Then 
the two-dimensional space $L(p_1,p_2)$ can be contracted to 
$L(t^n,t^n-t)$ if $\deg l=1$ and to $L(t^n,t^n-1)$ if $\deg l=0$.   

\begin{prop}\label{bound-dim}
For any $F\in\gS^d(\gt q)^{\gt q}$, we have $\dim(\Pol(F)\cap \gZ(p_1,p_2))\ge d(n{-}1)+1$. 
\end{prop}
\begin{proof}
We have no restriction on $p_1$, therefore there is no harm in assuming that %%%... the case of 
either $l=t$ or $l=1$. %%%%%%  .....
For $d=0$, the statement is obvious, thereby suppose that $d\ge 1$.  

Suppose first that $l=t$. Consider $p_{s,k}=s^n\varphi_{s}^{-1}(p)+kt$, where $k\in\Z$ and $1\le k\le dn$. 
Here $\lim\limits_{s\to 0} p_{s,k}=t^n+kt$. 
There is a non-empty open subset $U\subset \mK$ such that $0\in U$ and any $p_{s,k}$
%%%%If $s$ 
has distinct roots for each $s\in U$. 
Until the end of the proof, assume that $s\in U$.
Let $\bar r_{i,s,k}$ and $\bar r_{i,k}$ be polynomials from Example~\ref{ex-diff} associated with $p_{s,k}$ and 
with $t^n+kt$, respectively. 
The numbering is chosen in such a way that 
%% such that 
$\lim\limits_{s\to 0}\bar r_{i,s,k}=\bar r_{i,k}$ for all $i$ and $k$.  
%%%%% where 
Then clearly $\lim\limits_{s\to 0} F[\bar r_{i,s,k}]=F[\bar r_{i,k}]$. 
Using \eqref{tn-c0}, \eqref{tn-c}, we conclude that 
\[\dim\left< F[\bar r_{i,k}] \mid 1\le k\le dn\right>=d(n{-}1)+1.
\]
Thereby there is at least one $s\in U\setminus\{0\}$ such that $\dim\left< F[\bar r_{i.s,k}] \mid 1\le k\le dn\right>\ge d(n{-}1)+1$. 
%It 
By  definition,
$s^n\varphi_{s}^{-1}(p)+kt=s^n\varphi_{s}^{-1}(p+ks^{1-n}t)$.  Hence
$\varphi_s(F[\bar r_{i,s,k}])\in\gZ(p,p+t)$ for all $i$ and $k$.  Since $\varphi_s$ is an isomorphism, we are done in 
this case. 

The case $l=1$ can be treated similarly.  Instead of the polynomials $t^n+kt$, one works with 
$t^n+k$. Here $\left< F[\bar r_{i,k}] \mid 1\le k\le dn\right>=\left<F^{[\dd]}\mid  0\le \dd\le d(n-1) \right>$ and this space has dimension $d(n-1)+1$ as well. 
\end{proof}    
   
\begin{ex}[A curiosity] \label{n=2}   
   Suppose that $n=2$. 
Here any two Lie brackets    $[\,\,,\,]_{p_1}$ and $[\,\,,\,]_{p_2}$ are compatible. 
Assume that $p_1\ne p_2$. For any $F\in\gS^d(\q)$, 
we have  $\dim\Pol(F)=d+1$. If $F\in\gS^d(\q)^\q$,  then 
$\Pol(F)\subset \gZ(p_1,p_2)$ by Proposition~\ref{bound-dim}. 
Suppose that $\q$ satisfies the assumptions of Theorem~\ref{thm:k-T}, in particular, 
$\cz(\q)=\mK[F_1,\ldots,F_m]$. Then \[\gZ=\gZ(p_1,p_2)=\mathsf{alg}\langle \Pol(F_i) \mid  1\le i\le m\rangle\]
by Proposition~\ref{sym-lim}. 
 This $\gZ$ is independent of the choice of  $(p_1,p_2)$ and arises from a tri-Hamiltonian system
$(\mW,[\,\,,\,]_{t^2},[\,\,,\,]_{t^2+1},[\,\,,\,]_{t^2+t})$. 

Another feature of the case $n=2$ is directly related to {\it  Mishchenko--Fomenko subalgebras}, see e.g. \cite{vi90} for 
a definition. 
For $\gamma\in\q^*$, consider the homomorphism of commutative algebras 
$\varrho_\gamma\!:\gS(\mW)\to\gS(\q)$ with $\varrho_\gamma(x{\cdot}1+y{\cdot}\bar t)=x+\gamma(y)$ (hier $x,y\in\q$). Then $\varrho_\gamma(\gZ(p_1,p_2))$ contains the Mishchenko--Fomenko subalgebra $\cF_\gamma\subset\gS(\q)$ associated with $\gamma$. If  $\q$ satisfies the assumptions of Theorem~\ref{thm:k-T}, then 
$\varrho_\gamma(\gZ)=\cF_\gamma$. 
\end{ex}   
   
\begin{lm}\label{ind-}
For $\ell=[\,\,,\,]_{p_1}-[\,\,,\,]_{p_2}$, we have $\ind\!(\mW,\ell)=\dim\q+(n{-}1)\ind\gt q$.
\end{lm}    
\begin{proof}
First we consider the particular cases of $[\,\,,\,]_{t^n-1}-[\,\,,\,]_{t^n}$ and $[\,\,,\,]_{t^n-t}-[\,\,,\,]_{t^n}$.
If $\ell=[\,\,,\,]_{t^n-1}-[\,\,,\,]_{t^n}$, then $(\mW,\ell)$ is an $\mathbb N$-graded Lie algebra, isomorphic to 
$(\tilde t\q[\tilde t])/(\tilde t^{n{+}1})$ and to %%, 
the nilpotent radical of 
$\q\langle n{+}1\rangle$, see Example~\ref{inf}. Take $\xi\in\mW^*$ such that 
$\bar\xi:=\xi|_{\gt q{\cdot}1}\in\gt q^*_{\sf reg}$. Then 
\begin{equation} \label{1-1}
\ker\pi_{1,-1}(\xi)=\gt q_{\bar\xi}\bar t^{n{-}1}\oplus\ldots\oplus\gt q_{\bar\xi}\bar t\oplus\gt q{\cdot}1
\end{equation}
 and $\dim\ker\pi_{1,-1}(\xi)=(n{-}1)\ind\q+\dim\q$. Thereby $\ind\!(\mW,\ell)\le  (n{-}1)\ind\q+\dim\q$ in this case. 

Suppose now that $\ell=[\,\,,\,]_{t^n-t}-[\,\,,\,]_{t^n}$. Then $(\mW,\ell)\cong\q\langle n{-}1\rangle\oplus\gt q^{\rm ab}$.  Here $\ind\!(\mW,\ell)=\ind\q\langle n{-}1\rangle+\ind\gt q^{\rm ab} = (n{-}1)\ind\q+\dim\q$ by \cite{rt}. 

According to Lemma~\ref{contr-}, a general $\ell$ can be contracted to either $[\,\,,\,]_{t^n-1}-[\,\,,\,]_{t^n}$ 
or $[\,\,,\,]_{t^n-t}-[\,\,,\,]_{t^n}$.  Under this procedure, the index cannot decrease,
%% under 
therefore always 
%%%%%......  
$\ind\!(\mW,\ell)\le\dim\q+(n{-}1)\ind\gt q$. %%% .................. 

Take now $\xi\in\mW^*$ such that $\rk\pi_{1,0}(\xi)=n(\dim\gt q-\ind\gt q)$ and $\bar\xi\in\q^*_{\sf reg}$. Then %% a holds 
$\rk\pi_{a,b}(\xi)=n(\dim\gt q-\ind\gt q)$ for all pairs $(a,b)$ from a 
 non-empty open subset of $\mK^2$. Note that $\ker\pi_{a,b}(\xi)$ depends only 
on $v=(a:b)\in\mP^1$ and not on its representative. There is a sequence of points $v(k)\in\mP^1\setminus\{(1\,{:}\,-1)\}$ such that 
$\dim\ker\pi_{v(k)}(\xi)=n{\cdot}\ind\q$ and $\lim_{k\to\infty}v(k)=(1\,{:}\,-1)$.  
Set $\gt v=\lim_{k\to\infty}\ker\pi_{v(k)}(\xi)$. Then $\dim\gt v=n{\cdot}\ind\q$ and $\gt v\subset\ker\pi_{1,-1}(\xi)$. 
Furthermore, since $\pi_{a,1-a)}(\xi)|_{\mW{\times}\q}=\pi_{1,0}(\xi)|_{\mW{\times}\q}$ for each $a\in\mK$, we have  
$\pi_{1,0}(\xi)(\gt v,\q)=0$.  
Thereby $\gt v\cap\gt q$ is contained in the stabiliser $\gt q_{\bar \xi}$ of $\bar\xi$ and hence  $\dim(\gt v\cap\gt q)\le\ind\q$. Since $\ell(\gt q,\mW)=0$, we have 
$\dim\ker\pi_{1,-1}(\xi)\ge \dim\q+(\dim\gt v-\ind\q)= \dim\q+(n{-}1)\ind\q$ and this inequality holds for generic points $\xi\in\mW^*$.
Thus  $\ind\!(\mW,\ell)\ge\dim\q+(n{-}1)\ind\gt q$.
\end{proof}
    
In the proof of Lemma~\ref{ind-}, we have seen that for generic $\xi\in\mW^*$,
\begin{equation} \label{eq-inf-v}
\ker\pi_{1,-1}(\xi)=\gt q+\gt v, \ \text{ where } \ \pi_{1,0}(\xi)(\gt q,\gt v)=0, \ \dim\gt v=n{\cdot}\ind\q,  \ \dim(\gt q\cap\gt v)=\ind\q.
\end{equation}    
In the above sentence, generic means that $\xi\in\mW^*_{(1,0),{\sf reg}}\cap\mW^*_{(1,-1),{\sf reg}}$ and  $\bar\xi\in\q^*_{\sf reg}$. 
    
\begin{thm} \label{trdeg}
Suppose that $\trdeg\gS(\gt q)^\gt q=\ind\gt q$ and that $\gt q$ has  the {\sl codim}--$2$ property.
Then $\trdeg\gZ(p,p+l)= \bb(\q,n)$ for any $p\in\mK[t]$ of degree $n$ and any nonzero
$l\in\mK[t]$ with $\deg l\le 1$. %%of .
\end{thm}
\begin{proof}
By Lemma~\ref{lm-codim}, each Lie algebra $(\mW,[\,\,,\,]_{\tilde p})$ with $\deg\tilde p=n$ also has  the {\sl codim}--$2$ property. Then 
by Lemma~\ref{lm-X}{\sf (i)}, $\mW^*_{L,{\sf reg}}$ contains a non-empty open subset of $\mW^*$ for any $L=L(p,p+l)$. 
Without loss of generality, we may assume that $p$ has distinct roots, cf. Remark~\ref{rem-rd}.  
Since $\trdeg\gS(\gt q)^\gt q=\ind\gt q$, there is another  non-empty open subset $U_{\sf inv}\subset\mW^*$ such that 
$\dim\textsl{d}_\gamma \cz_p=n{\cdot}\ind\q$ for each $\gamma\in U_{\sf inv}$. 

Let $\gamma\in \mW^*_{L,{\sf reg}}\cap U_{\sf inf}$ be such that $\bar\gamma:=\gamma|_{\q{\cdot}1}\in\q^*_{\sf reg}$. 
Consider the pencil \[\cP(\gamma)=\{ \pi_{a,b}(\gamma)\mid a,b\in\mK\}\] corresponding to $L$. 
If $a\ne -b$, then $\pi_{a,b}(\gamma)$ is a regular form. For the singular line $\mK \pi_{1,-1}(\gamma)$, we have
$\rk \pi_{1,-1}(\gamma)=(n-1)(\dim\q-\ind\q)$ by Lemma~\ref{ind-}. Furthermore, 
$\widetilde{\gt v}:=\ker\pi_{1,-1}(\gamma)=\gt q\oplus\gt v_0$, where $\pi_{1,0}(\gamma)(\q,\gt v_0)=0$, see \eqref{eq-inf-v}. Next we compute $\rk\!(\pi_{1,0}(\gamma)|_{\widetilde{\gt v}})$. Since $\bar\gamma\in\q^*_{\sf reg}$, 
the rank in question is larger than or equal to $\dim\gt q-\ind\gt q$; also 
\begin{equation} \label{rkv}
\rk\!(\pi_{1,0}(\gamma)|_{\widetilde{\gt v}})\le \dim\widetilde{\gt v} - n{\cdot}\ind\q=\dim\q-\ind\q
\end{equation}
by~\cite[Appendix]{OY}.  Thus, we have the equality here. By Proposition~\ref{P}, 
\begin{equation} \label{dimV}
\dim V(\gamma)=n{\cdot}\ind\q+\frac{n-1}{2}(\dim\q-\ind\q)=\bb(\q,n)
\end{equation}
 for $V(\gamma)=\sum_{a\in\mK}\ker\pi_{a,1-a}(\gamma)$. 

Since $\gamma\in U_{\sf inv}$ and in view of \eqref{incl}, we have $\textsl{d}_\gamma\cz_p=\ker\pi_{1,0}(\gamma)$. 
For almost all $\alpha\in\mK$, the polynomial $p+\alpha l$ has distinct roots, cf. Remark~\ref{rem-rd}. 
In other words,
 \[
 U_{\sf rt}=\{\alpha\in \mK \mid  (\mW,[\,\,,\,]_{p+\alpha l})\cong\gt q^{\oplus n}\}
   \]
  is a non-empty open subset of $\mK$. %%  be the 
If $\alpha\in   U_{\sf rt}$, then the Poisson centre $\cz_{p+\alpha p}$ depends on $\alpha$ continuously, cf. \eqref{eq-lim-pol}. Thereby 
$\textsl{d}_\gamma\cz_{a,1-a}=\ker\pi_{a,1-a}(\gamma)$ for almost all $a\in\mK$, i.e., there is a non-empty open subset 
$\Omega\subset\mK$ such that $\textsl{d}_\gamma\cz_{a,1-a}=\ker\pi_{a,1-a}(\gamma)$ for each $a\in\Omega$. 
Set $\gZ= \gZ(p,p+l)$.
Now 
\[
\textsl{d}_\gamma \gZ = \sum_{a\in\mK}\textsl{d}_\gamma \cz_{a,1-a} \supset 
\sum_{a\in \Omega}\textsl{d}_\gamma \cz_{a,1-a}=\sum_{a\in \Omega} \ker\pi_{a,1-a}(\gamma)=
\sum_{a\in\mK} \ker\pi_{a,1-a}(\gamma)= V(\gamma),
\]
because $\sum_{a\in \Omega} \ker\pi_{a,1-a}(\gamma)=
\sum_{a\in\mK} \ker\pi_{a,1-a}(\gamma)$ by \cite[Appendix]{OY}. Thus 
$\trdeg\gZ\ge \dim V(\gamma)= \bb(\q,n)$. Since also $\trdeg\gZ \le \bb(\q,n)$ by \eqref{bound}, the result follows. 
%%............. ........................
%%%%  
% 
\end{proof}

\subsection{Gaudin Hamiltonians}      \label{sb-GH}
Suppose that $\gt q=\gt g$ is semisimple. 
Let $\{x_i\mid 1\le i\le\dim\gt g \}$ be 
a basis of $\gt g$ that is orthonormal w.r.t. the Killing form $\kappa$. Set $\gt h=\gt g^{\oplus n} \cong \gt g[t]/(t^n-1)$ and 
let $x_i^{(k)}\in\gt h$ be a copy of $x_i$ belonging to the   $k$-th copy of $\gt g$. 
Choose a vector $\vec{z}=(z_1,\ldots,z_n)\in \mK^n$ such that $z_j\ne z_k$ for $j\ne k$ 
 and consider the quadratic elements 
\begin{equation}\label{eq-GH}
{\mathcal H}_k=\sum_{j\ne k} \frac{\sum_{i=1}^{\dim\gt g } x_i^{(k)} x_i^{(j)}}{z_k-z_j}  , \enskip 1\le k\le n\,,
\end{equation}
called {\it Gaudin Hamiltonians}. They can be regarded as elements of either 
$\U(\gt g)^{\otimes n}\cong\U(\gt h)$ or $\gS(\gt h)$. %%, 
Note that $\sum_{k=1}^n {\mathcal H}_k=0$. 

For $k\ne s$, we have
\[
\begin{array}{l}
[{\mathcal H}_k,{\mathcal H}_s]=\sum\limits_{\nu\ne k,s}\left( \dfrac{1}{(z_k-z_\nu)(z_s-z_\nu)}\sum\limits_{i,j} 
x_i^{(k)}x_j^{(s)} [x_i,x_j]^{(\nu)}\right. + \\
+\dfrac{1}{(z_k-z_s)(z_s-z_\nu)} \sum\limits_{i,j} x_i^{(k)} x_j^{(\nu)} [x_i,x_j]^{(s)} +
\dfrac{1}{(z_k-z_\nu)(z_s-z_k)} \left.\sum\limits_{i,j} x_i^{(\nu)} x_j^{(s)} [x_i,x_j]^{(k)} \right).\\
\end{array}
\]
Next $\sum\limits_{i,j} x_i^{(k)}x_j^{(s)}[x_i,x_j]^{(\nu)}=\sum\limits_{i,j,\ups}\kappa([x_i,x_j],x_{\ups}) x_i^{(k)}x_j^{(s)}x_{\ups}^{(\nu)}=-\sum\limits_{i,j,\ups}\kappa([x_i,x_{\ups}],x_j) x_i^{(k)}x_{\ups}^{(\nu)}x_j^{(s)}$. Thereby 
$[{\mathcal H}_k,{\mathcal H}_s]$ is equal to 
\[
\begin{array}{l}
\sum\limits_{\nu\ne k,s}\left( \dfrac{1}{(z_k-z_\nu)(z_s-z_\nu)}- \dfrac{1}{(z_k-z_s)}\left(\dfrac{1}{z_s-z_\nu}-\dfrac{1}{z_k-z_\nu}\right)\right)\sum\limits_{i,j,\ups}\kappa([x_i,x_j],x_\ups)x_i^{(k)}x_j^{(s)}x_{\ups}^{(\nu)}=\\
=\sum\limits_{\nu\ne k,s}\left( \dfrac{1}{(z_k-z_\nu)(z_s-z_\nu)}- \dfrac{1}{(z_s-z_\nu)(z_k-z_\nu)}\right)\sum\limits_{i,j,\ups}\kappa([x_i,x_j],x_\ups)x_i^{(k)}x_j^{(s)}x_{\ups}^{(\nu)}=0.
\end{array}
\]
As we have just seen,
%%% 
$[{\mathcal H}_k,{\mathcal H}_j]=0$ for any $j,k$. 
Therefore the Gaudin Hamiltonians also Poisson-commute with each other.
Higher Gaudin Hamiltonians are elements of $\U(\gt g)^{\otimes n}$ that commute with all ${\mathcal H}_k$.

By the construction, each ${\mathcal H}_k$ is an invariant of the diagonal copy of $\gt g$, i.e., of $\Delta \gt g\subset\gt h$.  %%% ... 
In \cite{FFRe}, an astonishing discovery was made:  {\it there is a large  commutative algebra
$\eus C\subset \U(\gt h)^{\Delta\gt g}$ that contains all ${\mathcal H}_k$.} 
Our goal is to reveal a bi-Hamiltonian nature of $\gr\!(\eus C)\subset\gS(\gt h)$.

\section%{% 
{The loop algebra and quadratic elements} 
\label{sec-inf}

Set $\gZ(\wq,t^{-1})=\gS(t^{-1}\gt q[t^{-1}])^{\gt q[t]}$, where $\gS(t^{-1}\gt q[t^{-1}])$ is regarded as the 
quotient of $\gS(\gt q[t,t^{-1}])$ by the ideal $(\gt q[t])$. In the down to earth terms, $\gZ(\wq,t^{-1})$ consists of the elements  
$Y\in\gS(t^{-1}\gt q[t^{-1}])$ such that $\{xt^{k},Y\}\in \gt q[t]\gS(\gt q[t,t^{-1}])$ for each $x\in\gt q$ and each $k\ge 0$.  

\begin{prop} \label{P-com-inf}
We have $\{\gZ(\wq,t^{-1}),\gZ(\wq,t^{-1})\}=0$. 
\end{prop}
\begin{proof}
Suppose that $F_1,F_2\in \gZ(\wq,t^{-1})$. There is $N\ge 1$ such that 
$$F_1,F_2\in \gS(\gt q t^{-N}\oplus\gt q t^{1-N}\oplus\ldots\oplus\gt q t^{-1}).
$$ 
%%Since $, 
Set $n=2N$. Then $\{F_1,F_2\}=0$ if and only if we images $\bar F_1$, $\bar F_2$ of $F_1,F_2$ in 
$\gS(\gt q[t^{-1}])/(t^{-n}-1)$ Poisson-commute.  Next we identify $\overline{t^{1-n}}$ with $\bar t$. This leads to the 
identification 
$$
\gt q[t^{-1}]/(t^{-n}-1)\cong (\mW, [\,\,,\,]_{t^n-1}).$$  
By the construction, $\bar F_1,\bar F_2\in\cz_{t^n}\subset\gS(\mW)$. The brackets $\{\,\,,\,\}_{t^n-1}$ and
$\{\,\,,\,\}_{t^n}$ are compatible, furthermore   $\{\,\,,\,\}_{t^n}$ is a regular Poisson structure in the pencil
$L(t^n-1,t^n)$. Therefore  $\{\cz_{t^n},\cz_{t^n}\}_{t^n-1}=0$ and the result follows. 
\end{proof}

The evaluation at $t=1$ defines an isomorphism ${\sf Ev}_1\!: \gS(\gt q t^{-1})\to \gS(\gt q)$ of $\gt q$-modules. For $F\in\gS(\gt q)$, set $F[t^{-1}]:={\sf Ev}_1^{-1}(F)\in  \gS(\gt q t^{-1})$. 
If $F\in\gS(\gt q)^{\gt q}$, then  $F[t^{-1}]\in\gZ(\wq,t^{-1})$. Here is another well-known statement:  
$\partial_t$ acts on  $\gZ(\wq,t^{-1})$. Let us present a brief explanation. First of all, note that 
$\partial_t$ acts on the ideal $\gt q[t]\gS(\gt q[t,t^{-1}])$, since $\partial_t x=0$ if $x\in\gt q$.  
Suppose that $H\in\gZ(\wq,t^{-1})$. 
Then, for $a\ge 0$ and $x\in\gt q$, we have  
\begin{equation}\label{dt}
\{x t^a, \partial_t H\}=\partial_t(\{x t^a,H\})-a\{x t^{a-1},H\}=\partial_t 0=0
\end{equation}
in the quotient $\gS(\gt q[t,t^{-1}])/(\gt q[t])$. 

\begin{thm}\label{gen-gZ}
Suppose that $\gt q$ satisfies the assumptions of Theorem~\ref{thm:k-T}, in particular, 
$\mK[\q^*]^\q=\mK[F_1,\dots,F_m]$. Then $\gZ(\wq,t^{-1})$ is a polynomial ring generated by 
$\partial_t^k F_i[t^{-1}]$ with $1\le i\le m$ and $k\ge 0$.
%%% .......................
\end{thm}
\begin{proof}
Set $\mW_{-N}=\gt q t^{-N}\oplus\gt q t^{1-N}\oplus\ldots\oplus\gt q t^{-1}\subset \gt q[t,t^{-1}]/\gt q[t]$. Then 
$\gZ(\wq,t^{-1})=\varinjlim \gS(\mW_{-N})^{\q[t]}$ has a direct limit structure and
each  $\gS(\mW_{-N})^{\q[t]}\cong \gS(\q[t]/(t^N))^{\q[t]}$ is a polynomial ring by  Theorem~\ref{thm:k-T}. 
Furthermore, the explicit generators of $\gS(\q[t]/(t^N))^{\q[t]}$ given in \eqref{kT-gen}, see also 
\cite[Sect.~2]{k-T}, are images of $\frac{(-1)^k}{k!}\partial_t^k F_i[t^{-1}]$ with $1\le i\le m$ and $0\le k < N$ under the canonical isomorphism $\mW_{-N}\xrightarrow{\cong}\q[t]/(t^N)$ of $\q[t]$-modules. 
\end{proof}

For any $r\in\mK[t,t^{-1}]$, let $\gZ(\wq,r)\subset\gS(\gt q[t,t^{-1}])$ be the subalgebra obtained from  
$\gZ(\wq,t^{-1})$ by the substitute $t\mapsto r$. If $r\in\mK[t]$, then clearly $\gZ(\wq,r)\subset\gS(\gt q[t])$. 
Let $p\in\mK[t]$ be the same as in Section~\ref{q-cur}, let further  
$\psi_p\!: \gt q[t]\to\gt q[t]/(p)$ be the quotient map, which we extend to $\gS(\gt q[t])$. 
If $p(0)\ne 0$, then %%
 $\psi_p$ extends to $\gS(\gt q[t,t^{-1}])$. 

Let $\gzu\subset\gS(\gt q[t])$ be the subalgebra generated by the lowest $t$-components of
elements $F\in\gZ(\wq,t+1)$.  An alternative description is $\gzu=\lim_{\esi\to 0} \gZ(\wq,\esi t+1)$. It follows easily from any of these two descriptions that $\{\gzu,\gzu\}=0$. 

\begin{rmk} \label{rem-univ}
{\sf (i)} By the construction, $\gS(\q)^\q\subset\gzu$.  It is quite probable that 
$\gzu$ coincides with $\gZ(\wq,[0]):=\gS(\gt q[t])^{\gt q[t^{-1}]}$, where  $\q[t]$ is regarded as a quotient of
$\q[t,t^{-1}]$.  If $\q=\gt g$ is  semisimple, then the identity holds, see Proposition~\ref{red-gzu}.  \\[.3ex]
%%%%  (This 
{\sf (ii)} There are compatible Poisson brackets on $\gS(\gt q[t])$, the usual one $\{\,\,,\,\}$ and
$\{\,\,,\,\}_{(1)}$, for which $\{x t^a,y t^b\}_{(1)}=[x,y]t^{a+b+1}$ if $x,y\in\q$.
They help to understand $\gZ(\wg,[0])$ \cite[Sect.~4]{ir}.   %%% 
Applying the map $\psi_p$, we  obtain two compatible brackets on $\gS(\mW)$:
$\{\,\,,\,\}_p$ and  $\{\,\,,\,\}_{p,(1)}$, where $\{x\bar t^a,x\bar t^b\}_{p,(1)}=[x,y](t^{a+b+1}+(p))$ if $x,y\in\q$.
%%  .........  
However, for this pair, 
the special r\^ole of $\gt q{\cdot}1\subset \mW$ disappears and the  Poisson-commutative subalgebra 
arising from $\left<\{\,\,,\,\}_p,\{\,\,,\,\}_{p,(1)}\right>$ equals $\cz_p$.
\end{rmk}

Let us consider $\gZ(\wq,t)$, replace $F[t^{-1}]\in\gS(\gt q t^{-1})^{\gt q}$ with 
$F[t]\in\gS(\gt q t)^{\gt q}$ and $\partial_t$ with $\tau=t^2\partial_t$. Then 
$\tau^k(F[t])\in\gZ(\wq,t)$ for all $k\ge 0$.  

\begin{lm}\label{pol-gZ}
The subspace $\left<\psi_p(\tau^k(F[t]))\mid k\ge 0\right>$ lies in $\Pol(F)$ and 
\[
\dim\left<\psi_p(\tau^k(F[t]))\mid k\ge 0\right>\ge d(n{-}1)+1\]
 if $F\in\gS^d(\q)$ and $C:=p(0)\ne 0$.
\end{lm}
\begin{proof}
The task is to show that each $\psi_p(\tau^k(F[t]))$ is a linear combination of polarisations 
$F[\vec{k}]\in\gS^d(\mW)$. For example, $\psi_p(F[t])=F[(1,\ldots,1)]$ and $\psi_p(\tau(F[t]))=F[(1,\ldots,1,2)]$, if $n>2$. Furthermore, 
$\psi_p(\tau^2(F[t]))=2F[(1,\ldots,1,3)]+2F[(1,\ldots,1,2,2)]$, if $n>3$. The general formula is more complicated,
since it involves  the coefficients of $p$. However, %%we
if $Y=y_1\ldots y_d\in\gS^d(\q)$, then 
$\tau^k(Y[t])$ is a linear combination of elements 
\begin{equation}\label{t-pol}
Y[\vec{k},t]=\sum_{\sigma\in{\tt S}_d} y_1 t^{\sigma(k_1)}\ldots y_d  t^{\sigma(k_d)},
\end{equation}
where $\vec{k}=(k_1,\ldots,k_d)\in\Z^d$ and $1\le k_1\le k_2\le\ldots\le k_d$. 
Here $\psi_p(Y[\vec{k},t])\in\Pol(Y)$ and hence 
the inclusion $\left<\psi_p(\tau^k(F[t]))\mid k\ge 0\right>\subset\Pol(F)$ takes place. 

On $K[0,n{-}1]:=\{\vec{k}=(k_1,\ldots,k_d) \mid  0\le k_1\le k_2\le\ldots\le k_d<n\}$, we will define a total order. 
First  to each $\vec{k}\in K[0,n{-}1]$ we associate a partition $\underline{\boldsymbol{k}}\in K[1,n]$ by replacing each %% 
$k_i=0$ with $n$. For example, if $\vec{k}=(0,0,1,\ldots,1)$, then $\underline{\boldsymbol{k}}=(1,\ldots,1,n,n)$. 
Next we say 
 that 
$\vec{k} \prec \vec{k}'$ if $\underline{\boldsymbol{k}}\prec\underline{\boldsymbol{k}'}$ in the right lexicographic order. 
%% for  
Then the largest component of $\psi(\tau^k(F[t]))$   is equal to 
$k!C^d F[(0,\ldots,0)]$ if $k=d(n{-}1)$ and to $k!C^u F[(\underbrace{0,\ldots,0}_{u \text{ times}},\underbrace{1,\ldots,1}_{d-u-1},1{+}z)]$ if $k=u(n{-}1)+z$ with  $0\le z< n{-}1$ and $0\le u\le d{-}1$.
These components are clearly linearly independent and there are $d(n{-}1)+1$ of them.
Thus indeed  $\dim\left<\psi_p(\tau^k(F[t]))\mid k\ge 0\right>\ge d(n{-}1)+1$.
\end{proof}

\begin{conj} \label{conj}
%%S
For any normalised  $p\in\mK[t]$ of degree $n$, we have \\[.2ex] %%%% Then
{\sf (i)}  $\psi_p(\gZ(\wq,t))=\gZ(p,p+t)$ if $p(0)\ne 0$, \\[.2ex]
{\sf (ii)}   $\psi_p(\gZ(\wq,l))=\gZ(p,p+l)$ for any $l\in\mK[t]$ of degree $1$ that does not divide $p$, \\[.2ex]
{\sf (iii)} $\psi_p(\gzu)=\gZ(p,p+1)$.
%%% f. 
\end{conj}

\begin{prop} \label{123}
{\sf (1)} Part {\sf (ii)}  of Conjecture~\ref{conj} follows from {\sf (i)}. \\[.2ex]
{\sf (2)} If $\q$ satisfies the assumptions of Theorem~\ref{thm:k-T}, then 
%%also {\sf (iii)} follows from {\sf (i)}. 
$\lim\limits_{\esi\to 0} \gZ(p,p{+}\esi t{+}1)=\gZ(p,p+1)$ for any $p$.
\end{prop} 
\begin{proof}
{\sf (1)} We obtain the implication {\sf (i)} $\Rightarrow$ {\sf (ii)} by the change of variable $t\mapsto l$.   The polynomial $p$ from $\gZ(p,p+l)$ is replaced by $\tilde p$ such that $\tilde p(l)=p$ and the condition 
$l\,{\nmid}\,p$ becomes $\tilde p(0)\ne 0$. 

{\sf (2)} 
Clearly $\lim\limits_{\esi\to 0} \gZ(p,p{+}\esi t{+}1)\subset\gZ(p,p+1)$. 
If $\q$ satisfies the assumptions of Theorem~\ref{thm:k-T}, then $\lim_{\esi\to 0}\cz_{p+\ap(\esi t+1)}=\cz_{p+\ap}$ for any 
fixed $\ap\in\mK$, see Proposition~\ref{sym-lim}. Hence in this case $\lim\limits_{\esi\to 0} \gZ(p,p{+}\esi t{+}1)=\gZ(p,p+1)$. 
\end{proof}

In Section~\ref{sec-Ind}, 
we show that parts {\sf (i)} and {\sf (iii)} of Conjecture~\ref{conj} hold for reductive Lie algebras and, more generally,  
for Lie algebras
satisfying the assumptions of Theorem~\ref{thm:k-T}. %%%  in particular, it holds in the reductive case.
As a preparation to that argument, we need to work out some facts about invariants of degree two.

\subsection{Quadratic Lie algebras}\label{sec-qv}  
Suppose that $\gt q=\Lie Q$, %%  
where $Q$ is an algebraic group, 
and that $(\,\,,\,)$ is a non-degenerate $Q$-invariant scalar product on $\gt q$. 
Let $\{x_i\}$ be an orthonormal basis of $\gt q$, i.e., $(x_i,x_j)=\delta_{i,j}$. For 
$a,b\in \Z$, set $\bH[a,b]=\sum_{i=1}^{\dim\gt q} x_i t^a x_i t ^b\in\gS(\gt q[t,t^{-1}])$.
Clearly, each $\bH[a,b]$ is a $\gt q$-invariant.  Set further $\bH=\bH[1,1]$.

For $a,b,c\in\Z$, set $\cX[a,b,c]:=\sum_{i,j,\ups} ([x_i,x_j],x_\ups) x_i t^a x_j t^b x_\ups t^c\in\gS(\gt q[t,t^{-1}])$. 
We have seen such sums in Section~\ref{sb-GH}. Note that 
$\cX[a,b,c]=-\cX[b,a,c]=\cX[b,c,a]$. By a straightforward calculation, 
\begin{equation}\label{cX}
\{\bH[a,b],\bH[c,d]\}=\cX[b,d,a{+}c]+\cX[b,c,a{+}d]+\cX[a,d,b{+}c]+\cX[a,c,b{+}d].
\end{equation}
Whenever $a, b, c$ are pairwise distinct, we have $\cX[a,b,c]\ne 0$, since $[\gt q,\gt q]\ne 0$.

For $0\le a,b<n$, let $\bh[a,b]=\sum_{i=1}^{\dim\gt q} x_i \bar t^a x_i \bar t^b\in\gS(\mW)$ be the
image of $\bH[a,b]$ under the map $\psi_p$.  Set further $\bh=\bh[1,1]$ and 
$\bh[\bar r_k,\bar r_s]=\sum_{i=1}^{\dim\gt q} x_i \bar r_k x_i \bar r_s$ for $\bar r_k,\bar r_s\in\mK[t]/(p)$. 

Let 
$
\eP^{(\,,\,)}_{\bh,p}$ %%%=
be the Poisson centraliser of $\bh$ w.r.t. the bracket $\{\,\,,\,\}_p$ in the span of \[\{\bh[a,b]\mid 0\le  a,b<n\}.\]

\begin{lm}\label{claim}
We have $\eP^{(\,,\,)}_{\bh,t^n}= \gZ(t^n{+}t,t^n)\cap \left<\bh[a,b]\mid 0\le a,b<n\right>$ and $\dim\eP^{(\,,\,)}_{\bh,t^n}= 2n-1$, furthermore  $\dim\eP^{(\,,\,)}_{\bh,p}\le 2n-1$
for any $p$ of degree $n$.
\end{lm}
\begin{proof}
From computations presented in Section~\ref{sub-sym}, see in particular \eqref{tn-c0} and \eqref{tn-c}, one deduces that
$V(\bh):=\gZ(t^n{+}t,t^n)\cap \left<\bh[a,b]\mid 0\le a,b<n\right>$ is the span of 
\[
\{\bh[0,0], \ \bh[0,n{-}1], \ \bh, \ \sum_{1\le a\le b;\,a+b=k}\bh[a,b] \mid 3\le k\le 2n{-}2\}.
\] 
In particular, $\dim V(\bh)=2n{-}1$. Since the algebra $\gZ(t^n{+}t,t^n)$ is Poisson-commutative,  we have the 
inclusion $V(\bh)\subset  \eP^{(\,,\,)}_{\bh,t^n}$ and hence
$\dim \eP^{(\,,\,)}_{\bh,t^n}\ge 2n{-}1$. 

Next we show  
that $\dim \eP^{(\,,\,)}_{\bh,t^n}\le 2n{-}1$. %% 
Since $\deg_t \{\bh,F\}_{t^n}=\deg_t F+1$ for any $F$ in $\gS(\mW)$, the Poisson centraliser $\eP^{(\,,\,)}_{\bh,t^n}$
is a homogeneous in $t$ subspace. 
For a homogeneous in $t$ element $F\in\eP^{(\,,\,)}_{\bh,t^n}$, let $C\bh[a,b]$ be its nonzero summand  with $b\ge a$ and with the largest difference $b-a$.
%% Replacing 
By \eqref{cX}, $\{\bh,F\}_{t^n}$ is a linear combination of $\psi_{t^n}(\cX[1,\alpha,\beta+1])$
with $\alpha,\beta\ge 0$. The largest difference $(\beta+1)-\alpha$ that can occur here %% in %%% $\{\bh,F\}_{t^n}$ 
is $b+1-a$ and the corresponding summand of $\{\bh,F\}_{t^n}$ is  either $C\psi_{t^n}(\cX[1,a,b+1])$, if $a\ne b$, 
or $2C\psi_{t^n}(\cX[1,a,a+1])$, if $a=b$. Since 
 $\{\bh,F\}_{t^n}=0$, we must have $\psi_{t^n}(\cX[1,a,b+1])=0$. This implies   that either 
$b=n{-}1$ or $a=1$ or  $a=b=0$. Altogether we have 
$n+(n{-}2)+1=2n-1$ possibilities for $(a,b)$. The inequality is settled for $p=t^n$.  

With the help of the linear maps $\varphi_s$ with $s\in\mK^\times$, any Lie bracket $[\,\,,\,]_p$ can be contracted to 
$[\,\,,\,]_{t^n}$, see Lemma~\ref{contr0}. We have $\varphi_s(\bh[a,b])=s^{a+b}\bh[a,b]$ for any $a,b$. 
Thereby 
 $\eP^{(\,,\,)}_{\bh,\tilde p}=\varphi_s^{-1}(\eP^{(\,,\,)}_{\bh,p})$ if $\tilde p=s^n\varphi_s^{-1}(p)$ and 
 $\dim\eP^{(\,,\,)}_{\bh,\tilde p}=\dim \eP^{(\,,\,)}_{\bh,p}$ in this case. Furthermore, 
$\lim_{s\to 0}  \varphi_s^{-1}(\eP^{(\,,\,)}_{\bh,p})\subset \eP^{(\,,\,)}_{\bh,t^n}$ and thus
$\dim \eP^{(\,,\,)}_{\bh,p}\le \dim\eP^{(\,,\,)}_{\bh,t^n}=2n-1$.
\end{proof}

\begin{lm}\label{claim2}
Let $\eP^{(\,,\,)}_{\bh[0,1],p}$ stand for the Poisson centraliser of $\bh[0,1]$ w.r.t. the bracket $\{\,\,,\,\}_p$ in the 
span of $\{\bh[a,b]\mid 0\le  a,b<n\}$. Then 
%%% In 
$\dim\eP^{(\,,\,)}_{\bh[0,1],p}\le 2n-1$.
\end{lm}
\begin{proof}
Our argument here is very similar to the one used  in the case of $\bh$. %%Again, 
%%%We 
For  a homogeneous in $t$ element $F\in\eP^{(\,,\,)}_{\bh[0,1],t^n}$, let $C\bh[a,b]$ be its nonzero summand  with $b\ge a$ and with the largest difference $b-a$.
%% Replacing 
%% ........   ................. 
Then $\psi_{t^n}(\cX[0,a,b+1])=0$, which means that either 
$b=n{-}1$ or $a=0$. Altogether we have 
$n+(n{-}1)=2n-1$ possibilities for $(a,b)$.

The inequality 
$\dim \eP^{(\,,\,)}_{\bh[0,1],p}\le \dim\eP^{(\,,\,)}_{\bh[0,1],t^n}$ holds by the same reason as in the proof of Lemma~\ref{claim}. 
%% ...... ....... 
\end{proof}

\begin{lm} \label{com-H}
For any $\xi t^c\in\gt q[t,t^{-1}]$ with $\xi\in\gt q$, we have $$\{\bH[a,b],\xi t^c\}=\sum_{j,i}(\xi,[x_j,x_i]) (x_j t^{a+c}  x_i t^b + x_j t^{b+c} x_i t^a).$$
\end{lm}
\begin{proof}
By definition, $\{\bH[a,b],\xi t^c\}=\sum_{i} (\{x_i,\xi\}t^{a+c} x_i t^b+\{x_i,\xi\}t^{b+c} x_i t^a)$. Furthermore  
$$
\sum_{i} \{x_i,\xi\}t^{a+c} x_i t^b = \sum_{i,j} ([x_i,\xi],x_j)x_j t^{a+c} x_i t^b=\sum_{j,i} (\xi,[x_j,x_i]) x_j t^{a+c}x_it^b.
$$
Since the summand $\sum_i \{x_i,\xi\}t^{b+c} x_i t^a$ decomposes similarly, the result follows. 
\end{proof}

For $\xi\in\gt q$ and $\tilde a,\tilde b\ge 0$ set,
$$Y_{\xi}[\tilde a,\tilde b]=\sum_{j,i}(\xi,[x_j,x_i])  x_j  t^{\tilde a}  x_i t^{\tilde b} \in\gS(\gt q[t]).$$ 
Note that 
$Y_{\xi}[\tilde a,\tilde b]=-Y_{\xi}[\tilde b,\tilde a]$, in particular, $Y_{\xi}[\tilde a,\tilde a]=0$.
By Lemma~\ref{com-H}, we have 
\[\{\bH[\tilde a,\tilde b],\xi t\}=Y_\xi[\tilde a{+}1,\tilde b]+Y_{\xi}[\tilde b{+}1,\tilde a].\] 
Assuming %%
that $0\le a,b<n$, set $\bar Y_\xi[a,b]=\psi_p(Y_\xi[a,b])\in\gS(\mW)$. 
If $(\xi,[\gt q,\gt q])=0$, then each $Y_{\xi}[\tilde a,\tilde b]$ is zero. If  $(\xi,[\gt q,\gt q])\ne 0$, then 
$\{ \bar Y_\xi[a,b] \mid 0\le a<b<n\}$ is a set of linearly independent elements. %%%polynomials. 

\begin{lm} \label{lm-Y2}
Suppose $p=t^n-(c_{n-1}t^{n-1}+\ldots +c_1t+c_0)$. 
For $3\le k\le n$, set 
\begin{align*}
& X_k=\bh[k{-}1,2]+\bh[k{-}2,3]+\ldots+\bh[a{+}1,a] \ \text{  if } \ k=2a \ \text{ is even}; \\
& X_k=\bh[k{-}1,2]+\bh[k{-}2,3]+\ldots+\bh[a{+}1,a{-}1]+\frac{1}{2}\bh[a,a] \ \text{ if } \
k=2a-1 \ \text { is odd.} 
\end{align*} 
Then  $X:=c_0\bh[1,0]+\frac{1}{2} c_1\bh -\left(\sum_{k=3}^{n-1} c_k X_k\right)+ X_n$ 
belongs to $\cz_p$. 
\end{lm}
\begin{proof}
First we show that $\{X,\xi \bar t\}_p=0$ if $\xi\in\q$.   If $3\le k<n$, then $\{X_k,\xi \bar t\}_p=\bar Y_{\xi}[k,2]$.
Furthermore, $\{X_n,\xi \bar t\}_p=c_0 \bar Y_{\xi}[0,2]+c_1 \bar Y_{\xi}[1,2]+c_3 \bar Y_{\xi}[3,2]+\ldots+c_{n{-}1}\bar Y_{\xi}[n{-}1,2]$. 
Summing up,
\[
\begin{array}{l}
\{X,\xi \bar t\}_p=c_0 \bar Y_{\xi}[2,0]+c_1 \bar Y_{\xi}[2,1]-(c_3\bar Y_{\xi}[3,2]+\ldots+c_{n{-}1}\bar Y_{\xi}[n{-}1,2])+ \\ 
\enskip + c_0 \bar Y_{\xi}[0,2]+c_1 \bar Y_{\xi}[1,2]+c_3 \bar Y_{\xi}[3,2]+\ldots+c_{n{-}1}\bar Y_{\xi}[n{-}1,2]=0.
\end{array}
\]
%% Now 
Thus $\{\gt q\oplus\gt q\bar t,X\}_p=0$. If %%T
 $[\gt q,\gt q]=\gt q$, then %%
$\gt q \bar t^{k+1}=[\gt q\bar t,\gt q\bar t^k]_p$ as long as $k<n-1$. %\ and 
Hence 
we have  $X\in\cz_p$ in that particular case. %% Note that t
This covers semisimple Lie algebras. 

Suppose that $2\le k<n$ and $p$ are fixed. Then
%%In 
$\{\xi\bar t^k, X\}_p=\sum_{a<b} C_{a,b} \bar Y_\xi [a,b]$, where the coefficients $C_{a,b}\in\mK$ depend only on 
$a$ and $b$; they are independent of $\xi$ and $\q$.
In case $\gt q$ is semisimple and $\xi\ne 0$, we must have $C_{a,b}=0$ for all  $0\le a<b\le n-1$. Hence
$\{\xi\bar t^k, X\}_p=0$ for any $\gt q$ and any $\xi\in\gt q$. %%, 
\end{proof}

\begin{prop} \label{prop-qh11}
%%Suppose 
For any $p\in\mK[t]$ as above, we have $\bh\in\cz_p+\cz_{p+t}\subset\gZ(p,p+t)$.
\end{prop}
\begin{proof}
Write $p=t^n-(c_{n-1}t^{n-1}+\ldots +c_1t+c_0)$. 
Let $X\in\cz_p$ be the same as in Lemma~\ref{lm-Y2}. 
By a similar argument, $X-\frac{1}{2}\bh\in\cz_{p+t}$. Hence $\bh\in \cz_p+\cz_{p+t}$. 
\end{proof}

\subsection{Complete integrability of non-reductive Gaudin models} \label{sec-nrG}

Suppose  $p=\prod_{i=1}^n(x-a_i)$, where the roots $a_i$ are distinct and nonzero. Let $\bar r_k$ be the same as in 
\eqref{sum-r}. Then 
\[\bh=2\left(\sum_{k<i} a_k a_i\bh[\bar r_k,\bar r_i]\right) + \sum_{k} a_k^2\bh[\bar r_k,\bar r_k]\]
and here
$\bh[\bar r_k,\bar r_k]\in\cz(\gt q[t]/(p))$ for each $k$. Recall that we identify $\gt q[t]/(p)$ with $\gt q^{\oplus n}$.

Let ${\mathcal H}_k\in\gS(\gt q^{\oplus n})$ be defined by \eqref{eq-GH}  with $\gt q$ in place of $\gt g$, i.e., ${\mathcal H}_k=\sum_{j\ne k} \frac{1}{z_k-z_j} \bh[\bar r_k,\bar r_j]$. 
For distinct $i,j,k$, set  
\[\bar\cX[\bar r_i,\bar r_j,\bar r_k]:=\sum_{\iota,\nu,\ups} ([x_\iota,x_\nu],x_\ups) x_\iota \bar r_i x_\nu \bar r_j  x_\ups \bar r_k.
\]
 Since $\gt q\bar r_i$, $\gt q\bar r_j$, $\gt q\bar r_k$ pairwise commute, we can regard 
$\bar\cX[\bar r_i,\bar r_j,\bar r_k]$  as an element of either  
$\gS(\mW)$ or $\U(\mW,[\,\,,\,]_p)$. The same applies to all $\cH_i$. 
By the same calculation as we performed in Section~\ref{sb-GH}, 
\[
[{\mathcal H}_k,{\mathcal H}_i]=\sum_{j\ne k,i}\left(\frac{1}{(z_k-z_j)(z_i-z_j)}-\frac{1}{(z_k-z_i)(z_i-z_j)}-\frac{1}{(z_k-z_j)(z_i-z_k)}
\right)\bar\cX[\bar r_k,\bar r_i,\bar r_j]=0
\] 
for  $k\ne i$. 
Suppose now that   $z_i=a_i^{-1}$ for each $i$, then  
${\mathcal H}_k=\sum\limits_{j\ne k} \dfrac{-a_ja_k}{a_k-a_j} \bh[\bar r_k,\bar r_j]$. 
Furthermore, 
\begin{equation} \label{bhGo}
\bh=-2\left(\sum_k a_k {\mathcal H}_k\right)+\sum_{k} a_k^2\bh[\bar r_k,\bar r_k].
\end{equation} 

\begin{thm}
 {\sf (i)} We have $\dim\eP^{(\,,\,)}_{\bh,p} = 2n-1$ for any $p$ of degree $n$. \\[.2ex]
{\sf (ii)}  If  $p=\prod_{i=1}^n(x-a_i)$ with $a_i\ne a_j$ for $i\ne j$, \ $p(0)\ne 0$,  and the Hamiltonians $\cH_k$ are associated with 
$\vec{z}=(a_1^{-1},\ldots,a_n^{-1})$, then  $\cH_k\in\psi_p(\gZ(\wq,t))$ and 
$\cH_k\in\gZ(p,p+t)$ for each $k$. 
\end{thm}
\begin{proof}  By Proposition~\ref{prop-qh11}, $\bh\in\gZ(p,p+t)$ for any $p$ of degree $n$. 
Hence  \[\gZ(p,p+t)\cap\Pol(F)\subset \eP^{(\,,\,)}_{\bh,p} \ \ \text{ for } \ F=\sum_{i} x_i^2. \] 
By Proposition~\ref{bound-dim}, $\dim \gZ(p,p+t)\cap\Pol(F)\ge 2n-1$. This leads to %%the
$\dim \eP^{(\,,\,)}_{\bh,p}\ge 2n-1$. Now part {\sf (i)}  follows from Lemma~\ref{claim}.
We have also $\gZ(p,p+t)\cap\Pol(F)= \eP^{(\,,\,)}_{\bh,p}$. \\[.3ex]
{\sf (ii)} By the construction, $\bH\in \gZ(\wq,t)$. Hence $\bh\in \psi_p(\gZ(\wq,t))$ for any $p$. 
Set \[\hat V_p(\bh):=\psi_p(\gZ(\wq,t))\cap \left<\bh[a,b]\mid 0\le a,b<n\right>.
\]
We are assuming that $p(0)\ne 0$. Therefore 
 $\dim\hat V_p(\bh)\ge 2n-1$ by Lemma~\ref{pol-gZ}. Since $\hat V_p(\bh)\subset \eP^{(\,,\,)}_{\bh,p}$, we have 
$\hat V_p(\bh)=\eP^{(\,,\,)}_{\bh,p}$.  
From \eqref{bhGo} we deduce that $\cH_k\in  \eP^{(\,,\,)}_{\bh,p}$ for each $k$. Thus 
$\cH_k\in \gZ(p,p+t)$ and $\cH_k\in  \psi_p(\gZ(\wq,t))$. 
\end{proof}

For any $p$, the subspace 
$\hat V_p(\bh)$ contains the following linearly 
independent elements:
\begin{equation}\label{el}
\bh  \ \text{ and the sums } \  \psi_p(\tau^{k-2}(\bH))=(k-2)!\!\!\!\sum_{1\le a,b;\,a+b=k}\bh[a,b] \ \text{ with } \ 3\le k\le n.
\end{equation}
In case $p$ has  distinct roots, there are 
 $\bar r_i=r_i+(p)$ with $1\le i\le n$ defined in  Example~\ref{ex-diff},  %%  Then 
 and $\bh[\bar r_i,\bar r_i]\in \eP^{(\,,\,)}_{\bh,p}$ for each $i$, since $\bh[\bar r_i,\bar r_i]\in\cz_p$. 
Write $r_i=c_{i,n{-}1}t^{n-1}+\ldots+c_{i,1}t+c_{i,0}$ and assume that $p(0)\ne 0$. Then any 
$c_{i,0}=r_i(0)\ne 0$ and 
\begin{multline*}
\bh[\bar r_i,\bar r_i]=r_i(0)(c_{i,0}\bh[0,0]+2c_{i,1}\bh[0,1]+\ldots+2c_{i,n-1}\bh[0,n-1])+\\(\text{summands } \,C_{a,b}\bh[a,b] \text{ with } \,1\le a\le b). 
\end{multline*}
Since  the polynomials $\bar r_1,\ldots,\bar r_n$ are linearly independent, we conclude that 
the elements listed in \eqref{el} together with  $\bh[\bar r_i,\bar r_i]$, where $1\le i\le n$, form a basis 
of $\hat V_p(\bh)=\eP^{(\,,\,)}_{\bh,p}$.  Clearly 
$\{{\mathcal H}_1,\ldots,\cH_{n{-}1},\bh[\bar r_i,\bar r_i] \mid 1\le i\le n\}$ is another basis of 
$\eP^{(\,,\,)}_{\bh,p}$.
Our  discussion leads to the following statement.
\begin{itemize}
\item[$(\cG\bh)$] If $p$ has nonzero distinct roots, then the Gaudin model 
$(\q^{\oplus n}, [\,\,,\,], \cH_1,\ldots,\cH_{n})$ is   equivalent to 
$\mathscr{G}_p:=(\mW, [\,\,,\,]_p,\bh, \psi_p(\tau(\bH)),\ldots,\psi_p(\tau^{n{-}2}(\bH)))$. 
\end{itemize}

Next we would like to understand the quadratic part of $\gzu$. Since 
$\bH[1,1]$  and $\bH[1,2]$ are elements of $\gZ(\wq,t)$, we have 
\begin{multline} \label{gzu-qv}
\esi^2\bH[1,1]+2\esi\bH[0,1]+\bH[0,0]\in\gZ(\wq,\esi t+1) \ \text{ and} \\
\esi^3\bH[2,1]+\esi^2\bH[2,0]+2\esi^2\bH[1,1]+3\esi\bH[0,1]+\bH[0,0] \in\gZ(\wq,\esi t+1).
\end{multline}
Hence $\bH[0,1]\in \gzu$\,. This implies $\{\gzu,\bH[0,1]\}=0$. For $\dd\ge 2$, set 
$\bH^{[\dd]}=\sum\limits_{a,b\ge 1,\,a{+}b=\dd} \bH[a,b]$; then for $\dd\ge 0$, set  
$\widetilde{\bH}^{[\dd]}= \sum\limits_{a,b\ge 0,\,a{+}b=\dd} \bH[a,b]$. Let $\Psi_\esi\!:\q[t]\to \gt q[t]$ with $\esi\in\mK$ be defined
by $$
\Psi_{\esi}(xt^k)=x(\esi t+1)^k$$ 
for $x\in\q$ 
and $k\ge 0$. Set further $\Psi=\Psi_1$. We extend this maps to $\gS(\q[t])$. Finally let 
$\Psi(F)_\bullet$ stand for  the lowest $t$-component of $\Psi(F)$ with $F\in\gS(\q[t])$. In this terms,  
$\bH^{[2]}=\bH[1,1]$, \ $\bH^{[3]}=2\bH[1,2]$, and
$\Psi(\bH[1,1])_\bullet=\bH[0,0]$,  \ $\Psi(\bH[1,2]-\bH[1,1])_\bullet=\bH[0,1]$. 

\begin{lm} \label{P-H}
Let $\eP^{(\,,\,)}_{\bH[0,1]}$ be the Poisson centraliser of $\bH[0,1]$ in the 
span of $\{\bH[a,b]\mid   a,b\ge 0\}$. Then 
%%% In 
$\eP^{(\,,\,)}_{\bH[0,1]}=\left< \widetilde{\bH}^{[\dd]} \mid \dd\ge 0\right>$.
\end{lm}
\begin{proof}
Using~\eqref{cX}, one sees that 
$\{\bH[0,1],\bH[a,b]\}=\cX[a+1,0,b]+\cX[b+1,0,a]$ and checks
readily  that $\left< \widetilde{\bH}^{[\dd]} \mid \dd\ge 0\right>\subset\eP^{(\,,\,)}_{\bH[0,1]}$.

The subspace $\eP^{(\,,\,)}_{\bH[0,1]}$ is spanned by homogeneous in $t$ elements. 
Let $F\in \eP^{(\,,\,)}_{\bH[0,1]}$ be homogeneous in $t$. Without loss of generality we may assume that 
$\bH[a,b]$ with $b\ge a$ is a summand of $F$ with the largest $b$. Now 
$\cX[1+b,0,a]$ is a summand of $\{\bH[0,1],F\}$ with the largest difference $b+1-a$. Hence $\cX[1+b,0,a]=0$. %%%it has to be zero.
Since $b+1>a$ and $b+1>0$, this is possible if and only if $a=0$.  Thus $F\in \mK \widetilde{\bH}^{[\dd]}$ with 
$\dd=\deg_t F$. 
\end{proof}

\begin{lm} \label{gzu-des}
For any $\dd\ge 2$, there are $c_{\dd-1},\ldots,c_2\in\mK$ such that 
$\Psi(\bH^{[\dd]}-\sum\limits_{u=2}^{\dd-1} c_u \bH^{[u]})_\bullet =C\widetilde{\bH}^{[\dd-2]}$ with 
$C\in\mK^\times$.
\end{lm}
\begin{proof}
For $\dd=2$ and $\dd=3$, we have checked the statement, see~\eqref{gzu-qv}.  Therefore 
suppose  that $\dd\ge 4$. 
%%Then $$ \Psi(\bH^{[4]})=2\bH[1,3]+\bH[2,2]+2\bH[0,3]+10\bH[1,2]+10\bH[1,1]+8\bH[0,2]+12\bH[0,1]+3\bH[0,0]
% $$
% and $\Psi(\bH^{[4]}-6\bH[1,2]+3\bH[1,1])_\bullet=2\bH[0,2]+\bH[1,1]=\widetilde{\bH}^{[2]}$.
Suppose further %% now that $\dd\ge 5$ and 
that the statements holds for all $u$ such that $\dd>u\ge 2$. 
Then 
$$
\lim_{\esi\to 0} \left<\Psi_{\esi}(\bH^{[u]})\mid 2\le u<\dd\right>=\left<\widetilde{\bH}^{[u]} \mid 0\le u\le \dd-3\right>.
$$
Thereby there are $c_{\dd-1},\ldots,c_2\in\mK$ such that
$\deg_t\!\Upsilon\ge \dd-2$ for
 $\Upsilon\!:=\Psi(\bH^{[\dd]}-\sum\limits_{u=2}^{\dd-1} c_u \bH^{[u]})_\bullet$.  Our  goal is to prove that 
 $\deg_t\!\Upsilon= \dd-2$.
 
Set $k:=\deg_t\!\Upsilon$. Since $\{\bH[0,1],\Upsilon\}=0$, we have $\Upsilon\in\mK \widetilde{\bH}^{[k]}$
by Lemma~\ref{P-H}. %By the same lemma, 
Since $\Upsilon$ is nonzero, it must have a nonzero summand proportional to 
$\bH[0,k]$. Thereby $k\le \dd-1$. Assume that $k=\dd-1$. 

If $u>b>0$, then the coefficient of $\bH[0,b]$ in $\Psi(\bH^{[u]})$ is equal to
\begin{equation}\label{coef}
2\sum_{i=1}^{u-b} \binom{u-i}{b}=2\left(\underbrace{1+\binom{b+1}{b}}_{\binom{b+2}{b+1}}+\binom{b+2}{b}+\ldots+\binom{u-1}{b}\right)=2\binom{u}{b+1}.
 \end{equation}
The coefficient of $\bH[0,0]$ in $\Psi(\bH^{[u]})$ is equal to $u-1$. 
%%% Since $\Upsilon$ does not contain nonzero summands proportional to $\bH[0,u]$ with $u\le \dd-2$, 
Thereby 
we must have
\[
(-1,c_{\dd-1},c_{\dd-2},\ldots,c_2)A=0
\]
for 
\[
A=A_{\dd}=\begin{pmatrix}
\binom{\dd}{\dd-1}  & \binom{\dd}{\dd-2} & \binom{\dd}{\dd-3} & \ldots & \binom{\dd}{2} & \dd-1 \\  
1  & \binom{\dd-1}{\dd-2} & \binom{\dd-1}{\dd-3} & \ldots & \binom{\dd-1}{2} & \dd-2 \\
0 & 1 & \binom{\dd-2}{\dd-3} &   \ldots & \binom{\dd-2}{2} & \dd-3 \\
0 & 0 & 1 & \ldots  & \binom{\dd-3}{2} & \dd-4 \\
\vdots & . & . & . &.  & \vdots\\
0 & . & . & 1 & 3 & 2 \\
0 & . & . & 0 & 1 & 1 \\
\end{pmatrix}.
\]
Elementary manipulations with lines of $A$ show that $\det(A)=\det(A_{\dd-1})$. 
Since $A_4=\begin{pmatrix} 4 &  6 & 3  \\ 1& 3 & 2 \\ 0 & 1 & 1 \end{pmatrix}$, we have $\det(A_{\dd})=\begin{vmatrix}
4 & 3 \\ 1 & 1 
\end{vmatrix}=1$ for all $\dd\ge 4$. 
This is a contradiction and hence $\deg_t\!\Upsilon=\dd-2$. Since $\Upsilon\ne 0$, indeed 
$\Upsilon=C \widetilde{\bH}^{[\dd-2]}$ with $C\in\mK^\times$. %%% commutes with $\bH[0,1]$ and 
%%% has a summand 
%%% 
\begin{comment}
We have
\begin{multline*}
\Psi(\bH^{[\dd]})=2\dd\bH[0,\dd{-}2] +2\left({\binom{\dd-1}{2}}+
\dd{-}2 +1\right)  \bH[0,\dd{-}3]+ \\
+ 
2\left({\binom{\dd-1}{2}}+2(\dd{-}2)+3 \right)\bH[1,\dd{-}3]
%%2\bH[0,\dd{-}1] + 2((\dd-1)+2) \bH[1,\dd{-}2] +\ap((\dd-2)+3)\bH[2,\dd{-}3] + 
+(\,\text{some other terms with different
$\bH[a,b]$}\,), \\ 
\Psi(\bH^{[\dd{-}1]})=2\bH[0,\dd{-}2] + 2(\dd{-}1) \bH[0,\dd{-}3] +2(\dd{-}2+2)\bH[1,\dd{-}3]
%%2 \bH[1,\dd{-}2]+\ap\bH[2,\dd{-}3]  
+ \\
+(\,\text{some other terms with different
$\bH[a,b]$}\,), \text{ and } \\
\Psi(\bH^{[\dd{-}2]})=2\bH[0,\dd{-}3]+2\bH[1,\dd{-}3]+(\,\text{some other terms with different
$\bH[a,b]$}\,).
\end{multline*}
%%where $\ap=2$, if $\dd>5$, and $\ap=1$, if $\dd=5$.
%% Note that $\Psi(\bH^{[\dd-1]})$ has  no nonzero summand proportional to $\bH[0,\dd{-}1]$. 
Since $\Psi(\bH^{[u]})$ with $u<\dd{-}1$, has  no nonzero summand proportional to $\bH[0,\dd{-}2]$, 
we obtain $c_{\dd-1}=-\dd$. Since $\Psi(\bH^{[u]})$ with $u<\dd{-}2$, has  no summands of $t$-degree 
$\dd{-}2$, as well as no nonzero summand proportional to $\bH[0,\dd{-}3]$, there are the equalities:
$$
\left\{
\begin{array}{l}
\binom{\dd}{2}-\dd(\dd-1)+c_{\dd-2}=0, \\
\binom{\dd+1}{2}-\dd^2+c_{\dd-2}=0.
\end{array}
\right.
$$
But $\binom{\dd+1}{2}-\binom{\dd}{2}=\dd$ 
....
...
....
and 
$\Upsilon=2\bH[0,\dd{-}1]+2\bH[1,\dd{-}2]+\ap\bH[2,\dd{-}3] $ ......... {\it does not work!}
 
%% as well as none 
 %% of $t$-degree ,  
\end{comment} 
\end{proof}

\begin{thm} \label{Ham-gzu}
 If  $p=\prod_{i=1}^n(x-a_i)$ with $a_i\ne a_j$ for $i\ne j$, %\ $p(0)\ne 0$,  
 and the Hamiltonians $\tilde\cH_k$ are associated with 
$\vec{z}=(a_1,\ldots,a_n)$, then  $\tilde\cH_k\in\psi_p(\gzu)$ and $\tilde\cH_k\in\gZ(p,p+1)$ for each $k$. 
\end{thm}
\begin{proof}
Let $X$ be the same as in Lemma~\ref{lm-Y2}. Then $X\in\cz_p$ and $X-\bh[0,1]\in \cz_{p+1}$. Thereby 
$\bh[0,1]\in \gZ(p,p+1)$. Hence  \[\gZ(p,p+1)\cap\Pol(F)\subset \eP^{(\,,\,)}_{\bh[0,1],p} \ \ \text{ for } \ F=\sum_{i} x_i^2. \] 
By Proposition~\ref{bound-dim}, $\dim \gZ(p,p+1)\cap\Pol(F)\ge 2n-1$.
This leads to %%the
$\dim \eP^{(\,,\,)}_{\bh[0,1],p}=2n-1$ in view of Lemma~\ref{claim2}; and to the equality 
$\gZ(p,p+1)\cap\Pol(F)= \eP^{(\,,\,)}_{\bh[0,1],p}$. 
%%
%%We have
According to \eqref{gzu-qv}, $\bH[0,1]\in\gzu$\, and hence 
%%%  ....... \, and 
 $\bh[0,1]\in \psi_p(\gzu)$ for any $p$. 
Set \[\hat V_p(\bh[0,1]):=\psi_p(\gzu)\cap \left<\bh[a,b]\mid 0\le a,b<n\right>.
\]
 Using Lemma~\ref{gzu-des}, 
we show that $\psi_p(\gzu)$ contains the sums
\[
\left(\sum\limits_{0\le a,b\le n-1,\,a{+}b=k} \bh[a,b]\right)+(\,\text{terms with lower $\bar t$-degree}\,)
\] 
with $0\le k\le 2(n-1)$. 
% the assumption $p(0)\ne 0$,  and the \colorbox{red}{definition of $\gzu$}, we obtain the inequality 
Thereby  $\dim\hat V_p(\bh[0,1])\ge 2n{-}1$. Since $\hat V_p(\bh[0,1])\subset \eP^{(\,,\,)}_{\bh[0,1],p}$, this implies  
 the  identity 
$\hat V_p(\bh[0,1])=\eP^{(\,,\,)}_{\bh[0,1],p}$.  Next
\[
\bh[0,1]=\left(\sum_{i<j} (a_i+a_j)\bh[\bar r_i,\bar r_j]\right)+\sum_i  a_i \bh[\bar r_i,\bar r_i] =
\left(\sum_{i} a_i^2\tilde\cH_i\right)+\sum_i  a_i \bh[\bar r_i,\bar r_i].
\]
From this we deduce that $\tilde\cH_i\in  \eP^{(\,,\,)}_{\bh[0,1],p}$ for each $i$. Thus 
$\tilde\cH_i\in \gZ(p,p{+}1)$ and $\tilde\cH_i\in  \psi_p(\gzu)$. 
\end{proof}

Making use of the fact that  $\eP^{(\,,\,)}_{\bh[0,1],p}=\psi_p(\gzu)\cap\Pol(F)$, where $F=\sum_{i} x_i^2$, 
is a homogeneous in $\bar t$ subspace,
 it is easy to see that 
$F^{[\dd]}\in \eP^{(\,,\,)}_{\bh[0,1],p}$ if  $1\le \dd<n$. 
These elements are linearly independent. In case $\dd=1$, we obtain $F^{[1]}=2\bh[0,1]$. 
Suppose that $\xi\in\q$ and $[\xi,\q]\ne 0$. 
A straightforward calculation shows that  $\{F^{[\dd]},\xi\bar t\}_p=2\psi_p(Y_\xi[\dd{+}1,0])$. Thereby 
$\left<F^{[\dd]} \mid 1\le \dd<n\right>\cap \cz_p\ne\{0\}$
 only if $\psi_p(Y_\xi[n,0])$ does not have nonzero summands proportional to 
$\bar Y_\xi[1,0]$, i.e., only if $(\partial_t p)(0)=0$. 
  In all other cases, 
$\eP^{(\,,\,)}_{\bh[0,1],p}$ has a basis  
\[
\{ \bh[0,1], \ F^{[\dd]},  \ \bh[\bar r_i,\bar r_i] \mid 2\le \dd<n, 
\ 1\le i\le n\}. 
%%%%, 
\]
It also has a basis 
\[
\{\tilde\cH_1,\ldots,\tilde\cH_{n{-}1},  \bh[\bar r_1,\bar r_1], \ldots, \bh[\bar r_n,\bar r_n]\} 
\]
Thus the following statement is true. 

\begin{itemize}
\item[$(\cG\!\mathscr{v})$] If a polynomial $p$ has  distinct roots and $(\partial_t p)(0)\ne0$, then the 
 Hamiltonian system % 
$(\q^{\oplus n}, [\,\,,\,], \tilde\cH_1,{\ldots},\tilde\cH_{n})$  is   equivalent to 
$\mathscr{G\!v}_{\!p}:=\left(\mW, [\,\,,\,]_p,\bh[0,1], F^{[\dd]} \mid 2\le\dd\le n{-1} \right)$ with $F=\sum_{i} x_i^2$. 
\end{itemize}

%  $. 

\begin{thm} \label{integr}
Suppose that $\q$ has the  {\sl codim}--$2$ property and $\trdeg\cz(\q)=\ind\q$. Then
both systems, $\mathscr{G}_p$ and $\mathscr{G\!v}_{\!p}$, are completely integrable for any $p$.
\end{thm}
\begin{proof}
The Hamiltonians of $\mathscr{G}_p$ are contained in $\eP^{(\,,\,)}_{\bh,p}\subset \gZ(p,p+t)$. 
The Hamiltonians of $\mathscr{G\!v}_{\!p}$ are contained in $\eP^{(\,,\,)}_{\bh[0,1],p}\subset \gZ(p,p+1)$. 
Both algebras, $\gZ(p,p+t)$ and $\gZ(p,p+1)$, are Poisson-commutative. They consists of $\q$-invariants 
and have  transcendence degree $\bb(\q,n)$, see Theorem~\ref{trdeg}.  In view of \eqref{bound}, one 
may already say that the systems are ``completely integrable". 

Our conditions on $\q$ imply that there is a (reasonably nice) Poisson-commutative subalgebra $\ca\subset\cz(\q{\cdot}1)$ such that 
$\trdeg\ca=\bb(\q)$ \cite{bols}.  Then $\mathsf{alg}\langle \gZ,\ca\rangle$ with $\gZ=\gZ(p,p+t)$ or $\gZ=\gZ(p,p+1)$
is still Poisson-commutative. One can show that $\trdeg\mathsf{alg}\langle \gZ,\ca\rangle=n\bb(\q)$, cf. \cite[Sect.~6.2]{OY}.

The most classical notion of an integrable system requires the underlying space to be a symplectic manifold (or variety). 
In case of $\q^*$, one takes the coadjoint orbits. Since $\q=\Lie Q$ is an algebraic Lie algebra, each 
$(\mW,[\,\,,\,]_p)$ is algebraic as well, i.e.,   $(\mW,[\,\,,\,]_p)=\Lie Q^{\times n,p}$. If $p$ has distinct roots, then 
$Q^{\times n,p}\cong Q^n$. 

%% and %%
%%% that 
The restriction of $\langle \gZ,\ca\rangle$ to a generic coadjoint orbit 
$Q^{\times n,p}{\cdot}\gamma\subset\mW^*$ has trancendence degree $\frac{1}{2}\dim(Q^{\times n,p}{\cdot}\gamma)$,
see e.g. \cite[Lemma~1.2]{kruks}. This means that  (the restrictions of) $\mathscr{G}_p$ and $\mathscr{G\!v}_{\!p}$ are completely integrable 
on $Q^{\times n,p}{\cdot}\gamma$.
%%
%% 
%%% ................. 
\end{proof}

\begin{ex} \label{qv-dim2}
A Takiff Lie algebra $\qgk$ modelled on  reductive $\gt g$ is a quadratic Lie algebra and, as we already know,  it has 
all nice 
%%the {\sl codim}--$3$ 
properties. Another example is provided by centralisres $\gt q=\gt g_e$ of rectangular nilpotent elements $e\in\gt g$, where 
either 
\begin{itemize}
\item[$\diamond$]
$\gt g=\gt{sp}_{2m}$ and $e$ has $2a$ Jordan blocks of odd size $b$ (i.e., $ab=m$) or 
\item[$\diamond$] $\gt g=\gt{so}_{2m}$ and $e$  has $2a$ Jordan blocks of odd size $b$ (i.e., $ab=m$) or 
\item[$\diamond$] $\gt g=\gt{so}_{2m+1}$ and $e$  has $a$ Jordan blocks of odd size $b$ (i.e., $ab=2m+1$).
\end{itemize}
In all these cases, $\gt g_e$ is a quadratic Lie algebra and has the  {\sl codim}--$2$ property, furthermore, 
$\trdeg\gS(\gt g_e)^{\gt g_e}=\ind\gt g_e=\rk\gt g$. 
\end{ex}

\section{Gaudin subalgebras: the semisimple case} \label{GA}
Let $\gt g$ be  semisimple.  
The enveloping algebra $\U(\gt g[t])$ contains a large commutative subalgebra, the 
{\it Feigin--Frenkel centre} $\gt z(\wg)=\gt z(\wg,t)$, the {\it FF-centre} for short.   
Historically, the FF-centre  was constructed as a %%c
  subalgebra  $\gt z(\wg,t^{-1})\subset\U(t^{-1}\gt g[t^{-1}])$ \cite{ff}. 
But it is more convenient for us to switch the variable $t^{-1}\mapsto t$. 
Then the original differential operator $-\partial_t$ involved in the construction of  Feigin and Frenkel 
has to be replaced  with $\tau=t^2\partial_t$. 
A description of $\gr\!(\gt z(\wg,t^{-1}))$ is obtained in \cite{ff}, in our notation 
$\gr\!(\gt z(\wg))=\gZ(\wg,t)\subset\gS(t\gt g[t])$. %% ,
Because of the change of variable, a slight modification in the description of the 
{\it Gaudin subalgebra} introduced in~\cite{FFRe} is needed as well. 

Let $\Delta\U(\gt g[t])\cong\U(\gt g[t])$ be the diagonal of 
$\U(\gt g[t])^{\otimes n}$. Set $\gt h=\gt g^{\oplus n}$. For $x\in\gt g$ and $1\le i \le n$, let 
$x^{(i)}\in\gt h$ be a copy of $x$ belonging to the   $i$-th copy of $\gt g$. 
Any vector $\vec a=(a_1,\ldots,a_n)\in \mK^n$ defines 
a natural homomorphism $\rho_{\vec a}\!:\Delta\U(\gt g[t]) \to \U(\gt g)^{\otimes n}$, where 
\[
 \rho_{\vec a}(x t^k) = a_1^k x^{(1)}+a_2^k x^{(2)}+\ldots + a_n^k x^{(n)}\in\gt g\oplus\gt g\oplus\ldots\oplus\gt g \ \text{ for } \ x\in\gt g.
\] 
Let ${\mathcal G}={\mathcal G}(\vec a)$ be the image of $\gt z(\wg)$ under $\rho_{\vec a}$. 
If $a_j\ne 0$ for each $j$ and $a_j\ne a_k$ for $j\ne k$, then ${\mathcal G}$
contains the Hamiltonians 
${\mathcal H}_k$ associated with $\vec z=(a_1^{-1},\ldots,a_n^{-1})$ \cite{FFRe},
see  Section~\ref{sb-GH} for the Definition of ${\mathcal H}_k$. %%%  for all $k$. 
%% ad 
In this case,  $[{\mathcal G}(\vec a),{\mathcal H}_k]=0$ for each $k$. %% $ c
One also has ${\mathcal G}\subset %%
\U(\gt h)^{\Delta\gt g}$  for any $\vec a$ by the construction. 
Let us say that  ${\mathcal G}(\vec a)$ is a {\it Gaudin subalgebra (or algebra)} if $\vec a\in(\mK^\times)^n$ and 
$a_j\ne a_k$ for $j\ne k$.  The Feigin--Frenkel centre  $\gt z(\wg)$ is a homogeneous in $t$ algebra \cite{ff}. Thereby 
$\cG(\vec a)=\cG(c\vec a)$ for any $c\in\mK^\times$. 

Several  other facts about  ${\mathcal G}(\vec a)$  are known.
 For instance, 
${\mathcal G}(\vec a)$ is a polynomial ring with $\bb(\gt g,n)$ generators
and $\mathcal{ZU}(\gt g{\cdot}1)\subset {\mathcal G}(\vec a)$, %%if  $
see  
%%% 
\cite[Prop.~1]{G-07}.

\begin{prop}\label{G-quo}
Assuming that $a_j\ne a_k$ for $j\ne k$,
set $p=\prod_{j=1}^n (t-a_j)$ and identify $\gt h$ with $\gt g[t]/(p)\cong(\mW,[\,\,,\,]_p)$ using~\eqref{sum-r}. Then 
the Gaudin subalgebra ${\mathcal G}(\vec a)$ identifies with the image of $\gt z(\wg)$ in the quotient 
$\U(\gt g[t])/(p)\cong\U(\gt g[t]/(p))$.
\end{prop}
\begin{proof}
In the notation of~\eqref{sum-r}, the image of $x t^k$ in $\gt g[t]/(p)$ is 
equal to $$a_1^k x \bar r_1+ a_2^k  x\bar r_2 +\ldots+ a_n^k  x\bar r_n.$$
This is the image of $\rho_{\vec a}(xt^k)$ under the isomorphism $\gt h\to \gt g[t]/(p)$
with  $x^{(i)}\mapsto x \bar r_i$.  
\end{proof}

In the future, we regard $\gt g=\gt g{\cdot}1\subset\mW$ as the diagonal copy of $\gt g$ in $\gt h$. 
Let us state a few standard facts, valid for all reductive Lie algebras, $\ind\gt g=\rk\gt g$; $\cz(\gt g)=\mK[F_1,\ldots,F_m]$, where $m=\rk\gt g=\trdeg\cz(\gt g)$ and each $F_i$ is homogeneous; the subset  
$\Omega_{\gt g^*}$ defined in Theorem~\ref{thm:k-T} coincides with $\gt g^*_{\sf reg}$ and 
$\dim\gt g^*_{\sf sing}\le \dim\gt g-3$, see e.g. \cite{ko63}, \cite[Sect~2]{OY}, \cite{kruks}.
By \cite{ff}, $\gt z(\wg)$ is freely generated by elements $\tau^k(S_i)$, where  
$1\le i\le m$ and $k\ge 0$, such that $\gr\!(S_i)=F_i[t]$ for each $i$.  

Suppose that both $\vec a$ and $\vec b$ have nonzero pairwise distinct entries. 
Let $p$ and $\bar r_i=r_i+(p)$ be the same as in \eqref{sum-r}. Set 
$r=\sum_i b_i r_i$. Let $\gt z(\wg,r)\subset \U(\gt g[t])$ be the subalgebra obtained from 
$\gt z(\wg)$ by the substitute $t^k\mapsto r^k$. It is still a commutative subalgebra. 

\begin{prop}\label{G-quo-2}
The Gaudin algebra ${\mathcal G}(\vec b)$ identifies with the image of $\gt z(\wg,r)$ in the quotient 
$\U(\gt g[t]/(p))$. %% , whe
\end{prop}
\begin{proof}
By the construction, the image of $xr^k$ in $\gt h\cong \gt g[t]/(p)$ equals 
\[b_1^kx^{(1)}+b_2^kx^{(2)}+\ldots+b_n^kx^{(n)}=\rho_{\vec{b}}(xt^k)
\] 
for each $k\ge 0$.
\end{proof}

\begin{ex} 
Suppose  $n=2$, then  $\cG(a_1,a_2)=\cG(\frac{a_1}{1-a_1c},\frac{a_2}{1-a_2c})$ if 
$c$ belongs to $\mK\setminus\{a_1^{-1},a_2^{-1}\}$, see \cite[Prop.~1]{G-07}. 
Hence $\cG(\vec a)=\cG(\frac{2a_1a_2}{a_2-a_1},\frac{2a_1a_2}{a_1-a_2})=\cG(1,-1)$ if $a_1\ne a_2$ and $a_1,a_2\in\mK^\times$. Therefore $\psi_p(\gt z(\wg,r))=\psi_p(\gt z(\wg))$, whenever $p$ has distinct nonzero roots and $r=\alpha t+c$ (with $\alpha\in\mK^\times$, $c\in\mK$) does not divide $p$. 
\end{ex}

Set $\oG(\vec a)=\gr\!({\mathcal G}(\vec a))\subset\gS(\gt h)$. %%% 

\begin{prop}\label{G-psi}
Suppose that $p=\prod_i(x-a_i)$, where $a_i\ne a_j$ for $i\ne j$, and $p(0)\ne 0$. Then 
$\oG(\vec a)\subset\gS(\gt h)$ identifies with $\psi_p(\gZ(\wg,t))\subset\gS(\mW)$. %%  Under the  same assumptions on $p$, 
Furthermore, 
the associated graded algebra $\oG(\vec b)\subset\gS(\gt h)$ of the Gaudin algebra $\cG(\vec{b})$ identifies with 
$\psi_p(\gZ(\wg,r))$, where $\gZ(\wg,r)=\gr\!(\gt z(\wg,r))$ and 
$r=\sum_i b_i r_i$. 
\end{prop}
\begin{proof}
The maps $\gr\!$ and $\psi_p$ %%\colorbox{red}{
almost commute. If we have $\psi_p(\gr\!(\Xi))\ne 0$ for $\Xi\in\U(\gt g[t])$, then 
$\psi_p(\gr\!(\Xi))=\gr\!(\psi_p(\Xi))$. Thereby  $\psi_p(\gZ(\wg,t))$
is contained in 
$\oG(\vec a)$, up to our usual identification, % if $p=\prod_i(x-a_i)$, 
see Proposition~\ref{G-quo}.

Set $d_i=\deg F_i$. Suppose that $p(0)\ne 0$. Then 
\[\dim\left<\psi_p{\circ}\tau^k(F_i[t])\mid 0\le k\le d_i(n{-}1)\right>\ge d_i(n{-}1)+1,
\]
 see the proof of 
Lemma~\ref{pol-gZ}. It follows from results of Section~\ref{sec-Ind}, see Theorems~\ref{sovp-t},\,\ref{free}, that 
%%%we will show that 
$\psi_p(\gZ(\wg,t))$ is freely generated by $\psi_p{\circ}\tau^k(F_i[t])$ with $1\le i\le m$ and $0\le k\le d_i(n{-}1)$.
Furthermore, $\psi_p(\gZ(\wg,t))$ is a maximal (w.r.t. inclusion) Poisson-commutative 
subalgebra of $(\gS(\mW),\{\,\,,\,\}_p)^{\gt g}$ by Theorems~\ref{sovp-t},\,\ref{max}.
Since $\oG(\vec a)\subset\gS(\gt h)^{\gt g}$, 
we obtain $\psi_p(\gZ(\wg,t))=\oG(\vec a)$.
%% the algebras are equal. 

The inclusion $\psi_p(\gZ(\wg,r))\subset \oG(\vec b)$  follows from Proposition~\ref{G-quo-2}. The change of variable 
$t\mapsto r$ leads to an isomorphism $\psi_{\tilde p}(\gZ(\wg,t))\cong\psi_p(\gZ(\wg,r))$, where 
$\tilde p=\prod_{i=1}^n (t-b_i)$. Since $\tilde p(0)\ne 0$, 
%%
%%% r = b_i (mod x-a_i) 
the subalgebra $\psi_p(\gZ(\wg,r))$ is  a polynomial ring and its  Poincar{\'e} series coincides 
with that of $\psi_p(\gZ(\wg,t))=\oG(\vec a)$ and of 
any Gaudin algebra. Hence $\psi_p(\gZ(\wg,r))=\oG(\vec b)$.  
%%
%%% also a maximal Poisson-commutative 
%% subalgebra of $(\gS(\mW),\{\,\,,\,\}_p)^{\gt g}$
%% Thus the assertion about $\cG(\vec{b})$ is proven.
%%%
% ... {\bf why equality? ??????}. 
\end{proof}

As a corollary of Proposition~\ref{G-psi}, we obtain the following two facts, which are not entirely new,
%and hence
\begin{gather}
{\mathcal G}(\vec a) \,\text{ is freely generated by the elements } \,\psi_p{\circ}\tau^k(S_i), \ \text{s.t. } 1\le i\le m, \ 0\le k\le d_i(n{-}1);  \label{G-gen} \\
\cG(\vec a) \,\text{ is a maximal commutative subalgebra of } \,\U(\gt h)^{\gt g}, 
\end{gather}
cf. \cite[Prop.~1]{G-07}, \cite[Sect.~7]{fo}.

\begin{ex} \label{tn}
In case $p=t^n-1$ and $r=t^{n-1}$, we have $\psi_p(\gZ(\wg,r))=\gZ(p,t^n)\subset\gS(\mW)$, see \cite{fo}. 
\end{ex}

The image $\psi_p(\gZ(\wg,t))\subset\gS(\mW)$ %%% 
definitely depends on $p$. The statement of Conjecture~\ref{conj}{\sf(i)} implies that $\psi_{p+t}(\gZ(\wg,t))=\psi_p(\gZ(\wg,t))$. 
It is difficult to interpret this equality in terms of $\gt g^{\oplus n}$, since our fixed isomorphism 
$(\mW,[\,\,,\,]_p)\cong\gt g^{\oplus n}$  depends on $p$. 
In the enveloping algebra, we have 
$\psi_p(\gt z(\wg))\subset \U(\mW,[\,\,,\,]_p)$  and there is no sense in comparing 
$\psi_p(\gt z(\wg))$ with $\psi_{p+t}(\gt z(\wg))$. %%%%%%It is 
However, we may consider $\psi_p(\gt z(\wg,r))\subset \U(\mW,[\,\,,\,]_p)$ for any $r\in\mK[t]$.

\subsection{Subalgebra $\gzu\subset \gS(\gt g[t])$}
%%%
By a remarkable result of L.\,Rybnikov \cite{r:un}, $\gt z(\wg)$ is the centraliser in $\U(t\gt g[t])$ of 
$\boldsymbol H= \sum_{i=1}^{\dim \gt g} (x_i t)^2$ and $\gZ(\wg,t)$ is the Poisson centraliser 
of $\boldsymbol H$ in $\gS(t\gt g[t])$.  By \cite[Prop.~4.9.]{ir}, 
$\gZ(\wg,[0])=\gS(\gt g[t])^{\gt g[t^{-1}]}$  is the Poisson centraliser  of $\bH[0,1]$ in $\gS(\gt g[t])^{\gt g}$. This recent   result 
helps us to understand $\gzu$.

Let $\vec{k}=(k_1,\ldots,k_d)$ be a tuple of non-negative integers. %%$. 
For  $Y=y_1{\ldots} y_d\in\gS^d(\gt g)$%% and
, let
$Y[\vec{k},t]\in\gS(\gt g[t])$ be the $\vec{k}$-polasiration of $Y$ defined in the same way as in Section~\ref{s-pol}, but with 
$t$ in place of $\bar t$.   The notation extends to all elements of $\gS^d(\gt g)$. For 
$F\in\gS^d(\gt g)$,  set $\Pol(F,t)=\left<F[\vec{k},t] \mid \vec{k}\right>$.

\begin{prop}\label{red-gzu}
For a  semismple Lie algebra $\gt g$, 
%%In
we have $\gzu=\gZ(\wg,[0])$.
\end{prop}
\begin{proof}
Recall that $\gzu=\lim_{\esi\to 0} \gZ(\wg,\esi t+1)$. By~\eqref{gzu-qv},
%%In the proof of Theorem~\ref{Ham-gzu}, 
we have %%established that 
$\bH[0,1]\in \gzu$. Combining the facts that  
$\gzu\subset\gS(\gt g[t])^{\gt g}$ is Poisson-commutative and that $\gZ(\wg,[0])$  is the Poisson centraliser  of $\bH[0,1]$ in $\gS(\gt g[t])^{\gt g}$ \cite[Prop.~4.9.]{ir}, we obtain $\gzu\subset\gZ(\wg,[0])$. %%%, since 

Let 
$\calv\!: t \gt g[t]\to\gt g[t]$ be the linear map given by $\calv(x t^k)=x t^{k-1}$ for $x\in\q$ und $k\ge 1$. We extend it to a 
homomorphism from $\gS(t\gt g[t])$ to $\gS(\gt g[t])$. The algebra $\gZ(\wg,t)$ is generated by 
$\tau^k(F_j[t])$ with $1\le j\le m$ and $k\ge 0$; the algebra $\gZ(\wg,[0])$ is generated by 
$\calv{\circ}\tau^k(F_j[t])$ with $1\le j\le m$ and $k\ge 0$, cf. Theorem~\ref{gen-gZ}. 

Similar to the proof of Lemma~\ref{gzu-des}, we will show that for any $j$ and $k$, there are elements 
$c_{k-1},\ldots,c_0\in\mK$ such that 
\begin{equation} \label{psiC}
\Psi(\tau^k(F_j[t])-c_{k-1}\tau^{k-1}(F_j[t])-\ldots-c_0 F_j[t])_\bullet=C\calv{\circ}\tau^k(F_j[t])
\end{equation}
 with $C\in\mK^\times$. Let $j$ be fixed. It is  convenient to replace 
each $\tau^k(F_j[t])$ with $\bF_k=\frac{1}{k!}\tau^k(F_j[t])$. If $k=0$, then the statement is obvious, 
$\Psi(\bF_0)_\bullet=\calv(\bF_0)$. Therefore assume that $k\ge 1$ and that for all $u<k$ the statement is proven. 

For any $F\in\Pol(F_j,t)\cap \gZ(\wg,t)$, we have $\Psi(F)_\bullet\in\Pol(F_j,t)\cap \gZ(\wg,[0])$. 
Since the field $\mK$ is algebraically closed, $|\mK|=\infty$ and we can define an 
evaluation map ${\sf Ev}_{\vec c}\!: \gS(\gt g[t])\to \gS(\gt g)$, where $x t^\nu\mapsto xc_\nu$ 
if $x\in\gt g$ and $\nu\ge 0$, such that $c_u\ne c_i$ for $i\ne u$. Clearly 
${\sf Ev}_{\vec c}(\Pol(F_j,t))=\mK F_j$. 
It is also clear that for a nonzero $F\in\Pol(F_j,t)$, one can find a suitable vector $\vec c$ such that 
${\sf Ev}_{\vec c}(F)\ne 0$. Since the polynomials $F_1,\ldots,F_m$ are algebraically independent,
the existence of  ${\sf Ev}_{\vec c}$ leads to 
\begin{equation} \label{c}
\Pol(F_j,t)\cap\mK[\calv{\circ}\tau^k(F_i[t]) \mid i\ne j, \, k\ge 0]=\{0\}.
\end{equation}
The element $\Psi(F)_\bullet$ with $F\in\Pol(F_j,t)\cap \gZ(\wg,t)$  is a $t$-polarisation of $F_j$ and it is $t$-homo\-ge\-neous. Therefore 
$\Psi(F)_\bullet=C\calv{\circ}\tau^k(F_j[t])$ for some $k\ge 0$ and $C\in\mK$. 

%% As a graded commutative algebra $\gZ(\wg,[0])$ is isomorphic to $\gZ(\wg,t)$. 
%% Furthermore, both algebras are homogeneous in $t$ and, if $\gS_k(\gt g[t])\subset\gS(\gt g[t])$ stands for 
% the subspace generated by elements of $t$-degree $k$, then 
%% \[ \dim\left(\gZ(\wg,[0])\cap \gS_k^d(\gt g[t])\right)=\dim\left(\gZ(\wg,t)\cap \gS^d_{k{+}d}(\gt g[t])\right) \]
%% for all $d,k\ge 0$. 
%% Therefore 
%% $\gzu$ and $\gZ(\wg,[0])$ have  the same  Poincar{\'e} series, this implies  

By the inductive hypothesis, $\calv{\circ}\tau^u(F_j[t])\in\gzu$ for $u\le k-1$. Thereby 
there are $c_{i}\in\mK$ with $i\in\{0,\ldots,k-1\}$ such that 
$\deg_t\!\Upsilon\le k$ for $\Upsilon=\Psi(\bF_k-c_{k-1}\bF_{k-1}-\ldots -c_0 \bF_0)_\bullet$.  Our goal is to show that 
$\deg_t\!\Upsilon=k$. Assume that this is not the case.
Set $d=\deg F_j$. 
 The coefficient 
of $F_j[(b,0,\ldots,0),t]$ with $0<b \le u+1$ in $\bF_u$ is equal to 
\[
\sum_{i=0}^{u-b+1} \binom{b+i}{b} \binom{u-b-i+d-1}{d-2}=%%\sum_{i=0}^{u-b+1} \binom{b+i}{i} \binom{u-b-i+d-1}{u-b-i+1}=\ ??? \ 
\binom{u+d}{b+d-1}.
\]
The coefficient of $F_j=F_j[(0,\ldots,0),t]$ in $\bF_u$ is equal to 
$\binom{u+d-1}{d-1}$. 
Thereby we must have
\[
(-1,c_{k-1},c_{k-2},\ldots,c_0)A=0
\]
for 
\[
A=A_{k,d}=\begin{pmatrix}
\binom{k+d}{1}  & \binom{k+d}{2} & \binom{k+d}{3} & \ldots & \binom{k+d}{k} & \binom{k+d-1}{k} \\  
1  & \binom{k+d-1}{1} & \binom{k+d-1}{2} & \ldots & \binom{k+d-1}{k-1} &\binom{k+d-2}{k-1} \\
0 & 1 & \binom{k+d-2}{1} &   \ldots & \binom{k+d-2}{k-2} & \binom{k+d-3}{k-2} \\
0 & 0 & 1 & \ldots  & \binom{k+d-3}{k-3}& \binom{k+d-4}{k-3} \\
\vdots & . & . & . &.  & \vdots\\
0 & . & . & 1 & d{+}1 & d \\
0 & . & . & 0 & 1 & 1 \\
\end{pmatrix}.
\]
Subtracting the last column from the but-last one, we see that 
$\det(A_{k,d})=\det(A_{k-1,d+1})$. 
Since $A_{1,d'}=\begin{pmatrix}   d'{+}1 & d'  \\ 1&  1  \end{pmatrix}$ for any $d'\ge 1$, we have 
$\det(A_{k,d})=1$ for all $k\ge 1$ and $d\ge 1$. 
This is a contradiction and hence indeed $\deg_t\!\Upsilon=k$.  

%%%%
%% $\gzu =\gZ(\wg,[0])$. %%ey
%% Clearly $\dim\left(\gzu\cap\gS_0^d(\gt g[t])\right)=\dim\left(\gZ(\wg,t)\cap \gS_d^d(\gt g[t])\right)$. 

This proves that all $\calv{\circ}\tau^k(F_j[t])$ are elements of $\gzu$. Thus  
%%%Hence ......... the Poincar{\'e} series .........  of 
$\gzu=\gZ(\wg,[0])$.  %%% coincides with that of    $.  
\end{proof}

Consider now $\widetilde{\gzu}=\lim_{\esi\to 0} \gt z(\wg,\esi t+1)\subset \U(\gt g[t])^{\gt g}$.
This is clearly a commutative algebra. 
If we regard $\bH[0,1]$ as an element of $\U(\gt g[t])$, then $\bH[0,1]\in\widetilde{\gzu}$.
Thereby $\widetilde{\gzu}\subset \gt z(\wg,[0])$, where  $\gt z(\wg,[0])$ is a slight modification of the Feigin--Frenkel 
centre and the unique quantisation of $\gZ(\wg,[0])$, see \cite[Sect.~5]{ir}. Recall that 
$\gt z(\wg,t)=\mK[\tau^k(S_j)\mid 1\le j\le m, \ k\ge 0]$, where $\gr\!(S_j)=F_j[t]$
and each $S_{j,k}=\tau^k(S_j)$ is a homogeneous in $t$ element \cite{ff}.  %%In this terms
Quite similarly 
\[
\gt z(\wg,[0])=\mK[\widetilde{S}_{j,k} \mid 1\le j\le m, \ k\ge 0],
\]
 where $\gr\!(\widetilde{S}_{j,k})=\calv{\circ}\tau^k(F_j[t])$  and  each $\widetilde{S}_{j,k}$ is a homogeneous in $t$ element of $t$-degree $k$.  
Note that the maps $\Psi_\esi$ and $\Psi=\Psi_1$ extend from $\gt g[t]$ to $\U(\gt g[t])$. 

\begin{prop} \label{prop-gzu}
%% {\sf (i)} 
We have $\gr\!(\widetilde{\gzu})=\gzu$. %%\\[.2ex]
%%%%
%%% .......
%% By a \colorbox{yellow}{similar argument} one can prove that $\lim_{\esi\to 0} \gt z(\wg,\esi t+1)\subset \U(\gt g[t])^{\gt g}$
%%% is the .......  mentioned in .
%%% where  
\end{prop}
\begin{proof}
%%% {\sf (i)} 
As we have already observed, $\widetilde{\gzu}\subset\gt z(\wg,[0])$. Thereby 
$\gr\!(\widetilde{\gzu})\subset \gr\!(\gt z(\wg,[0]))=\gzu$. The task is to show that 
each $\calv{\circ}\tau^k(F_j[t])$ is an element of $\gr\!(\widetilde{\gzu})$. 

There is no harm in assuming that $\gt g$ is simple. We choose the numbering of  basic symmetric invariants 
of $\gt g$ in such a way that $\deg F_1\le\deg F_2\le\ldots\le\deg F_m$. Then $\deg F_1=2$ and $F_1=\sum_{i}x_i^2$ up to a nonzero scalar. %%% $ is proportional to the Killing form  
%% The centre of $\gt g$ is equal to $\{0\}$ and $ we must have 
Each  $\tau^k(F_1[t])=k!\bH^{[k+2]}$ can be regarded as an element of 
$\U(\gt g[t])$ and 
\[
\lim_{\esi\to 0}\left<\Psi_\esi(\bH^{[\dd]},\Psi_\esi(\bH^{[\dd-1]}),\ldots,\Psi_{\esi}(\bH^{[2]})\right> 
=\left<\widetilde{\bH}^{[u]} \mid 0\le u\le \dd-2\right>\subset \U(\gt g[t])
\]
for all $\dd\ge 2$, see Lemma~\ref{gzu-des}. 
This proves that $\calv{\circ}\tau^k(F_1[t])\in\gr\!(\widetilde{\gzu})$ for all $k\ge 0$. 
Suppose that the same holds for all $u<j$ and consider $F_j$. Set $d=\deg F_j$. 

Clearly $\gr\!(\Psi(S_j)_\bullet)=F_j=\calv{\circ}\tau^0(F_j[t])$.  %%%Since $\Psi(S_j)_\bullet\in\U_d(\gt g)$, we have 
%% Suppose that $\calv{\circ}\tau^u(F_j[t])\in\gr\!(\widetilde{\gzu})$ for all $u$ such that $k>u\ge 0$. 
Consider 
$$
\Upsilon=\Psi(\Xi)_\bullet \ \text{ with } \ \Xi=\tau^{k}(S_j)-c_{k-1}\tau^{k-1}(S_j)-\ldots -c_1\tau(S_j)-c_0 S_j,
$$ 
where $c_{k-1},\ldots,c_0\in\mK$. 
We have $\Upsilon\in\U_d(\gt g[t])\cap \gt z(\wg,[0])$.
Assume that $\deg\gr\!(\Upsilon)<d$. Then we can write $\gr\!(\Upsilon)\in\gZ(\wg,[0])$ as a polynomial in elements 
$\calv{\circ}\tau^\nu (F_i[t])$ with $i<j$ and $\nu\ge 0$. Since these elements belong to 
$\gr\!(\widetilde{\gzu})$, there is $S\in \U_{d-1}(\gt g[t])\cap\widetilde{\gzu}$
such that  $\deg_t S\le d+k$ and 
\[
\deg_t\!\Upsilon'>\deg_t\!\Upsilon \ \ \text{ or } \ \ \deg\gr\!(\Upsilon')<\deg\gr\!(\Upsilon) \ \ \text{ for } \ \ 
\Upsilon'=\Psi(\Xi-S)_\bullet.
\]
 If $\deg\gr\!(\Upsilon')<d$, we modify 
$S$ and produce a new element $\Upsilon'$, decreasing the degree of $\gr\!(\Upsilon')$ or increasing the $t$-degree
of $\Upsilon'$.  On the one hand, the degree of $\gr\!(\Upsilon')$ cannot decrease infinitely. On the other hand,
always $\deg_t\!\Upsilon'\le d+k$. %%\deg_t\!\Psi(\tau^k(F_j[t])-\sum_{i=0}^{k-1}c_i\tau^i(F_j[t])_\bullet$ and if
%%$\deg_t\!\Upsilon' = \deg_t\!\Psi(\tau^k(F_j[t])-\sum_{i=0}^{k-1}c_i\tau^i(F_j[t])_\bullet$, then
%%  $\deg\gr\!(\Upsilon')=d$. 
 
 Thus, there is $S\in \U_{d-1}(\gt g[t])\cap\widetilde{\gzu}$ such that $\deg\gr\!(\Upsilon')=d$
%%% ........ reduce to zero, which is a contradiction. 
%%
%%%Therefore $\deg\gr\!(\Upsilon)=d$ 
and hence 
$$\gr\!(\Upsilon')=\Psi(\tau^k(F_j[t])-\sum_{i=0}^{k-1}c_i\tau^i(F_j[t])_\bullet\,.
$$
For a suitable choice of $c_0,\ldots,c_{k-1}$, the lowest $t$-component 
$\Psi\big(\tau^k(F_j[t])-\sum_{i=0}^{k-1}c_i\tau^i(F_j[t])\big)_\bullet$ is equal to $\calv{\cdot}\tau^k(F_j[t])$ up to a nonzero scalar,
see the proof of  Proposition~\ref{red-gzu} and in particular \eqref{psiC}. Hence  $\calv{\cdot}\tau^k(F_j[t])$ belongs to $\gr\!(\widetilde{\gzu})$. 
The result follows by induction on $j$. 
%%%%......... almost finished  ........ 
%%
\end{proof}

\subsection{Computations with Poisson brackets}\label{sec-bra}
A direct generalisations to ${\mathcal G}(\vec a)$ and $\oG(\vec a)$
%% $\U(\gt h)$ and $\gS(\gt h)$
 of the description as a centraliser of a single element 
%% above statements 
 is not possible. The quadratic elements $\bh=\sum_{i=1}^{\dim\gt g} (x_i  \bar t)^2$ 
and $\bh[0,1]$ commute with 
$\gt g{\cdot}1\subset\gt h$. %%
Thereby we will consider the Poisson centraliser of  $\bh$ in the subspace spanned by the 
polarisations of basic $\gt g$-invariants.

Suppose that $\hat Y=\hat y_1\ldots \hat y_d\in\gS^d(\gt g[t])$ and  $\hat y_j=y_j t^{k_j}$ with $y_j\in\gt g$.
Then 
\[
P_{\hat Y}:= \{ {\bH}, \hat Y\} =2 \sum_{j=1,i=1}^{j=d,i=\dim\gt g}  [x_i t,\hat y_j] x_i t \frac{\hat Y}{\hat y_j} = 
 2 \sum_{j,i,u} \kappa([x_u,x_i],y_j)  x_u t^{1{+}k_j} x_i t \frac{\hat Y}{\hat y_j}\,. 
\]
Recall that  $\{\gt g,\bH\}=0$. Hence %%a Hence
\begin{equation} \label{P2}
P_{\hat Y}=2\sum_{i,u;\,j:\,k_j\ne 0} \kappa([x_u,x_i],y_j)   x_u t^{1{+}k_j} x_i t  \frac{\hat Y}{\hat y_j}\,. 
\end{equation}
The  Killing form $\kappa$  extends to a non-degenerate $\gt g$-invariant  scalar product $(\,\,,\,)$ on %%  
%%% 
$\gS(\gt g[t])$. We will assume that $(\gt gt^a,\gt gt^b)=0$ for $a\ne b$, that 
$(xt^a,yt^a)=\kappa(x,y)$ for $x,y\in\gt g$,  and 
that 
$$
(\xi_1\ldots\xi_{\boldsymbol k},\eta_1\ldots\eta_{\boldsymbol d})=  \delta_{{\boldsymbol k},{\boldsymbol d}}
\sum_{\sigma\in{\tt S_{\boldsymbol k}}} \kappa(\xi_1,\eta_{\sigma(1)})\ldots \kappa(\xi_{\boldsymbol k},\eta_{\sigma({\boldsymbol k})})
$$
if $\xi_j,\eta_j\in\gS(\gt g[t])$, ${\boldsymbol d}\ge {\boldsymbol k}$. 
Let ${\mathcal B}$ be a monomial basis of $\gS(\gt g[t])$ consisting of the elements 
$\hat v_1\ldots \hat v_{\boldsymbol k}$, where  $\hat v_j=v_j t^{\nu_j}$ and $v_j\in\{x_i\}$. 
Then ${\mathcal B}$ 
%% Note,  
is an orthogonal, but not an orthonormal basis. 
For instance, if $\Xi=x_1^{\gamma_1}\!\ldots x_{\boldsymbol k}^{\gamma_{\boldsymbol k}}$, then $(\Xi,\Xi)=\gamma_1!\ldots\gamma_{\boldsymbol k}!$. Next we present several constructions from \cite[Sect.~3.3]{sym}. 
%% For $k\ge 1$, set 

Let  $\vec{\alpha}=(\alpha_0,\ldots,\alpha_{M})$ be a tuple of non-negative integers such that 
%%% 
$\sum_i \alpha_i=d{+}1$. 
%% $
Set $\gS^{\vec{\alpha}}(\gt g[t]):=\prod_{j=0}^{M} \gS^{\ap_j}(\gt g t^j)\subset \gS(\gt g[t])$ and  
${\mathcal B}(\vec{\alpha}):={\mathcal B}\cap \gS^{\vec{\alpha}}(\gt g[t])$. 
Fix  different $i, j\in\{0,\ldots,M\}$ such that $\alpha_j,\alpha_i\ne 0$. %% 
Assume that a monomial 
$\mathbb V=\hat v_1\ldots \hat v_{d{+}1}\in{\mathcal B}(\vec{\alpha})$ with %%
$\hat v_k=v_k t^{\nu_k}$ is written in such a way that $\nu_k=i$ 
for   $1\le k\le \ap_i$ and $\nu_k=j$ for  $\ap_i<k\le \ap_i+\ap_j$.  
Finally suppose that $F\in\gS^d(\gt g)$. 
In this notation, set
\begin{align} 
&\mathscr{F}[\vec{\alpha},i, j] :=\sum_{{\mathbb V}\in{\mathcal B}(\vec{\alpha})} A(\mathbb V) \mathbb V     
\ \ \text{ with }  \nonumber  \\
 &  
\qquad \ A(\hat v_1\ldots \hat v_{d{+}1})= \big(\mathbb V,\mathbb V\big)^{-1}\!\!\!\!\!\!\!\!\!\!\!\!\!\!\!\!\sum_{{\footnotesize \begin{array}{c}1\le k\le \ap_i, \\ \ap_i< q \le \ap_i+\ap_j \end{array}}}\!\!\!\!\!\!\!\!\!\!\!\!(F,[v_k,v_q] \prod_{u\ne q,k} v_u)\in\mK. 
 \label{W} 
\end{align} 
If  $\alpha_i=0$ or $\alpha_j=0$, then we set $\mathscr{F}[\vec\ap,i,j]=0$. 
Clearly $\mathscr{F}[\vec\ap,i,j]=-\mathscr{F}[\vec\ap,j,i]$. 
Certain Poisson brackets $\{\bH[a,b],{\hat Y}\}$ can be expressed in terms of $\mathscr{F}[\vec\ap,i,j]$.

\begin{prop}[cf. {\cite[Prop.~3.10]{sym}}] \label{ff-br}
%%The 
If\/ $Y=y_1{\ldots} y_d\in\gS^d(\gt g)$, then 
$$
\frac{-1}{2}%
P_{Y[\vec{k},t]} =\frac{1}{2} %
\{Y[\vec{k},t],\bH\}
$$ 
equals  
the sum of \,$\mathscr{Y}[\vec{\alpha},1, j]$ over all tuples $\vec{\alpha}$ as above and  $2\le j\le M$
with $\ap_j\ne 0$ such that 
the multi-sets $\{0^{\alpha_0},1^{\alpha_1},{\ldots},(j{-}1)^{\ap_{j-1}+1}, j^{\ap_j-1},{\ldots}, M^{\alpha_{M}}\}$ and 
$\{1,k_1,\ldots,k_{d}\}$ coinside. \qed
\end{prop}
%%, 

Let $F\in\gS^d(\gt g)^{\gt g}$ and $i$ be fixed. %% 
By \cite[Prop.~3.9]{sym}, we have 
\begin{equation}\label{univ}
\sum\limits_{j\ne i} \mathscr{F}[\vec\ap,i,j]=0
\end{equation}
for each $\vec\ap$. Set $\tilde M(\vec\ap):=|\{j\mid 0\le j\le M, \ \ap_j\ne 0\}|-1$. 

\begin{lm}\label{lin} Let us fix  $\vec{\ap}$ and $0\le i\le M$ such that $\ap_i\ne 0$. %%% be 
Suppose that $\gt g=\gt{sl}_d$, where $d+1=\sum_{i=0}^M \ap_i$, and that $F(\xi)=\det(\xi)$ for $\xi\in\gt g^*\cong\gt{sl}_d$. 
Then 
%%%
$
\dim\left<\mathscr{F}[\vec\ap,i,j] \mid j\ne i\right>=\tilde M(\vec\ap)-1$.
\end{lm}
\begin{proof}
If $\ap_k=0$, then $\mathscr{F}[\vec\ap,i,k]=0$, if $j=i$, then also $\mathscr{F}[\vec\ap,i,j]=0$. Thereby
we have at most $\tilde M=\tilde M(\vec\ap)$ nonzero Elements $\mathscr{F}[\vec\ap,i,j]$.
They satisfy \eqref{univ}, because $F\in\gS(\gt g)^{\gt g}$. Thus $\dim\left<\mathscr{F}[\vec\ap,i,j] \mid j\ne i\right>\le\tilde M-1$.  We need the equality 
here. If $\tilde M=1$, there is nothing to prove. Therefore suppose that $\tilde M\ge 2$. Then also $d\ge 2$. 

Assume that $\dim\left<\mathscr{F}[\vec\ap,i,j] \mid j\ne i\right> <\tilde M-1$. Then there are 
$j\ne k$ such that $j,k\ne i$, \  $\ap_k,\ap_j\ne 0$ and a relation $\sum_{q=0}^M C_q \mathscr{F}[\vec\ap,i,q]=0$ with  
$C_j=1$, $C_k=0$.

Now we work in the basis $\{E_{\nu\ups}, E_{\nu\nu}-E_{(\nu+1)(\nu+1)}\}$ of $\gt{sl}_d$ and
 pick out a suitable summand of  $\mathscr{F}[\vec\ap,i,j]$.  %%
To this end we construct a certain monomial $\hat Y\in\gS^{\vec\ap}(\gt g[t])$. 
Set $\hat y_1=(E_{11}{-}E_{22})t^i$, $\hat y_2=E_{12} t^j$, and $\hat y_3=E_{21}t^k$,  
let the factors $\hat y_u$ with $3<u \le d{+}1$ be elements of $\gt h[t]$, 
where $\gt h=\left<E_{11}-E_{uu} \mid 2\le u\le d\right>$. Assume further that 
$(E_{12}E_{21}y_4\ldots y_{d{+}1},F)\ne 0$. %%%%  $ .  
A suitable choice is $y_{u+1}=E_{(u-1)(u-1)}-E_{uu}$ for $u\ge 3$. 
The monomial 
$\hat Y=\hat y_1\ldots \hat y_{d+1}$ appears with a non-zero coefficient  in $\mathscr{F}[\vec\ap,i,j]$ and in
 $\mathscr{F}[\vec\ap,i,k]$, but in no $\mathscr{F}[\vec\ap,i,q]$ with $q\ne j,k$. Thus $C_j=0$. This contradiction finishes the proof.
 %%%    
\end{proof} 

It is quite probable that Lemma~\ref{lin} holds for all invariants 
$F\in\gS^d(\gt g)$, also for $\gt g\ne\gt{sl}_N$.
%%%%% 

Suppose now that the entries of $\vec{k}$ satisfy  $ 0\le k_i \le n{-}1$. %%
In Section~\ref{s-pol}, 
we have defined 
$\Pol(F)=\left<F[\vec{k}] \mid  0\le k_i \le n{-}1 \right>\subset\gS(\mW)^{\gt q}$ for $\gt q$. Now we  
use this object in case $\gt q=\gt g$. It is more convenient to rewrite $\vec{k}$ in a multi-set form $\vec{k}=(0^{\ups_0},1^{\ups_1},\ldots,(n{-}1)^{\ups_{n{-}1}})$ and set 
$\vec{\ups}(\vec{k})=\vec{\ups}=(\ups_0,\ldots\ups_{n{-}1})$. %% r 
On $\{\vec{\ups}\}$, we use the left lexicographic order. % 

Recall that $\ell_0=[\,\,,\,]_{t^n}$ is a Lie bracket on $\mW$. We use the same symbol for the Poisson 
bracket $\{\,\,,\,\}_{t^n}$ on $\gS(\mW)$. 

\begin{thm} \label{dim-F}
Let $F\in\gS^d(\gt g)$ be such that 
$\dim\left<\mathscr{F}[\vec\ap,i,j] \mid i\,\text{ is fixed and } \,j\ne i\right>=\tilde M(\vec\ap)-1$ whenever $\ap_i\ne 0$. Then 
%%% 
$\dim\eP \le     (n-1)d+1$ for $\eP=\{  f\in \mathrm{Pol}(F) \mid \ell_0(f,\bh)=0 \}$.
\end{thm} 
\begin{proof}
Each element $\bar f\in\eP$ is a linear combination 
$\sum_{\vec{k}} C_{\vec{k}} F[\vec{k}]$ with $C_{\vec{k}}\in\mK$. 
Replacing each summand $C_{\vec{k}} F[\vec{k}]$ with $C_{\vec k}F[\vec k,t]$, we obtain
$f\in\gS(\gt g[t])$ with the property  $\bar f=f+(p)$. 
Let $\bar f_{\diamond}$ be the minimal w.r.t. the left lexicographic order on $\{\vec{\ups}\}$ nonzero component of %%a 
$\bar f\in \eP$. Without loss of generality assume 
that $\bar f_{\diamond}=F[\vec{k}']$ for some $\vec{k}'$. The task is to prove that there are at most $(n{-}1)d+1$ possibilities for such 
$\vec{k}'$.

By Proposition~\ref{ff-br},  $\frac{1}{2}\{F[\vec k,t],\bH\}$ is a sum of $\mathscr{F}[\vec\ap,1,j]$ over 
some $\vec\ap$ depending on $\vec k$ and over $j\ge 2$ with $\ap_j\ne 0$. 
If $j$ is fixed and $\vec k=(0^{\ups_0},1^{\ups_1},\ldots,(n{-}1)^{\ups_{n{-}1}})$, then 
$\ap_1=\ups_1+1$, $\ap_{j-1}=\ups_{j-1}-1$, $\ap_j=\ups_j+1$, and $\ap_i=\ups_i$ for all other $i$.
%%%%  
If $\vec\ups(\vec k)>\vec\ups(\vec{k'})$ and $j\ge j'$, then $\vec\ap>\vec\ap'$ for $\vec\ap', \vec\ap$ obtained by the above recipe from $(\vec{k}',j')$ and $(\vec k,j)$, respectively. %% 

Let us consider $\vec{\ups}=\vec{\ups}(\vec{k}')$. Suppose first that $\ups_0 > 0$. 
In case $\ups_0=d$, we have an element of $\gS(\gt g{\cdot}1)$, which Poisson-commutes with $\bh$. This is 
a possibility for $\bar f_\diamond$.  If $\ups_0<d$, then there is $j\ge 2$ such that 
$2 \mathscr{F}[\vec\ap,1,j]$ is a summand of $\{F[\vec{k}',t],\bH\}$. 
The term $\mathscr{F}[\vec\ap,1,j]$ cannot be produced by any other component of $f$, because 
$\vec\ap$ and $j$ define $\vec k$ uniquely. 
We have $\ap_0=\ups_0\ne 0$ and the term $\mathscr{F}[\vec\ap,1,0]$ does not appear in 
$\{f,\bH\}$. Our assumptions on $F$ imply that 
there is no way to annihilate $2 \mathscr{F}[\vec\ap,1,j]$ in $\{f,\bH\}$, thereby 
$\psi_{t^n}(\mathscr{F}[\vec\ap,1,j])=0$. This is possible only if $j=n$. Therefore we must have  
%%%
%%............ 
%%. ......... if 
$\ups_{n{-}1}=d-\ups_0$. %%... 
Altogether in case $\ups_0>0$, we obtain  $d$ possibilities for $\vec{k}'$.   
This finishes the proof for $n=2$. From now on, assume that $n\ge 3$. 

Suppose now that $\ups_0=0$. If $\ups_1\ne 0$, then $\ell_0(\bar f_\diamond,\bh)$ has a summand 
$2\psi_{t^n}(\mathscr{F}[\vec\ap,1,2])$ with $\vec\ap=(0,\ups_1,\ups_2+1,\ups_3,\ldots,\ups_{n{-}1})$. 
No other component of $f$ can produce this $\vec\ap$, since the smallest possible value of
$j$ is $2$. Thus $\mathscr{F}[\vec\ap,1,2]$ must be zero, which means that $\vec\ap$ has only two nonzero components.
These components are $\ap_1$ and $\ap_2$. Thereby $\ups_2=d-\ups_1$. This gives us another $d$ 
possibilities for $\bar f_\diamond$. 

Suppose next that $\ups_0=\ldots=\ups_{i-1}=0$  for $i >1$ and $\ups_{i}\ne 0$. If
$i=n-1$, then we have just one possibility for $\bar f_\diamond$. %%  
Assume that  $i \ne n-1$. Then 
 $\ell_0(\bar f_\diamond,\bh)$ has a summand 
$2\psi_{t^n}(\mathscr{F}[\vec\ap,1,i{+}1])$ with $\vec\ap=(0,1,0,\ldots,0,\ups_i-1,\ups_{i+1}+1,\ups_{i+2},\ldots,\ups_{n{-}1})$. This $\vec\alpha$ can be seen also  in $\mathscr{F}[\vec\ap,1,i]$, which comes from another component of $f$,  providing 
$\ups_i\ge 2$. However, no further $\mathscr{F}[\vec\ap,1,j]$ comes into question. 
 Our assumptions on $F$ imply that 
in case $\ups_i=1$,   $\vec\ap$ has at most two  nonzero components: $\ap_1,\ap_{i+1}$; and in 
case $\ups_i>1$,  at most three  nonzero components: $\ap_1,\ap_{i},\ap_{i+1}$.
%% still 
Anyway we must have 
 $\ups_{i+1}=d-\ups_i$. 

Summing up, we have $d(n-1)+1$ possibilities for $\vec{k}'$. 
%%
%%% ....... 
\end{proof}

With the help of  linear maps $\varphi_s$ with $s\in\mK^\times$, any Lie bracket $[\,\,,\,]_p$ can be contracted to 
$[\,\,,\,]_{t^n}$, see Lemma~\ref{contr0}, and  we have $\varphi_s(\bh)=s^{2}\bh$.
Thereby any $\{f\in \mathrm{Pol}(F) \mid \{f,\bh\}_p=0\}$ can be contracted to $\eP$. Using this line of argument, see also the proof of Lemma~\ref{claim}, we obtain the next statement. 

\begin{cl} \label{cl-F}
If $F\in\gS^d(\gt g)^{\gt g}$ satisfies the assumption of Theorem~\ref{dim-F}, then for any $p$ of degree $n$, we have 
$\dim\{  f\in \mathrm{Pol}(F) \mid \{f,\bh\}_p=0 \}\le  (n-1)d+1$.
\qed 
\end{cl}

\section{Independence principles and their applications } \label{sec-Ind}

Consider the polynomial rings 
\[\eR=\mK[y \bar t^k \mid n{-1}\ge k\ge 0, \ y\in\q]
\ \ \text{ and } \ \ \tilde\eR=\mK[y t^k \mid k\ge 0, \ y\in\q],
\]
 where $\q$ is a vector space.  
%% with . 
 Both rings are left $\mK[t]$-modules 
with $t{\cdot} yt^k=tyt^k=yt^{k+1}$ and $t{\cdot} y\bar t^k=ty\bar t^k=y\bar t^{k+1}$. 
Our examples of $\eR, \tilde\eR$ are $\gS(\mW), \gS(\q[t])$. 
Any  normalised polynomial $p\in\mK[t]$ of degree $n$ defines the 
quotient map $\psi_p\!:\tilde\eR\to\eR$.
Fix  $Y=y_1\ldots y_d\in\gS^d(\q)$ of degree $d$,  choose 
$l\in\mK[t]$ such that $0\le\deg l\le 1$, and set 
\[
\eL=\eL(p,p+l)=\{ap+(1-a)(p+l)\mid a\in\mK\}.\] 
Set further $Y[t]=y_1t\ldots y_dt$ and $Y[\bar r]=y_1\bar r\ldots y_d\bar r$ for $\bar r\in\mK[t]/(p)$. 
 
We remark that all statements of this section are obviously true for $d=0$ and $d=1$ and  there is no need to consider these two instances. 
 
Take any $\tilde p\in\eL$.  %%%has 
We have $\mK[t]/(\tilde p)=\bigoplus_{\alpha}\gt m_\alpha$, where the sum is taken over 
distinct roots of $\tilde p$.
%%%% 
If $\alpha$ is a simple root, then $\gt m_\alpha=\mK\bar r_\ap$ with $r_\ap^2-r_\ap\in(\tilde p)$. 
If a root $\alpha$ has multiplicity $k\ge 2$, then 
\[
\gt m_\alpha=\mK \bar r_{(\alpha,0)}\oplus\mK\bar r_{(\alpha,1)}\oplus\mK\bar r_{(\alpha,1)}^2\oplus\ldots\oplus \mK\bar r_{(\alpha,1)}^{k{-}1},
\]
 where $\bar r_{(\alpha,0)}\bar r_{(\alpha,1)}=\bar r_{(\alpha,1)}$, \ $\bar r_{(\alpha,1)}^k=0$, and
 $\bar r_{(\alpha,0)}^2=\bar r_{(\alpha,0)}$. In that case we define  the linear map $\Phi_\alpha\!:\mK[t]\to \gt m_\ap$ by setting 
\begin{itemize}
\item[ ] $\Phi_\alpha(t^{k})=\bar r_{(\alpha,0)}$, 
\item[ ] $\Phi_\alpha(t^u)=\bar r_{(\alpha,1)}^{k-u}$ for $1\le u<k$, 
\item[ ] $\Phi_\alpha(t^u)=0$ for $u=0$ and for  $u>k$. 
\end{itemize}
The map $\Phi_\alpha$ extends to $\eR$.
Recall that $\tau=t^2\partial_t$.

For each 
$\tilde p\in\eL$, we define a subset $\mathbb S[Y,\tilde p]=\bigcup_\alpha \mathbb S[Y,\tilde p]_\ap$, where 
%%% 
\begin{itemize}
\item[$\diamond$]  $\mathbb S[Y,\tilde p]_\ap=Y[\bar r_\ap]$ if $\ap$ is a simple root, %%% ;
\item[$\diamond$] $\mathbb S[Y,\tilde p]_\ap=\{\Phi_\ap(\tau^u(Y[t]))\mid 0\le u<k\}$ if the multiplicity of $\ap$ is $k\ge 2$. 
\end{itemize} 
In any case, $|\mathbb S[Y,\tilde p]|=n$. 
Set next $V_{Y,\eL}=\left< \mathbb S[Y,\tilde p] \mid \tilde p\in \eL\right>$.
Each element of $\mathbb S[Y,\tilde p]_\ap$ is a linear combination $\sum_{\vec{k}} C_{\vec{k},\tilde p,\ap} Y[\vec{k}]$
and clearly %% 
\begin{itemize}
%\item
\item[{\sf(Ip1)}] 
the coefficients $C_{\vec{k},\tilde p,\ap} \in\mK$  do not depend on $Y$.
%%%% 
\end{itemize} 
Thus also 
$\dim V_{Y,\eL}$ is independent of $Y$.
Extending our new notation from monomials to polynomials by linearity, we may state that 
$\dim V_{F,\eL}$ is independent of $F\in\gS^d(\q)\setminus\{0\}$.   
 According to Proposition~\ref{bound-dim}, 
$\dim V_{Y,\eL}\ge d(n-1)+1$. If %
$F(\xi)=\det(\xi)$ for $\xi\in\q^*$,  in case $\q=\gt{sl}_d$, %%%
then $\dim V_{F,\eL}\le d(n-1)+1$ by Corollary~\ref{cl-F} and Lemma~\ref{lin}.
Thus always 
\begin{equation} \label{dimY1}
\dim V_{Y,\eL}=d(n-1)+1.
\end{equation}
We are interested also  in $\tilde V_{Y,p}=\left< \psi_{p}(\tau^k(Y[t]))\mid k\ge 0\right>$. 
Each element  $ \psi_{p}(\tau^u(Y[t]))$ is a linear combination $\sum_{\vec{k}} \tilde C_{\vec{k},p,u} Y[\vec{k}]$
and 
\begin{itemize}
\item[{\sf (Ip2)}] 
the coefficients $\tilde C_{\vec{k},p,u} \in\mK$ do not depend on $Y$.
\end{itemize} Thus also 
%%  ........... 
$\dim \tilde V_{Y,p}$ is independent of $Y$ and, more generally,
$\dim \tilde V_{F,p}$ is independent of $F\in\gS^d(\q)\setminus\{0\}$.
According to Lemma~\ref{pol-gZ},  
$\dim \tilde V_{Y,p}\ge d(n-1)+1$ if $p(0)\ne 0$. If we take again   $F$ as  
the function $\xi\mapsto\det(\xi)$ with $\xi\in\gt{sl}_d^*$, 
 %%%%% 
 then  $\dim \tilde V_{F,p}\le d(n-1)+1$ by Corollary~\ref{cl-F} and Lemma~\ref{lin}.
Thus 
\begin{equation} \label{dimY2}
\dim \tilde V_{Y,p}=d(n-1)+1 \ \text{ for any } \ p \ \text{ with } \ p(0)\ne 0. 
\end{equation}
Furthermore, if $p(0)\ne 0$,  \,$\q=\gt{sl}_d$, and $F$ is the same as above, then  
$\tilde V_{F,p}$ and $V_{F,\eL}$ coincide with the Poisson centraliser of $\bh$ in 
$\Pol(F)$ w.r.t. $\{\,\,,\,\}_p$; in particular, for such polynomials $p$, we have $\tilde V_{F,p}=V_{F,\eL}$.
Thus, by the independence principles {\sf (Ip1)}, {\sf (Ip2)},   
\begin{equation} \label{Y3}
%%%%The 
\tilde V_{Y,p}=V_{Y,\eL} \ \ \text{ whenever } \ p(0)\ne 0. 
\end{equation}

\begin{thm} \label{sovp-t}
Suppose that $\gt q$ satisfies the assumptions of Theorem~\ref{thm:k-T} and that $p(0)\ne 0$.
Then $\psi_p(\gZ(\wq,t))=\gZ(p,p+t)$.
\end{thm}
\begin{proof}
We have $\gS(\q)^\q=\mK[F_1,\ldots,F_m]$, where $m=\ind\q$. %% 
Set $\eL=\{p+\ap t\ \mid \ap\in\mK\}$.  By Proposition~\ref{sym-lim}, each $\cz_{\tilde p}$ with $\tilde p\in\eL$ has a 
set of (algebraically independent) generators  $F_{i,u}\in \Pol(F_i)$.  
Extending our new notation from monomials to polynomials, we may state that 
$\cz_{\tilde p}$ is generated by $\bigcup_{i=1}^m\mathbb S[F_i,\tilde p]$. This implies that %% 
$\gZ(p,p+t)$ is generated by 
$\bigoplus_{i=1}^m V_{F_i,\eL}$. 

By Theorem~\ref{gen-gZ}, $\gZ(\wq,t)$ is generated by $\tau^k(F_i[t])$ with $k\ge 0$ and  $1\le i\le m$. 
Therefore $\psi_p(\gZ(\wq,t))=\mathsf{alg}\langle \tilde V_{F_i,p} \mid 1\le i\le m\rangle$. Since $p(0)\ne 0$, we have 
$ \tilde V_{F_i,p}=V_{F_i,\eL}$ for each $i$ by~\eqref{Y3}. % 
This finishes the proof. 
\end{proof}  

Suppose that
 $\q$ satisfies the assumptions of Theorem~\ref{thm:k-T}  and has  the {\sl codim}--$2$ property.
Then, in particular, the set $\Omega_{\q^*}$ defined in that theorem is a big open subset of $\gt q^*$. 
Furthermore, 
$\Omega_{\q^*}=\gt q^*_{\sf reg}$ by a generalisation of the  {\it Kostant regularity criterion\/} \cite[Thm~9]{ko63} and 
$\sum_i \deg F_i=\bb(\q)$, see 
\cite[Thm~1.2]{Dima07}, \cite[Lemma~2.1]{contr}, \cite{ap}.

\begin{thm} \label{free}
Suppose that $\gt q$ satisfies the assumptions of Theorem~\ref{thm:k-T}  and has  the {\sl codim}--$2$ property.
Then $\gZ(p,p+l)$ is a polynomial ring with 
$\bb(\q,n)$ generators.
%%% 
\end{thm}
\begin{proof}
We have $\gS(\q)^\q=\mK[F_1,\ldots,F_m]$, where each  $F_i$ is homogeneous and $\deg F_i=d_i$.
According to Theorem~\ref{trdeg}, $\trdeg\gZ(p,p+l)=\bb(\q,n)=(n-1)\bb(\gt q)+m$.

Set $\eL=\{ap+(1-a)(p+l) \mid a\in\mK\}$. Then $\gZ(p,p+l)=\mathsf{alg}\langle %%$ 
%%$
V_{F_i,\eL} \mid 1\le i\le m\rangle$.  According to \eqref{dimY1}, 
$\dim V_{F_i,\eL}=d_i(n-1)+1$. This 
means that $\gZ(p,p+l)$ has at most 
\[
\sum_{i=1}^m (d_i(n-1)+1)=m+(n-1)\sum_{i=1}^m d_i=m+(n-1)\bb(\q)
\] 
generators. Thereby these generators have to be algebraically independent. 
%%
%% .......%
\end{proof}

\begin{lm}\label{3-}
Suppose  that $\gt q$  has  the {\sl codim}--$3$ property. Then the Lie algebra 
$(\mW,\ell)$
with $\ell=[\,\,,\,]_{p}-[\,\,,\,]_{p+l}$ also has the {\sl codim}--$3$ property.
\end{lm}
\begin{proof}
The Lie bracket $\ell$ can be contracted  either to 
$\ell_\infty:=[\,\,,\,]_{t^n-1}-[\,\,,\,]_{t^n}$ or to the bracket  $\ell':=[\,\,,\,]_{t^n-t}-[\,\,,\,]_{t^n}$, depending on $l$, 
see Lemma~\ref{contr-}. In any case, \[\ind\!(\mW,\ell)=\ind\!(\mW,\ell_\infty)=\ind\!(\mW,\ell')=
\dim\q+(n-1)\ind\q
\] by Lemma~\ref{ind-}. 
The Lie algebra $(\mW,\ell')$ is isomorphic to $\q\langle n{-}1\rangle\oplus\gt q^{\rm ab}$. Hence it has  
the {\sl codim}--$3$ property by Lemma~\ref{lm-codim}. If $\xi\in \mW^*$ and $\bar\xi=\xi|_{\q{\cdot}1}\in\q^*_{\sf reg}$, then $\xi\in\mW^*_{\ell_\infty,{\sf reg}}$ by \eqref{1-1}. Thereby $(\mW,\ell_\infty)$ also has  the {\sl codim}--$3$ property.  

Recall that $\ell_{(s)}(x,y)=\varphi_s^{-1}(\ell(\varphi_s(x),\varphi_s(y))$ for $s\in\mK^\times$ and any $x,y\in\mW$.
Let ${\mathcal M}_{(s)}$ be the structure matrix of the Lie  bracket $s^k\ell_{(s)}$, where $k=-n+\deg l$. %% if 
Then \[
\rk{\mathcal M}_{(s)}=(n-1)(\dim\gt q - \ind\gt q)\] for any nonzero $s$. 
Set $\tilde{\mathcal M}=\lim_{s\to 0}{\mathcal M}_{(s)}$.
Then $\tilde{\mathcal M}$ is the structure matrix of $\lim_{s\to 0}s^k\ell_{(s)}$, which is either $\ell_\infty$ or $\ell'$. Anyway 
%%% In 
 $\rk\tilde{\mathcal M} =\rk{\mathcal M}$. 
The  maximal nonzero minors of $\tilde{\mathcal M}$ are limits of  maximal nonzero 
minors of ${\mathcal M}_{(s)}$.  Therefore the Lie algebra $(\mW,\ell_{(s)})$ has the {\sl codim}--$3$ property for at least one $s\in\mK^\times$. Since $\varphi_s$ is an isomorphism of Lie algebras, we obtain 
$\dim\mW^*_{\ell,{\sf sing}}\le \dim\mW-3$, cf.~\cite[(4.1)]{Y-imrn}.
%% .
\end{proof}

\begin{thm} \label{max}
Suppose  that $\gt q$ satisfies the assumptions of Theorem~\ref{thm:k-T} and  has  the {\sl codim}--$3$ property.
Then $\gZ(p,p+l)$ is a maximal (w.r.t. inclusion) Poisson-commutative 
subalgebra of $(\gS(\mW),\{\,\,,\,\}_p)^{\gt q}$. 
%%%%  
\end{thm}
\begin{proof}
According to Theorem~\ref{free}, $\gZ=\gZ(p,p+l)$ is freely generated by homogeneous polynomials 
$H_1,\ldots,H_{\bb(\q,n)}$. Thus $\textsl{d}_\gamma \gZ(p,p+l) = \left<\textsl{d}_\gamma H_i \mid 1\le i\le \bb(\q,n)\right>$ for any $\gamma\in\mW^*$. Set $L=\left<\{\,\,,\,\}_{p},\{\,\,,\,\}_{p+l}\right>$.
Since $\q$ has  the {\sl codim}--$3$ property, $\dim\mW^*_{v,{\sf sing}}\le \dim\mW-3$ for any 
$v\in\mathbb P^1$ different from $(1\,{:}-\!1)$, see Lemma~\ref{lm-codim}.  
Lemma~\ref{3-} extends this inequality to $v=(1\,{:}-\!1)$. Now $\mW^*_{L,{\sf reg}}$ is a big open subset of $\mW^*$ by Lemma~\ref{lm-X}. Furthermore, 
\[U=\mW^*_{L,{\sf reg}}\cap \{\xi\in\mW^* \mid \bar\xi=\xi|_{\q{\cdot}1}\in\q^*_{\sf reg}\}
\]
 is also a big open subset. 

Take any $\xi\in U$. Then 
$\widetilde{\gt v}:=\ker\pi_{1,-1}(\xi)=\gt q\oplus\gt v_0$, where $\pi_{1,0}(\xi)(\q,\gt v_0)=0$, see \eqref{eq-inf-v}. Next  $\rk\!(\pi_{1,0}(\xi)|_{\widetilde{\gt v}})=\dim\q-\ind\q$ and 
$\dim V(\xi)=\bb(\q,n)$  for $V(\xi)=\sum_{a\in\mK}\ker\pi_{a,1-a}(\xi)$, see~\eqref{rkv}, \eqref{dimV}. 

For any Takiff Lie algebra $\qqk$, the set $\Omega_{\qqk^*}$ is a big open subset of $\qqk^*$, which coincides 
with $\qqk^*_{\sf reg}$, see \cite[Thm~2.2{\sf(ii)}]{k-T} and \cite[Sect.~2]{contr}. 
This implies the equality $\textsl{d}_\xi \cz_{\tilde p}=\ker\pi_{\tilde p}(\xi)$ 
for the Poisson tensor $\pi_{\tilde p}$ of any $\tilde p\in L$. Thereby $\textsl{d}_\xi \gZ=V(\xi)$ and 
$\dim \textsl{d}_\xi \gZ=\bb(\q,n)$.  %% 

Thus, the differentials $\textsl{d} H_i$ of the algebraically independent generators $H_i\in\gZ$ 
are linearly independent on a big open subset. Then by \cite[Theorem~1.1]{ppy}, 
$\gZ$ is an algebraically closed subalgebra of $\gS(\mW)$.

Suppose that $\gZ\subset \ca\subset  \gS(\mW)^{\q}$ and $\ca$ is Poisson-commutative. Then 
$\trdeg \ca\le \bb(\q,n)$ by~\cite[Prop.\,1.1]{m-y}, see also \eqref{bound}. Thereby $\gZ\subset \ca$ is an algebraic extension and we must have
$\gZ=\ca$.
\end{proof}  

\subsection{The case of $\gzu$\,} \label{s-gzu}
It is also possible to formulate an independence principle for $\gzu$.
If $F\in\gS^d(\gt q)$, then $\Psi(F[t])= \sum_{\vec k} C_{\vec k} F[\vec k,t]$, where the 
coefficients $C_{\vec k}\in\Z$ do not depend on $F$. Repeating the proof of Proposition~\ref{red-gzu}, we obtain 
the following statement. If  $\q$ satisfies the assumptions of Theorem~\ref{thm:k-T}, then 
\begin{equation} \label{T-1}
\gzu=\gZ(\wq,[0])=\mK[\calv{\circ}\tau^k(F_j[t]) \mid 1\le j\le m, \ k\ge 0].
\end{equation}

\begin{thm} \label{ft-gzu}
%%Suppose that
{\sf (i)} If
$\q$ satisfies the assumptions of Theorem~\ref{thm:k-T}, then %%we have \\[.2ex]
$\psi_p(\gzu)=\gZ(p,p+1)$. \\[.2ex]
{\sf (ii)} If $\gt q=\gt g$ is semisimple, then $\lim_{\esi\to 0}\psi_p(\gt z(\wg,\esi t+1))=\psi_p(\widetilde{\gzu})$. 
\end{thm}
\begin{proof}
%%% and then prove that , whenever 
{\sf (i)} We have 
\[\psi_p(\gzu)=\psi_p(\lim_{\esi\to 0} \gZ(\wq,\esi t+1))\subset \lim_{\esi\to 0} \psi_p(\gZ(\wq,\esi t{+}1)),
\]
where
$\psi_p(\gZ(\wq,\esi t{+}1))=\gZ(p,p{+}\esi t{+}1)$ for almost all $\esi\in\mK^\times$, see Theorem~\ref{sovp-t} and 
Proposition~\ref{123}. By the same proposition, $\lim\limits_{\esi\to 0} \gZ(p,p{+}\esi t{+}1)=\gZ(p,p+1)$. 
Thus $\psi_p(\gzu)\subset \gZ(p,p+1)$. 
Furthermore, we have  
\[\gZ(p,p+1)=\mathsf{alg}\langle %%$ 
%%$
V_{F_j,\eL} \mid 1\le j\le m\rangle\]
 with $\eL=\{ap+(1-a)(p+1) \mid a\in\mK\}$. 
By~\eqref{dimY1}, 
$\dim V_{F_j,\eL}=d_j(n-1)+1$. 

Next $\tilde{\boldsymbol f}\!_{j,k}:=\psi_p{\circ}\calv{\circ}\tau^k(F_j[t])\in \Pol(F_j)\cap  \gZ(p,p+1)$ for all $j$ and $k$.
Since the polynomials $F_1,\ldots,F_m$ are algebraically independent, 
\[
\Pol(F_j)\cap\mathsf{alg}\langle V_{F_i,\eL} \mid i\ne j\rangle =\{0\},
\]
cf.~\eqref{c}. Thus $\tilde{\boldsymbol f}\!_{j,k}\in V_{F_j,\eL}$. The highest $\bar t$-component
of $\tilde{\boldsymbol f}\!_{j,k}$ is $F_j^{[k]}$ if $k\le d(n{-}1)$.  These components are clearly linearly independent. 
Hence $\dim\left<\tilde{\boldsymbol f}\!_{j,k} \mid k\ge 0\right>\ge d_j(n{-}1)+1$ and 
$V_{F_j,\eL}\subset \psi_p(\gzu)$. The inclusion holds for each $j$, thereby 
$\psi_p(\gzu)=\gZ(p,p+1)$. 

%%% Next $\psi_p(\gzu)=\lim_{\esi\to 0}\gZ(p,p{+}\esi t{+}1)\subset \gZ(p,p+1)$.  
%%% {\sf (ii)} For any normalised $p\in\mK[t]$ of degree $n$, we have 

{\sf (ii)} Clearly $\psi_p(\widetilde{\gzu})=\psi_p(\lim_{\esi\to 0}\gt z(\wg,\esi t+1))\subset \lim_{\esi\to 0}\psi_p(\gt z(\wg,\esi t+1))$. 
Thereby it suffices to prove the equality 
%%inclusion 
$\gr\!(\psi_p(\widetilde{\gzu}))=\gr\!\!\left(\lim_{\esi\to 0}\psi_p(\gt z(\wg,\esi t+1))\right)$. 
We have 
\[
\gZ(p,p+1)\overset{\text{part~{\sf (i)}}}{=}\psi_p(\gzu)\overset{\text{Prop.~\ref{prop-gzu}}}{=}\psi_p(\gr\!(\widetilde{\gzu}))
\subset \gr\!(\psi_p(\widetilde{\gzu}))\subset \gr\!\!\left(\lim_{\esi\to 0}\psi_p(\gt z(\wg,\esi t+1))\right)=:\ca.
%
%%\gr\!(\lim_{\esi\to 0}\psi_p(\gt z(\wg,\esi t+1)))\subset \gr\!(\psi_p(\widetilde{\gzu})) \text{\colorbox{red}{$\subset\,$(?)$\,=$}}
%
%%\psi_p(\gzu)\overset{\text{Thm~\ref{ft-gzu}}}{=}\gZ(p,p+1).
\]
By Theorem~\ref{max}, $\gZ(p,p+1)$ is a maximal (w.r.t. inclusion) Poisson-commutative 
subalgebra of $(\gS(\mW),\{\,\,,\,\}_p)^{\gt g}$. Since $\ca$ is also Poisson-commutative  and 
$\ca\subset (\gS(\mW),\{\,\,,\,\}_p)^{\gt g}$, there is the equality $\gZ(p,p+1)=\ca$.
%% \colorbox{red}{........}
\end{proof}

As we know from Section~\ref{GA}, $\psi_p(\gt z(\wg,\esi t+1)$ is the Gaudin algebra 
$\cG(\esi a_1+1,\ldots,\esi a_n+1)$ with $\prod_{i=1}^n (t-a_i)=p$. Hence Theorem~\ref{ft-gzu}{\sf (ii)} states that 
$$\lim_{\esi\to 0}\cG(\esi a_1+1,\ldots,\esi a_n+1)=\psi_p(\widetilde{\gzu}).
$$ 
Furthermore, $\gr\!(\psi_p(\widetilde{\gzu}))=\gZ(p,p+1)$. Since the centre of $\gt g$ is trivial and 
$\psi_p(\widetilde{\gzu})\subset\U(\gt h)^{\gt g}$ with $\gt h=\gt g^{\oplus n}$, the algebra 
$\psi_p(\widetilde{\gzu})$ contains the Gaudin Hamiltonians associated with $\vec a$, cf. Theorem~\ref{Ham-gzu}. 
Reverting notation to the original setting of \cite{FFRe}, we obtain 
$\cG(\vec a)=\eus C(a_1^{-1},\ldots,a_n^{-1})$ and 
\begin{multline*}
\lim_{\esi\to 0} \eus C\left(\frac{1}{1+\esi a_1},\ldots,\frac{1}{1+\esi a_n}\right)=\lim_{\esi\to 0}\eus C(1-a_1\esi+\sum_{j\ge 2}(-a_1\esi)^j,\ldots, 1-a_n\esi+\sum_{j\ge 2}(-a_n\esi)^j)\overset{\text{\cite{G-07}}}{=}\\
\lim_{\esi\to 0}\eus C(-a_1\esi+\sum_{j\ge 2}(-a_1\esi)^j,\ldots, -a_n\esi+\sum_{j\ge 2}(-a_n\esi)^j)=\\
\lim_{\esi\to 0}\eus C(a_1+\sum_{j\ge 2}a_1^j(-\esi)^{j-1},\ldots, a_n+\sum_{j\ge 2}(a_n)^j(-\esi)^{j-1})=
\eus C(\vec a),
\end{multline*}
where for the last equality we have used the maximality of $\eus C(\vec a)$. 
Thus in the semisimple case, we have 
$\psi_p(\widetilde{\gzu})=\psi_p(\gt z(\wg,t^{-1}))$ and 
$\gZ(p,p+1)=\gr\!(\eus C(\vec a))=\psi_p(\gzu)$.
%% ..... probably, the Gaudin algebra $\cG(a_1^{-1},\ldots,a_n^{-1})$......... 

\subsection{A few words in conclusion}

\begin{Rem} 
{\sf (i)}  
 In order to extend our results to all Lie algebras, not necessary satisfying the assumptions of Theorem~\ref{thm:k-T},
one needs a better understanding of the interplay between   $\psi_p(\gZ(\wq,t))$ and $\gZ(p,p+t)$.  \\[.3ex] 
{\sf (ii)} An interesting question is, whether $\gZ(\q,t)$ has a quantisation.  In order to extend the method of Feigin and Frenkel~\cite{ff},
%% Probably 
one has to assume that $\q$ is quadratic. %% in order to get a positive answer. 
We note that 
for the centralisers $\gt g_\gamma$ with $\gamma\in\gt g$, the problem is settled, affirmatively, in 
%%% 
\cite{ap}, although $\gt g_\gamma$ is not always a quadratic Lie algebra. %%%%% 
\end{Rem}

\end{document}